\documentclass{cpamart1}     
\received{Month 200X}       
\volume{000}
\startingpage{1}                      

\authorheadline{I.Corwin, H.Shen}
\titleheadline{OpenASEP in Weak-Asymmetric Regime}

\usepackage{amssymb}
\usepackage{hyperref}
\usepackage{breakurl}
\usepackage[lefttag]{mhequ}
\usepackage{mhsymb}
\usepackage{booktabs}
\usepackage{mathrsfs}
\usepackage{times}
\usepackage{microtype}
\usepackage{comment}
\usepackage{wasysym}
\usepackage{centernot}
\usepackage[inline]{enumitem} 

\usepackage{mathptmx}

\newtheorem{theorem}{Theorem}[section]
\newtheorem{lemma}[theorem]{Lemma}
\newtheorem{corollary}[theorem]{Corollary}
\newtheorem{proposition}[theorem]{Proposition}
\newtheorem{assumption}[theorem]{Assumption}
\theoremstyle{definition}
\newtheorem{definition}[theorem]{Definition}
\theoremstyle{remark}
\newtheorem{remark}[theorem]{Remark}

\let\om\omega

\def\a{\mathbf{a}}
\def\b{\mathbf{b}}

\newcommand{\bfp}{ \mathbf{p}}

\newcommand{\e}{\varepsilon}

\def\${|\!|\!|}

\def\E{\mathbf{E}}

\usepackage{graphicx}
\newcommand*{\Cdot}{{\raisebox{-0.5ex}{\scalebox{1.8}{$\cdot$}}}} 

\usepackage{acronym}
\acrodef{SHE}[SHE]{Stochastic Heat Equation}
\acrodef{KPZ}[KPZ]{Kardar--Parisi--Zhang}
\acrodef{ASEP}[ASEP]{Asymmetric Simple Exclusion Process}
\acrodef{BDG}[BDG]{Burkholder--Davis--Gundy}

\newcommand{\be}{\begin{equation}}
\newcommand{\ee}{\end{equation}}

\def\eps{\varepsilon}
\def\Z{\mathbb{Z}} \def\N{\mathbb{N}}   \def\R{\mathbb{R}}

\usepackage{tikz}
\usetikzlibrary{arrows}

\makeatother

\begin{document}                        


\title{Open ASEP in the Weakly Asymmetric Regime}

\author{Ivan Corwin}{Columbia University}
\author{Hao Shen}{Columbia University}





\begin{abstract}
We consider ASEP on a bounded interval and on a half line with sources and sinks. On the full line, Bertini and Giacomin \cite{BG} proved convergence under weakly asymmetric scaling of the height function to the solution of the KPZ equation. We prove here that under similar weakly asymmetric scaling of the sources and sinks as well, the bounded interval ASEP height function converges to the KPZ equation on the unit interval with Neumann boundary conditions on both sides (different parameter for each side); and likewise for the half line ASEP to KPZ on a half line. This result can be interpreted as showing that the KPZ equation arises at the triple critical point (maximal current / high density / low density) of the open ASEP.
\end{abstract}

\maketitle   



 \tableofcontents



\section{Introduction}

The open asymmetric simple exclusion process (ASEP) is a default paradigm in statistical physics for studying transport in systems in contact with reservoirs which keep the two ends of the system at different local densities. When particles are driven in one direction, the system approaches a non-equilibrium steady state. In other words, while there is an invariant measure, there is also a net flux of particles flowing through the system, and there is no time-reversal symmetry. Consequently, standard notions of equilibrium statistical mechanics do not apply. Open ASEP enjoys certain exact relations (such as the matrix product ansatz, or solvability through Bethe ansatz) which has made it a profitable model through which to develop new physical predictions and theories for a broader class of models which share the same type of steady state behaviors.

In this paper we demonstrate how the stochastic heat equation (SHE) / Kardar-Parisi-Zhang (KPZ) equation / stochastic Burgers equation (SBE) on an interval with various types of boundary conditions arise from the open ASEP under a certain precise weakly asymmetric scaling limit of the model and its parameters (we also study the half line versions of these equations with boundary conditions at the origin).

While, to our knowledge, this connection between the KPZ equation and the open ASEP is new to both the mathematics and physics literature, it is certainly not without precedent. For ASEP on the whole line, the seminal work of Bertini and Giacomin \cite{BG} demonstrated convergence of ASEP under a similar scaling to the KPZ equation on $\R$ via a discrete Cole-Hopf transformation introduced by G{\"a}rtner   \cite{MR931030}. They assumed near equilibrium initial data, and narrow wedge initial data was treated in \cite{MR2796514}. There has since been great progress in expanding this weak universality of the KPZ equation on $\R$ using methods similar to those of Bertini and Giacomin \cite{DemboTsai, CST2016asep, labbe2016weakly, 2015arXiv1505}, as well as using energy solutions (introduced by Assing \cite{MR1888875} as well as Jara and Gon{\c{c}}alves \cite{MR3176353} and proved to be unique by Gubinelli and Perkowski \cite{gubinelli15}) \cite{MR3176353, goncalves15, MR3327509,gonccalves2016stochastic, Diehl2016Brownian, MR3445609,gubinelli2016hairer}, and Hairer's regularity structures (introduced in \cite{Regularity}) \cite{KPZJeremy, CLTKPZ}. These works have dealt entirely with the KPZ equation on $\R$ or the torus ($[0,1]$ with periodic boundary conditions).

We deal here with the KPZ equation with Neumann boundary condition on $[0,1]$ (with generically different derivatives on the two ends). For the SHE this corresponds with Robin boundary conditions and for the SBE with Dirichlet boundary conditions. None of these boundary conditions make a priori  sense (solutions of SHE and KPZ are not differentiable, and SBE does not take values). This is resolved through defining a mild solution and equivalent martingale problem for the SHE with boundary conditions.

\begin{figure}[]
\centering\includegraphics[scale=.67]{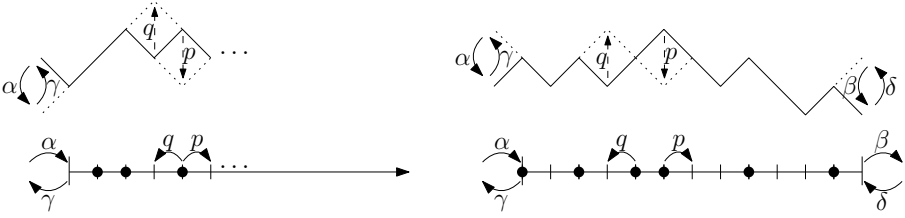}
\caption{The half line and open ASEP (left and right respectively) drawn in terms of particles and height functions (bottom and top respectively). In open ASEP, particles inside jump left and right (i.e. according to clocks of exponential distribution with this rate) at rates $q$ and $p$; particles are inserted (provided the destination is empty) 
from the left (resp. right) at rate $\alpha$ (resp. $\delta$); particles are removed (provided the removal site is occupied) from the left (resp. right) at rate $\gamma$ (resp. $\beta$). In half line ASEP, there is only a boundary on the left and the right goes on infinitely.}\label{Figure1}
\end{figure}

In our work, we tried to follow the scheme of Bertini and Giacomin, though were quickly forced to overcome some significant new complexities. First we apply the G\"{a}rtner (or microscopic Hopf-Cole) transform \cite{MR931030,dittrich91} which yields a microscopic SHE for the exponentiated ASEP height function. Parameterizing the open ASEP boundary conditions by two effective density parameters (one for each side of the interval) leads to Robin boundary conditions on the microscopic SHE. The challenge then becomes to prove convergence to the corresponding continuum SHE. This is done by first showing tightness and then identifying (through a martingale problem formulation) the limit. Tightness reduces to fine estimates about the heat kernel (and its derivatives and various weighted integrals) for the microscopic SHE which, in our case, become much more involved and non-standard than in the periodic or full line case (for instance they require us to invoke Sturm-Liouville theory and delicate method of images estimates). In identifying the limiting martingale problem, Bertini and Giacomin discovered a key identity (Proposition 4.8 and Lemma A.1 in \cite{BG}) for the quadratic martingale which identifies the white noise. In our case of bounded intervals, that identity does not hold. Instead we find (via a new method using Green's functions) an approximate version which suffices -- see Lemma \ref{lem:key-of-key} and Proposition \ref{prop:key-identity}.


Presently the other above mentioned KPZ equation convergence methods have not been developed into the context we consider here. These other methods are less reliant upon the exact structure of the underlying particle system, so it would be nice to see them developed so as to prove universality of the type of convergence we have shown here for the particular model of open ASEP. The energy solutions method  only applies in stationarity (of for initial data with finite entropy with respect to the stationary measure) and relies upon certain inputs from hydrodynamic theory which may be more complicated in this setting since the open ASEP generally lacks product invariant measures. However, in a special one-parameter case (see Remark \ref{rem:rho-rates}) the invariant measure reduce to product Bernoulli. Upon sending this present paper out for comments, we learned that \cite{Goncavlesetal} are completing a work in which they develop the energy solutions approach to study the KPZ limit of open ASEP in the special subcase when the invariant measure is product Bernoulli, and the effective density is exactly $1/2$. This is a special point in the two-dimensional family of parameters we consider here. For generic choices of our parameters, the invariant measure is not of product form.

\subsection{Existing open ASEP results}

Before defining the open / half line ASEP and stating our main convergence theorem, let us try to put our work into the context of known results. Much of this discussion involves results from the physics literature which have not received a mathematically rigorous treatment.

The open ASEP is a test-tube for developing predictions for the behavior of systems with non-equilibrium steady states. The model is illustrated in Figure \ref{Figure1} and defined mathematically in Definition \ref{def:halfASEP}. The matrix product ansatz \cite{MR1193854,schutz1993phase,derrida1993exact,uchiyama2004asymmetric, LazarescuMallickTASEP2011, GLMVASEP} is the main method which has been utilized in studying the open ASEP steady state (see also Derrida's ICM proceedings \cite{derrida2006matrix} for further references). Open ASEP has a highly non-trivial steady state, though it is still possible to compute various asymptotics about stationary expectations, such as that of the current. Let us give the phase diagram (see Figure~\ref{Figure2} for an illustration) which was first discovered and proved by Liggett \cite{MR0445644} in the special parameterization we take below in Definition~\ref{def:parameterization}.
\begin{figure}[]
\centering\includegraphics[scale=.8]{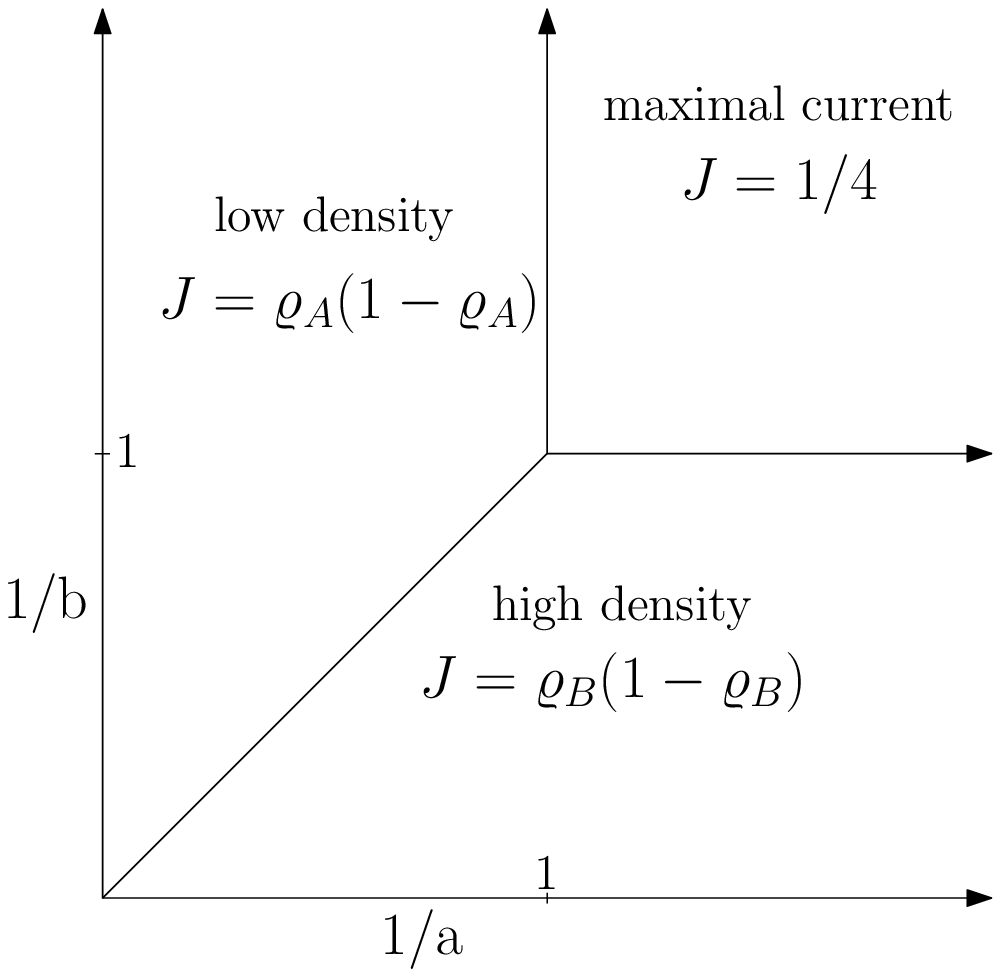}
\caption{The phase diagram for open ASEP current.}\label{Figure2}
\end{figure}

Consider open ASEP on an interval of length $N$ and with boundary parameters $\alpha, \beta, \gamma, \delta$ (see Figure \ref{Figure1}).
Define the average current $J_N$ to be $(p-q)^{-1}$ times the expected value in the steady state of the net number of particles to enter the system from the left in a unit of time.
As $N\to \infty$, this approaches a current $J$ which can be calculated through the matrix product ansatz. It depends on two density\footnote{In the main text of the paper we will work with centered occupation variables taking values of $\pm1$ instead of uncentered occupancy variables which takes values $1$ (particle) and $0$ (hole). The notion of density is with respect to uncentered occupancy variables (though it is easily affinely transformed into centered ones) and can be thought of as the expected value of the uncentered occupancy variable.} 
parameters $\a$ and $\b$ given by the formulas (see for instance \cite{GLMVASEP})
\begin{equs}[eqab]
\a&= \frac{p-q - \alpha +\gamma+\sqrt{(p-q-\alpha+\gamma)^2 + 4\alpha \gamma}}{2\alpha}\\
\nonumber \b&= \frac{p-q - \beta +\delta+\sqrt{(p-q-\beta+\delta)^2 + 4\beta \delta}}{2\beta}.
\end{equs}
We are assuming here (and throughout the paper) that $q<p$. Notice that $\a$ only includes the left source/sink parameters ($\alpha$ and $\gamma$) and $\b$ those on the right ($\beta$ and $\delta$). The effective left and right particle densities are given (respectively) by
 \begin{equ}\label{eqvarrho}
\varrho_A=1/(1+\a)\;,\qquad \textrm{and}\qquad \varrho_B= b/(1+\b)\;.
\end{equ}


The low density phase arises when $\frac{1}{\a} < \frac{1}{\b} \wedge 1$ (we use $1/\a$ and $1/\b$ to match the traditional appearance of the phase diagram) and corresponds to a phase with a current $J =\varrho_A(1-\varrho_A)$. In this phase, there is a low ($<1/2$) effective density of particles near the left boundary, and the bulk dynamics and the right boundary easily transport this density through the system. The high density phase arises when $\frac{1}{\b} < \frac{1}{\a} \wedge 1$ and corresponds to a phase  with a current $J =\varrho_B(1-\varrho_B)$. In this phase, there is a high ($>1/2$) effective density of particles near the right boundary and this causes an overall jamming of the system so that the bulk dynamics and left boundary cannot operate at their maximum efficiency. The maximal current phase, when $1/\a,1/\b>1$ corresponds with flux $J=1/4$ and arises when the system operates at maximum efficiency. The effective densities at the left boundary is more than $1/2$ and at the right boundary is less than $1/2$. Since the density in the bulk which propagates particles at the maximal speed is $1/2$, these boundaries do not inhibit transport.

The above phase diagram describes the law of large number behavior for current. On the other hand, less is understood about the finer scale nature of fluctuations in currents (and in associated height functions) for open ASEP.

The main result of our paper proves that the KPZ equation arises at the triple-point between the low density, high density and maximal current phases. In particular, for system size $N$, we scale $q=1/2-1/(2\sqrt{N})$ and $p=1/2+1/(2\sqrt{N})$ (approximately) and tune $\alpha,\beta,\gamma,\delta$ in such a way that the left and right densities also equal $1/2 + O(1/\sqrt{N})$ (equivalently $\a$ and $\b$ like $1+O(1/\sqrt{N})$). Then, in time of order $N^2$ and scaling fluctuations by order $1/\sqrt{N}$ we come upon the KPZ equation on $[0,1]$ with Neumann boundary conditions. 
This result is stated as Theorem \ref{thm:main-B} (see also Theorem \ref{thm:main-H} for the half line version of this result).

Returning to earlier result, the hydrodynamic limit (i.e., shape theorem or law of large numbers) has been proved for open ASEP under the symmetric and weakly asymmetric scaling limits. In both cases, the boundary parameters $\alpha,\beta,\gamma,\delta$ are fixed. In the symmetric case the internal left and right jump rate parameters $q$ and $p$ are chosen to be equal. Weakly asymmetric scaling in the hydrodynamic limit literature means something different than above -- one takes $q=1/2-1/N$ and $p=1/2+1/N$.
Under this type of weakly asymmetric scaling, the arguments in \cite{MR995290,MR931030} can prove convergence of the open ASEP density field to the viscous Burgers equation on $[0,1]$ with fixed boundary density.
(In fact the proof presented in \cite{MR995290,MR931030} is on the entire line, while another paper \cite{MR978701} is on the torus;  however the arguments can be adapted to the boundary driven case, see \cite{MR1124262,MR1069212,MR1340037,gonccalves2015nonequilibrium} for hydrodynamic limit results of different boundary driven models.)
To our knowledge, no results have been proved regarding the totally asymmetric (or partially asymmetric) open ASEP in which one might expect (as in the full line case) the inviscid Burgers equation.


The fluctuations of the particle density field around its limiting behavior is understood in the weakly asymmetric case. When started out of equilibrium, \cite{gonccalves2015nonequilibrium} and \cite{MR2437527} prove convergence in the (respectively) weakly asymmetric and symmetric cases to generalized Ornstein-Uhlenbeck processes (see also \cite{MR2130640} for an earlier physics prediction of these results). In the symmetric case, and with extremely depressed boundary rates, \cite{franco2016non} has also considered the limiting non-equilibrium fluctuations.

In the stationary state, the (single time) spatial fluctuations for the weakly asymmetric system are Gaussian \cite{MR2130640} while those of the totally asymmetric case are known and non-Gaussian \cite{MR2070099}. In both cases, the stationary state fluctuations are given in terms of the sum of two processes; however there seems to be no way to scale the weakly asymmetric limit process to the totally asymmetric limit process. It is compelling to speculate that, like in the full line case \cite{MR2796514}, the KPZ equation limit we consider here provides the mechanism for crossing over between these two types of steady-state behaviors. There is presently no description for the steady state of the KPZ equation on $[0,1]$ with Neumann boundary condition and it would be interesting to understand whether the matrix product ansatz (or the combinatorial methods used in \cite{MR2630104, MR2831874}) has a meaningful limit under the scalings considered in this paper.


Returning to the current, the variance of the actual number of particles to enter the system in time $t$ (centered by the expected value $t J_N$) grows like $\Delta_N t$ where $\Delta_N$ is called the diffusion coefficient. \cite{MR1330365} argues that in the high and low density phases, the diffusion coefficient has a positive limit as $N\to \infty$, whereas in the maximal current regime, it decays like $1/\sqrt{N}$ -- a fact that they say is consistent with the scaling exponents for the KPZ universality class. There does not seem to be any further results developing this perspective. So, in many ways the fluctuation profile for the various portions of the phase diagram (assuming they are dictated solely by these three phases) remains quite mysterious.


There is a special case of open ASEP in which all boundary parameters are set to zero and the system becomes closed with reflection at the boundaries. In this case, \cite{labbe2016weakly} recently studied various hydrodynamic and fluctuation limit theorems. That version of ASEP admits a reversible invariant measure due to the particle conservation. When started very far from equilibrium, the fluctuations of the system's height function under Bertini-Giacomin style weak asymmetry scaling converges to the KPZ equation on the {\it full line}. That convergence holds for a finite amount of time and then suddenly breaks once the invariant measure is reached. This should be contrasted to our present results in which the KPZ equation is on $[0,1]$ and remains valid for all time. At a technical level, \cite{labbe2016weakly} follows the Bertini-Giacomin approach. The reflecting boundary conditions for ASEP correspond with Dirichlet boundary condition for the associated microscopic heat equation, for which heat kernel estimates are easily accessible through comparison with the analogous full line heat kernel.

It is also worth mentioning that besides the law of large numbers and fluctuations behaviors, the large deviation principle is well-studied for open ASEP. A combination of matrix product ansatz and Bethe ansatz methods have led to a detailed description of the rate function and minimizers for open ASEP in various parameter regimes -- see for example \cite{MR1997915,MR2035624, MR2227085,GLMVASEP,MR2573561, deGierEssler} and references therein.


\subsection{Existing half line ASEP results}

A special case of the half line ASEP was introduced by Liggett in \cite{MR0410986} where the creation rate $\alpha=p\lambda$ and the annihilation rate $\gamma=q(1-\lambda)$. Here, $\lambda$ controls the overall density at the origin which equals $(1-\lambda)/\lambda$. Liggett showed that the long term behavior of this process has a phase transition: if the rate $\lambda<1/2$ and the original density is below a critical value, the stationary measure is a product measure with density $\lambda$, otherwise if $\lambda> 1/2$ the stationary measure is spatially correlated and behaves like product measure with density $1/2$ at infinity. Grosskinsky \cite{Grosskinsky2004} studied, using the matrix product ansatz, the correlations of these stationary measures, and \cite{Duhart2014semi} also employed the matrix product ansatz to study the large deviations for this system.

Using methods from Bethe ansatz and symmetric functions, there has also been some work on studying asymptotics of the fluctuations of the number of particles inserted into the system.
In the totally asymmetric case ($p=1$, $\alpha>0$, $q=\gamma=0$) \cite{Facilitated2016} proved a KPZ universality class limit theorem with either square-root / cube-root fluctuations and Gaussian /  GUE Tracy-Widom type statistics (depending on the exact strength of $\alpha$). This result was predated by work of Baik and Rains \cite{MR1845180} which demonstrated a similar phenomena for the polynuclear growth model.
In the partially asymmetric reflecting case ($\alpha=\gamma=0$) Tracy-Widom derived explicit formulas for the configuration probabilities in \cite{tracy2013asymmetric}, though no asymptotics have been accessible from these formulas as of yet (one generally expects from the universality belief that the same sort of dichotomy between square-root and cube-root fluctuations exists for the general partially asymmetric case. For the KPZ equation itself, \cite{gueudre2012directed,borodin2016directed} have employed the non-rigorous replica Bethe ansatz methods to derive a similar set of KPZ universality class limit theorems as was shown in the TASEP. There are also some formulas derived for the half space log-gamma polymer in \cite{MR3232009}, though they have yet to produce any asymptotic results.

After completing the present paper, \cite{BBCW17} proved an exact one-point distribution formula for half-line ASEP under the parameterization considered herein, when $A=-1/2$. While that value of $A$ is outside the class we consider (where $A, B \geq 0$), \cite{Par17} has been able to extend our analysis to all real $A, B$. Combining those two works proves the first one-point distribution formula and cube-root fluctuations for the half-space KPZ equation in the special case when $A=-1/2$.

\subsection*{Outline}
{\small
The introduction above contained a brief survey of the open and half line ASEP, as well as a number of relevant works in the literature. Section \ref{Sec2} contains the definitions of the models, the choices of parameters and scalings, assumptions on initial data, and finally the main results of this paper (Theorems~\ref{thm:main-H} and \ref{thm:main-B}) which prove convergence to the KPZ equation with Neumann boundary conditions. Section \ref{Sec3} contains the G\"artner transform of open and half line ASEP, and the resulting microscopic SHE on half line and bounded interval with Robin boundary condition. This simple boundary condition occurs on a two-dimensional subspace of the four boundary parameters $\alpha, \beta, \gamma, \delta$ when they are parameterized by two density parameters. Section \ref{Sec4} spend a significant amount of space developing a variety of bounds for the discrete heat kernel (and its discrete derivative) associated with Laplacian with Robin boundary conditions. Some of these bounds require us to derive a generalized method of images, as well as to apply methods from Sturm-Liouville theory. This section them uses the heat kernel bounds to prove tightness of the rescaled microscopic SHE. Section \ref{Sec5} contains the identification of all limit points with solutions to the martingale problems for the continuum SHE with Robin boundary conditions (and then shows the uniqueness of the martingale problem and identifies it with the mild solution, thus proving our main theorems). The section also contains a crucial cancellation in the form of Proposition \ref{prop:key-identity} and Lemma \ref{lem:key-of-key}, without which the identification of the limiting quadratic martingale would not be evident.
}

%

\section{Definitions of the models and the main results}\label{Sec2}


The aim of this paper is to study the {\it weakly} asymmetric  simple exclusion processes (ASEP) with open boundaries
and their  KPZ equation limits. We will discuss two cases: the first case is an ASEP model on the {\it half line} (i.e. positive integers) with an open boundary at origin, which will be called ASEP-H; the second case is an ASEP model on a {\it bounded interval}
with open boundaries at the two ends, which will be called ASEP-B.

We start with the definition of the asymmetric simple exclusion process (ASEP) on the positive integers. See Figure \ref{Figure1} for an illustration of these definitions.

\begin{definition}[Half line ASEP or ``ASEP-H"] \label{def:halfASEP}
Fix four non-negative real parameters  
$p,q,\alpha, \gamma \ge 0$.
Let $\eta(x) \in \{\pm 1\}$ denote
the centered occupation variable at site $ x\in\Z_{>0}=\{1,2,3,\cdots\} $.
The site $ x\in\Z_{>0}$ is said to be occupied by a particle if $\eta(x) =1$,
and empty if $\eta(x) =-1$.
The half line ASEP is a continuous-time Markov process on the state space
$ \big\{ (\eta(x))_{x\in\Z_{>0}}\in \{\pm 1\}^{\Z_{>0}}   \big\}$.
The state $\eta_t$ at time $t$ evolves according to the following dynamics (we could state the generator, though we do not make explicit use of it and hence avoid writing it):
at any given time $ t\in[0,\infty) $ and $x\in\Z_{>0}$
a particle jumps from the site $ x $ to the site $ x+1 $
at exponential rate
\begin{equ} [e:defcR]
c^R(\eta_t,x) = \frac{p}{4} \big(1+\eta_t(x)\big) \,\big(1-\eta_t(x+1)\big)
\end{equ}
and from site $ x+1 $ to site $ x $ at exponential rate
\begin{equ} [e:defcL]
c^L(\eta_t,x) = \frac{q}{4} \big(1-\eta_t(x)\big)\, \big(1+\eta_t(x+1)\big) \;;
\end{equ}
 and if $x=1$, a particle is created or annihilated at the site $1$ at exponential rates
\[
r_A^+\big(\eta_t(1)\big)=
\frac{\alpha}{2} \big(1-\eta_t(1)\big)
\quad\mbox{and}\quad
r_A^-\big(\eta_t(1)\big)=
\frac{\gamma}{2} \big(1+\eta_t(1)\big) \;,
\]
respectively.
All these events of jumps and creations / annihilations are
independently of each other.
\end{definition}

Such a process can be constructed by the standard procedures as in \cite{liggett12}.
Note that the rate $r_A^\pm\big(\eta(1)\big)$ is such that particles  are not allowed to be created at $1$  if $x=1$ is already occupied, and not allowed to be
 annihilated from $1$ if there is no particle at $x=1$. Also, the number of particles is not preserved. One may also imagine a pile of infinitely many particles at site $0$,
and particles are allowed to jump between $0$ and $1$ at rates $r_A^\pm\big(\eta(1)\big)$.


We now define ASEP on the lattice with two-sided open boundaries.

\begin{definition}[ASEP on bounded interval or ``ASEP-B"] \label{def:intervalASEP}
Fix six non-negative real parameters $p,q, \alpha,\beta,\gamma,\delta \ge 0$.
Let $\eta(x) \in \{\pm 1\}$ denote
the centered particle occupation variable
at site $ x\in\Lambda_N \eqdef \{1,2,\cdots,N\} $;
$x$ is occupied by a particle if $\eta(x) =1$
and empty if $\eta(x) =-1$.
The ASEP on the bounded interval $\Lambda_N$
 is a continuous-time Markov process on the state space
$ \big\{ (\eta(x))_{x\in\Lambda_N}\in \{\pm 1\}^{\Lambda_N} \big\}$. The state $\eta_t$ at time $t$ evolves according to the following dynamics:
at any given time $ t\in[0,\infty) $ and $x\in\Lambda_N$,
a particle jumps from site $ x $ to site $ x+1 $ (resp. from site $ x+1 $ to site $ x $)
at exponential rate $c^R(\eta_t,x)$ (resp. $c^L(\eta_t,x)$) where $c^R(\eta_t,x)$ and $c^L(\eta_t,x)$ are defined in \eqref{e:defcR} and \eqref{e:defcL}.

In addition to the jumps,  a particle is created (resp. annihilated) at the boundary site $x=1$ at exponential rates $r^+_A$ (resp. $r^-_A$),
and a particle is created (resp. annihilated) at the boundary site $x=N$ at exponential rates $r^+_B$ (resp. $r^-_B$), where
\begin{equs}
r^+_A\big(\eta(1)\big)&=
\frac{\alpha}{2} \big(1-\eta(1)\big) \;,
\qquad
 r^+_B\big(\eta(N)\big)=
\frac{\delta}{2} \big(1-\eta(N)\big) \\
r^-_A\big(\eta(1)\big)&=
\frac{\gamma}{2} \big(1+\eta(1)\big) \;,
\qquad
 r^-_B\big(\eta(N)\big)=
\frac{\beta}{2} \big(1+\eta(N)\big) \;.
\end{equs}
All these events of jumps and creations / annihilations are
independent of each other.
\end{definition}


\begin{definition}[Height functions]
We define the height function $h_t(x)$ for $x\in\Z_{\ge 0}=\{0,1,2,\cdots\}$
associated with the ASEP-H,
and the height function $h_t(x)$ for $x\in\Lambda_N \cup\{0\}$
associated with the ASEP-B, as follows.
\begin{itemize}
\item
Let $h_t(0)$ be $2$ times the {\it net} number of particles that are {\it removed} (i.e. the number of particles annihilated minus the number of particles created) from the site $x=1$
during the time interval $[0,t]$. (In particular $h_0(0)=0$.)
\item
For $x>0$, let
\[
h_t(x) \eqdef h_t(0)+\sum_{y=1}^x \eta_t(y) \;.
\]
\end{itemize}
\end{definition}

According to this definition we have $\nabla^+ h_t(x)=\eta_t(x+1)$ for every $t\ge 0$ and $x\ge 0$ (resp. $0\le x <N$) for ASEP-H (resp. ASEP-B). Here the forward and backward discrete gradients are defined as
\begin{equ} [e:FB-der]
\nabla^+ f(x) \eqdef f(x+1)-f(x) \;, \qquad
\nabla^- f(x) \eqdef f(x-1)-f(x) \;.
\end{equ}
Let us also record our convention for the discrete Laplacian
\begin{equ}[e:LaplacianDef]
\Delta f(x)\eqdef f(x-1) -2f(x) + f(x+1)\;.
\end{equ}

Note that the above definition of $h_t(x)$ is such that for both models, creating or annihilating particles at $x=1$ only affects the value of $h_t(0)$,  and none of the values of $h_t(x)$ for $x>0$.
Also, for ASEP-B, creating or annihilating particles at $x=N$ only affects the value of $h_t(N)$, but none of the values of $h_t(x)$ for $x<N$.
Finally, the internal particle jumps can only
affect the values of $h_t(x)$ for $x$ in the ``bulk", i.e.
$x>0$ for ASEP-H or $0<x<N$ for ASEP-B.
We can in fact equivalently consider the following interface growth model. See Figure \ref{Figure1} for an illustration of this connection.

\begin{definition}[Solid on solid (SOS) models with moving boundaries]
The SOS model on $\Z_{\ge 0}$ with a moving boundary
is a jump Markov process defined on the state space
\[
\{ h \in \Z^{\Z_{\ge 0}} \,:\, |\nabla^+ h (x)| =1\;,\;\;\forall x\in \Z_{\ge 0}\} \;.
\]
The dynamics for the height function $h_t$ at time $t$ are as follows.
For each $x\in \Z_{>0}$, if $\Delta h_t(x)=2$ then $h_t(x)$ increases by $2$ at exponential rate $q$, and if $\Delta h_t(x)=-2$ then $h_t(x)$ decreases by $2$ at exponential rate $p$.
Moreover, if $\nabla^+ h_t(0) =1$  then $h_t(0)$ increases by $2$ at exponential rate $\gamma$
and if $\nabla^+ h_t(0) =-1$  then $h_t(0)$ decreases by $2$ at exponential rate $\alpha$.
All the jumps are independent.

The SOS model on $\Lambda_N \cup \{0\}$ with moving boundaries
is a jump Markov process defined on the state space
\[
\{ h \in \Z^{\Lambda_N \cup \{0\}} \,:\, |\nabla^+ h (x)| =1\;,\;\;\forall x\in \{0,1,\cdots,N-1\} \} \;.
\]
For $x \in \{0,\cdots,N\}$, the $h_t(x)$ increases or decreases according to the same rule as for SOS on $\Z_{\ge 0}$ above, together with the rule that if $\nabla^- h_t(N) = 1$  then $h_t(N)$ increases by $2$ at exponential rate $\beta$ and if $\nabla^- h_t(N) = -1$  then $h_t(N)$ decreases by $2$ at exponential rate $\delta$.
\end{definition}

In this article, we will show via the Cole-Hopf / G\"artner transformation that under the weak asymmetry scaling,
the height function  converges  to the Cole-Hopf solution to the KPZ equation: 
\begin{equ}\label{e:KPZ}
	\partial_T H
=
	\tfrac12 \Delta H + \tfrac12 ( \partial_X H )^2 + \dot W,
\end{equ}
with Neumann (generically inhomogeneous) boundary conditions,
where $\dot W$ is the space-time white noise (formally,
$\E(\dot W_T(X) \dot W_S(X'))=\delta(T-S) \delta(X-X')$).
Here, the Cole-Hopf solution to \eqref{e:KPZ} with generically inhomogeneous Neumann  boundary condition
is defined by
$ H_T(X) = \log \mathscr Z_T(X)$
where $\mathscr Z  \in C([0,\infty),C(\R_+))$ for ASEP-H
or $\mathscr Z  \in C([0,\infty),C([0,1]))$ for ASEP-B
is the mild solution (defined immediately below) to the stochastic heat equation (SHE)
\begin{equ} \label{e:SHE}
\partial_T \mathscr Z = \tfrac12 \Delta \mathscr Z + \mathscr Z \dot W\;,
\end{equ}
with 
Robin boundary conditions.

\begin{definition} \label{def:mild}
We say that
$ \mathscr Z_T(X) $ is {\it a mild solution} to SHE \eqref{e:SHE} on $\R_+$
 starting from initial data $ \mathscr Z_0(\Cdot) \in C(\R_+) $
satisfying a Robin boundary condition with parameter $A\in \R$
\begin{equ} [e:SHERobin0]
\partial_X \mathscr Z_T (X)\Big\vert_{X=0} = A  \mathscr Z_T(0)  \qquad (\forall T > 0)
\end{equ}
 if $\mathscr Z_T$ is adapted to the filtration $\sigma \{\mathscr Z_0, W|_{[0,T]}\}$ and
\begin{equs} \label{e:SHE-mild}
 \mathscr Z_T (X) = \int_0^\infty \!\!\! \mathscr P^R_T(X,Y) \mathscr Z_0 (Y) \,dY
 	 + \int_0^T \!\!\! \int_0^\infty \!\!\!  \mathscr P^R_{T-S} (X,Y) \,\mathscr Z_S (Y) \, dW_S(dY)
\end{equs}
where the last integral is the  It\^o integral with respect to the cylindrical Wiener process $W$,
and $ \mathscr P^R$ is the  heat kernel  satisfying the Robin boundary condition
\begin{equ} 
\partial_X \mathscr P^R_T (X,Y)\Big\vert_{X=0} =A  \mathscr P^R_T (0,Y) \qquad (\forall T > 0,Y>0) \;.
\end{equ}

We say that
$ \mathscr Z_T(X) $ is
a {\it mild solution} to SHE \eqref{e:SHE} on $[0,1] $ starting from initial data   $ \mathscr Z_0(\Cdot) \in C([0,1])$
with Robin boundary conditions with parameters $(A,B)$ 
\begin{equ} [e:SHERobin1]
\partial_X \mathscr Z_T (X)\Big\vert_{X=0} = A  \mathscr Z_T(0) \;,
\quad
\partial_X \mathscr Z_T (X)\Big\vert_{X=1} = -B  \mathscr Z_T(1)  \qquad (\forall T > 0) \;
\end{equ}
if $ \mathscr Z_T(X)$ is adapted to the filtration $\sigma \{\mathscr Z_0, W|_{[0,T]}\}$ and
satisfies \eqref{e:SHE-mild}
 with the integration domain $[0,\infty)$ replaced by $[0,1]$
and $ \mathscr P^R$ satisfying   
\begin{equ} 
\partial_X \mathscr P^R_T (X,Y)\Big\vert_{X=0} =A  \mathscr P^R_T (0,Y) \;,
\quad
\partial_X \mathscr P^R_T (X,Y)\Big\vert_{X=1}
= - B \mathscr P^R_T (1,Y) \;,  
\end{equ}
for all $T> 0,Y\in (0,1)$.
\end{definition}


\begin{remark} \label{rem:bcKPZBur}
At least on the formal level, the Robin boundary condition for SHE \eqref{e:SHE}
corresponds to (via the Cole-Hopf transformation $ H_T(X) = \log \mathscr Z_T(X)  $) the inhomogeneous Neumann boundary condition for KPZ \eqref{e:KPZ}. For instance, if $\partial_X \mathscr Z_T(0) = A \mathscr Z_T(0)$,
then $\partial_X H_T(0)=A$.
In fact this also corresponds to the inhomogeneous  Dirichlet boundary condition for the stochastic Burgers equation
\begin{equ} [e:Burgers]
\partial_T u = \frac12 \partial_X^2 u + \frac12 \partial_X (u^2) + \partial_X \dot W
\end{equ}
where $u\eqdef \partial_X H$ with boundary condition $u_T(0)=A$.
Also $\partial_X \mathscr Z_T(1) = -B \mathscr Z_T(1)$
corresponds to  $\partial_X H_T(1)=-B$ and $u_T(1)=-B$. This is only formal because the KPZ equation is not differentiable and the stochastic Burgers equation does not take function valued solutions. This is why we have defined the Robin boundary condition in terms of the heat kernel and not the derivative of $\mathscr Z$.
After completing our present paper, \cite{MateMartin2017} provided solution theories to a class of singular SPDEs with boundary conditions, which includes KPZ equation with Neumann boundary conditions.
\end{remark}

\begin{proposition}
Let $A,B\ge 0$, $\bar T>0$, and $I$ be the interval $[0,1] $ or $\R_+$.
Given an initial data $\mathscr Z_0 \in C(I)$ satisfying
\[
\sup_{X\in I } e^{-b X} \E (\mathscr Z_0(X)^p) <\infty
\]
for some $b=b_p$ for all $p>0$,
there exists a  mild solution to SHE \eqref{e:SHE} in $C([0,\bar T],  C(I))$
with Robin condition \eqref{e:SHERobin0} if $I=\R_+$ or \eqref{e:SHERobin1} if $I=[0,1]$.
The mild solution is unique in the class of adapted processes satisfying
\begin{equ} [e:scrZ-growth]
\sup_{T\in [0,\bar T]} \sup_{X\in I } e^{-a X} \E (\mathscr Z_T(X)^2) <\infty
\end{equ}
for some $a>0$.
Finally, assuming that $\mathscr Z_0$ is given by a nonnegative measure, $\mathscr Z_T(X)$ is almost surely strictly positive for all $X\in I$ and all $T\in [0,\bar T]$.
\end{proposition}


\begin{proof}
Existence and uniqueness of mild solutions to the SHE \eqref{e:SHE} was shown in \cite[Theorem~3.2]{MR876085}
in  the case of a bounded interval $[0,1]$ with Neumann boundary condition.
The case of bounded interval $[0,1]$ or
half line $\R_+$ with Robin condition follows {\it mutatis mutandis},
so we only sketch the proof here.

First of all, we note that although \cite[Theorem~3.2]{MR876085} stated that the ``weak solution" (in the PDE sense, i.e. the equation \eqref{e:SHE} when integrating both sides against smooth test functions with corresponding boundary condition) exists and is unique, the proof there actually showed the existence and uniqueness of the {\it mild} solution.

The proof in \cite{MR876085}
only used the following properties
 of the (Neumann) heat kernel $\mathscr P^N$: for each $\bar T>0$ there exists a constant $C(\bar T)$ such that for all $T<\bar T$,
\[
\mathscr P^N_T(X,Y) \le \frac{C(\bar T)}{\sqrt T} e^{-\frac{|Y-X|^2}{4T}}
\]
(see \cite[(3.7)]{MR876085}, which actually also had a factor $e^{-T}$ due to an extra mass term considered in the equation therein but it is never used as long as one does not care about integrability at $T=\infty$.)
For Robin heat kernel $\mathscr P^R$ we can still prove the above bound:
see Lemma~\ref{lem:contP-bound}. Note that our bound has $2T$ in the exponent
instead of $4T$ because our diffusive term is $\frac12 \Delta \mathscr Z$.
Therefore the existence and uniqueness follows for $I=[0,1]$. For $I=\R_+$,
the only tweak in the proof is that one needs the following bound:
\[
\int_{0}^\infty \mathscr P^R_T(X,Y) e^{a Y} dY
\le C(\bar T) e^{a X} \qquad (X>0)
\]
which is easy to show.

The positivity result follows exactly as in Mueller's full line proof \cite{MR1149348} and we do not reproduce it here.
\end{proof}

We turn now to describe how the above continuum equations arise from our discrete ASEP models. The following parameterizations of the boundary parameters will be assumed throughout the remainder of this paper. The reason we make this choice of parameterization is that it enables us to write the Cole-Hopf / G\"artner transform of ASEP in terms of a discrete SHE with Robin boundary conditions.

\begin{definition}[Boundary parameterization] \label{def:parameterization}
For ASEP-B we parameterize $\alpha$ and $\gamma$ by a single parameter $\mu_A\in [\sqrt{q/p},\sqrt{p/q}]$ and $\beta$ and $\delta$ by a  single parameter $\mu_B\in [\sqrt{q/p},\sqrt{p/q}]$ as follows:
\begin{equs} [e:choose-4par]
\alpha &=\frac{p^{\frac32} (\sqrt{p} -\mu_A \sqrt{q})}{p -q}\;,\qquad
\beta=\frac{p^{\frac32} (\sqrt{p} -\mu_B \sqrt{q})}{p -q} \;,\\
\gamma&=\frac{q^{\frac32} (\sqrt{q} -\mu_A \sqrt{p})}{q - p}\;,\qquad
\delta=\frac{q^{\frac32} (\sqrt{q} -\mu_B \sqrt{p})}{q - p}\;.
\end{equs}
For ASEP-H we parameterize using the above formulas for $\alpha$ and $\gamma$.
\end{definition}

\begin{remark} \label{rem:rho-rates}
The condition on $\mu_A,\mu_B \in [\sqrt{q/p},\sqrt{p/q}]$ is necessary and sufficient for all rates to be positive. These choices of rates satisfy the simple relations
$$
\frac{\alpha}{p} + \frac{\gamma}{q} = 1\;,\qquad \frac{\beta}{p} + \frac{\delta}{q} = 1\;.
$$
In fact, these relations are exactly the relations assumed by Liggett in \cite{MR0410986,MR0445644} when considering ASEP on  the half line and bounded interval. In the open ASEP case, plugging these parameterizations into \eqref{eqab} yields simple formulas for
$$a=\frac{\mu_A \sqrt{pq}-q}{p-\mu_A\sqrt{pq}}\;,\qquad    b=\frac{\mu_B \sqrt{pq}-q}{p-\mu_B\sqrt{pq}}\;.$$
By \eqref{eqvarrho}, we may compute the effective densities $\rho_A = \alpha/p$ and $\rho_B= \delta/q$. In the special one-parameter subcase when the effective densities on the left and right are equal, i.e. $\varrho_A=\varrho_B$ in \eqref{eqvarrho}, then one may check (see, e.g. \cite{Goncavlesetal}) that the invariant measure for open ASEP is a product of Bernoulli random variables with density $\varrho_A$. The equality $\varrho_A=\varrho_B$ amounts to the relationships between $\mu_A$ and $\mu_B$ that
$$\mu_A = \frac{p+q -\mu_B \sqrt{pq}}{\sqrt{pq}}.$$
The special case when $\varrho_A=\varrho_B=1/2$ (and the invariant measure is product Bernoulli with parameter $1/2$) arises when $\mu_A=\mu_B = \frac{p+q}{2\sqrt{pq}}$.
\end{remark}

We now define the weakly asymmetric scaling which will be assumed throughout the remainder of this paper (whether stated explicitly or not). All functions, e.g. $Z_t(x)$, and all parameters will implicitly depend on $\eps$ throughout (though our notation will generally not make this explicit).

\begin{definition}[Weakly asymmetric scaling] \label{def:weak-asym}
Let $\eps>0$ be a small enough so that all parameters defined below are positive. Define, for $\e>0$
\begin{equ} [e:p-q]
q=\tfrac{1}{2}e^{-\sqrt{\eps}}\qquad \mbox{and}\qquad
p=\tfrac{1}{2}e^{\sqrt{\eps}}.
\end{equ}
In other words the particles jump to the right at a higher rate $p>1/2$ than to the left at rate $q<1/2$ (which is opposite the convention in \cite{BG}).
For the ASEP-H, the above definition is sufficient. For the ASEP-B, we further assume that $\eps = 1/N$ where $N$ is the size of the bounded interval. The boundary parameters are likewise weakly scaled so that $\mu_A=1-\eps A$ and $\mu_B=1-\eps B$ for some constants $A,B\ge 0$.
See Remark 3.2 for more on the restriction that $A, B \geq 0$.
\end{definition}

\begin{remark}\label{rem:asyexp}
Given the weakly asymmetric scaling, we have the following expansions (in $\eps$ small) for our model parameters:
$$
p= \tfrac{1}{2}+\tfrac{1}{2}\sqrt{\e} + O(\e),\quad
q= \tfrac{1}{2}-\tfrac{1}{2}\sqrt{\e} + O(\e),\quad
\mu_A= 1- A \e,\quad
\mu_B= 1- B \e,
$$
$$
\alpha= \tfrac{1}{4} + \left(\tfrac{3}{8}+\tfrac{1}{4}A\right) \sqrt{\e} + O(\e),\qquad
\beta= \tfrac{1}{4} + \left(\tfrac{3}{8}+\tfrac{1}{4}B\right) \sqrt{\e} + O(\e),
$$
$$
\,\,\gamma= \tfrac{1}{4} - \left(\tfrac{3}{8}+\tfrac{1}{4}A\right) \sqrt{\e} + O(\e),\qquad\,
\delta= \tfrac{1}{4} - \left(\tfrac{3}{8}+\tfrac{1}{4}B\right) \sqrt{\e} + O(\e).\,\,
$$
We also have expansions of the open ASEP phase diagram parameters and the associated effective densities:
$$
a= 1 - (1+2A)\sqrt{\e} + O(\e),\qquad
b=  1 - (1+2B)\sqrt{\e} + O(\e),
$$
$$
\varrho_A= \tfrac{1}{2} + \left(\tfrac{1}{4} + \tfrac{1}{2} A\right) \sqrt{\e} + O(\e^{3/2}),\qquad
\varrho_B= \tfrac{1}{2} - \left(\tfrac{1}{4} + \tfrac{1}{2} B\right) \sqrt{\e} + O(\e^{3/2}).
$$
From this, one sees that we are exactly tuning into a $\eps^{1/2}$-scale window around the triple critical point in the open ASEP phase diagram in Figure \ref{Figure2}.
\end{remark}

We now define the microscopic Cole-Hopf / G\"artner transformed  process \cite{MR931030} defined as follows.

\begin{definition}[G{\"a}rtner  transformation] \label{def:microCH}
For $x\in\Z_{\ge 0}$ in the case of ASEP-H,
or $x\in \{0\}\cup \Lambda_N$ in the case of ASEP-B, let
\begin{equ} [e:Z-defn]
Z_t(x) \eqdef e^{-\lambda h_t(x) + \nu t }
\end{equ}
where
\begin{equ} [e:lambda-nu]
\lambda \eqdef \frac12 \log\frac{q}{p}
\qquad
\nu \eqdef p +q -2\sqrt{pq} \;.
\end{equ}
Note that $\lambda<0$.
We extend $ Z_t(x) $ by linear interpolation to all $ x\in\R_+ $ for ASEP-H,
and to all
$ x\in [0,N] $ for ASEP-B,
 so that $ Z \in D([0,\infty),C(I)) $ with $I=\R_+$ or $I=[0,N]$,
 which is
the space of $ C(I) $-valued, right-continuous-with-left-limits processes.
\end{definition}

\begin{remark}
We remark that \cite{BG} adopted the convention that $q>p$ and the height profile grows upward, while we follow the convention here that $q<p$ and the height profile grows downward. In \cite{BG} the height profile {\it subtracting} a positive drift term scales to the KPZ equation \eqref{e:KPZ} with a negative sign in front of the nonlinearity, while in our case here since $-\lambda >0, \nu>0$,
the height profile {\it plus} a positive drift term scales to the KPZ equation \eqref{e:KPZ} with a positive sign in front of the nonlinearity.
\end{remark}

Let $\|f_t(x)\|_n \eqdef (\E|f_t(x)|^n)^{\frac1n}$ denote the $L^n$-norm.
Following \cite{BG}, we consider the following near equilibrium initial conditions for ASEP-H. For the purpose of the assumption we include a super-script $\e$ in $h_t$ and $Z_t$ to explicitly identify the $\e$-dependence of the model. This $\e$ will generally be suppressed in what follows.

\begin{assumption}[Near equilibrium i.c. for ASEP-H]\label{def:nearEq}
We assume that the sequence of $\e$-indexed initial height functions $ \{h^\e_0(\Cdot)\}_\e $
or the G{\"a}rtner transformed functions  $ \{Z^\eps_0(\Cdot)\}_\e $ (defined via $h^\eps_0$ as in \eqref{e:Z-defn})
associated with the ASEP-H
 is ``near equilibrium", namely   for any $\alpha \in(0,\frac12)$ and every $n\in\N$
there exist finite constants $C$ and $a$
 such that for every $x,x'\in\R_+$ and every $\eps>0$ one has
\begin{equs}
	\label{e:init-uniform}
	\Vert Z^\e_0(x) \Vert_n
	&\le C e^{a\eps x} \;,
	\\
	\label{e:init-Holder}
	\Vert Z^\e_0(x)-Z^\e_0(x')\Vert_n
	&\le C (\eps|x-x'|)^\alpha e^{a\eps(x+x')} \;.
\end{equs}
\end{assumption}

For ASEP-B we do not need to impose any growth condition at infinity,
but we still have the following bound that is uniform in $\eps$. Recall that $\eps$ and $N$ are related by $\eps=1/N$.

\begin{assumption}[Initial condition for ASEP-B] \label{def:icB}

For ASEP-B on $\{0\} \cup \Lambda_N$,
we assume that a sequence of $\e$-indexed initial height functions $ \{h^\eps_0(\Cdot)\}_\e $
or the G{\"a}rtner transformed functions  $ \{Z^\eps_0(\Cdot)\}_\e $ (defined via $h^\eps_0$ as in \eqref{e:Z-defn})
associated with the ASEP-B is such that,
for any $\alpha \in(0,\frac12)$ and every $n\in\N$
there exists a finite constant $C$, so that for every $x,x'\in [0,N]$
and every $\eps>0$ one has
\begin{equs}
	\label{e:init-uniform-B}
	\Vert Z^\e_0(x) \Vert_n
	&\le C  \;,
	\\
	\label{e:init-Holder-B}
	\Vert Z^\e_0(x)-Z^\e_0(x')\Vert_n
	&\le C (\eps|x-x'|)^\alpha  \;.
\end{equs}
\end{assumption}

Given $\bar T>0$, we endow the space $D([0,\bar T],C(\R_+))$ and $D([0,\bar T],C([0,1]))$ the Skorokhod topology
and the space $ C(\R_+) $ with the topology of uniform convergence on compact sets,
and use $ \Rightarrow $ to denote weak convergence of probability laws. When processes converge in these topologies, we refer to it as convergence as a space-time process.

\begin{definition}[Scaled processes]\label{def:scaled}
The scaled processes $\mathcal Z$ are defined by
\begin{equ}\label{e:calZ}
	\mathcal Z_{T}^\eps(X)
	\eqdef Z_{\eps^{-2} T}(\eps^{-1} X)\;,  \qquad T\in [0,\bar T] 
\end{equ}
where $Z$ (which depends on $\e$ through the model parameters) is defined in \eqref{e:Z-defn}.
For ASEP-H, \eqref{e:calZ} is defined for all  $X\in\R_+$, and $\mathcal Z^\eps \in D([0,\bar T],C(\R_+)) $;
 for ASEP-B, \eqref{e:calZ} is defined for all  $X\in [0,1]$,
and $\mathcal Z^\eps \in D([0,\bar T],C([0,1])) $.
\end{definition}

We state  our main theorems for the convergence of these scaled processes. Recall the definition of the mild solution to the SHE, the choice of ASEP parameters, weakly asymmetric scaling, and scaled processes given in the above definitions.

\begin{theorem}\label{thm:main-H}
%
%
Given any initial conditions $Z^\e_0$ satisfying Assumption~\ref{def:nearEq}
such that $ Z^\e_0 \Rightarrow \mathscr Z^{ic} $ as $ \e \to 0 $, where $ \mathscr Z^{ic} \in C(\R_+)$,
then $\mathcal Z^\eps \Rightarrow \mathscr Z$ as a space-time process,
as $ \eps \to 0 $, where
$ \mathscr Z $
is the unique mild solution to the SHE \eqref{e:SHE} on $\R_+$ from the initial data $ \mathscr Z^{ic} $ satisfying the Robin boundary condition with parameter $A$.
\end{theorem}

\begin{theorem}\label{thm:main-B}
Given any 
initial conditions $Z_0^\e$ satisfying Assumption~\ref{def:icB}
such that $ Z^\e_0 \Rightarrow \mathscr Z^{ic}$ as $ \e \to 0 $,
where $ \mathscr Z^{ic} \in C([0,1])$,
then $ \mathcal Z^\eps \Rightarrow \mathscr Z $ as a space-time process,
as $ \eps \to 0$, where
$ \mathscr Z $
is the unique mild solution to SHE \eqref{e:SHE} on $[0,1]$ from the initial data  $ \mathscr Z^{ic} $ satisfying the Robin boundary condition with parameters $(A,B)$.
\end{theorem}

\section{Hopf-Cole / G\"artner transform and the microscopic SHE}\label{Sec3}

We derive the microscopic SHE satisfied by the Hopf-Cole / G\"artner transformation \cite{MR931030,dittrich91} of ASEP-H and ASEP-B (recall Definition \ref{def:microCH}). For open ASEP, this transformation and the associated boundary parameterization which yields Robin boundary conditions was given recently in  \cite{gonccalves2015nonequilibrium} (though we were unaware of that work until it was kindly pointed out to us by Gon{\c{c}}alves).

We will assume the choices of parameters given in Definition \ref{def:parameterization} and the weakly asymmetric scaling given in Definition \ref{def:weak-asym}.
\begin{lemma}\label{lem:HC}
For ASEP-H, $ Z_t(x) $ defined in \eqref{e:Z-defn} satisfies
	\begin{equ}\label{e:DSHE}
		d Z_t(x) = \tfrac12 \Delta Z_t(x)\,dt + dM_t(x)
	\end{equ}
for all $x\in \Z_{\ge 0}$, with Robin boundary condition at $x=0$:
\begin{equ} [e:DBleft]
\mu_A\big(Z(-1)-Z(0)\big) +(1-\mu_A)Z(-1)=0 \;,\quad i.e. \;\; Z(-1)=\mu_A Z(0),
\end{equ}
where $M_{\Cdot}(x)$ for $x\in \Z_{\geq0}$ are martingales with bracket processes
\begin{equs}
\frac{d}{dt}\langle M(x),M(y) \rangle_t &=0 \qquad (x\neq y) \\
\frac{d}{dt}\langle M(x),M(x) \rangle_t
	&=   \Big(\big(\tfrac{q}{p}-1\big)^2  c^R(\eta_t,x)
			+\big(\tfrac{p}{q}-1\big)^2  c^L(\eta_t, x) \Big) Z_t(x)^2   \qquad (x>0) \\
\frac{d}{dt}\langle M(0),M(0) \rangle_t
	&=   \Big(\big(\tfrac{q}{p}-1\big)^2  r_A^+(\eta_t(1))
			+\big(\tfrac{p}{q}-1\big)^2  r_A^-(\eta_t(1)) \Big) Z_t(0)^2  \;.
\end{equs}	
For ASEP-B, \eqref{e:DSHE} holds for all $x \in \Lambda_N$, with Robin boundary conditions at $x=0$ and $x=N$:
\begin{equ} [e:DBright]
Z(-1)=\mu_A Z(0)\;, \qquad
Z(N+1)=\mu_B Z(N) \;,
\end{equ}
where $M_{\Cdot}(x)$ for $x\in \Lambda_N$ are martingales with bracket processes the same as in the ASEP-H case for $x\in \Lambda_N \backslash \{N\}$ while at the boundaries $x=N$,
\begin{equs}
\frac{d}{dt}\langle M(  N),M(  N) \rangle_t
	=   \Big(\big(\tfrac{q}{p}-1\big)^2  r_B^-(\eta_t(  N))
			+\big(\tfrac{p}{q}-1\big)^2  r_B^+(\eta_t(  N)) \Big) Z_t(  N)^2 \;.
\end{equs}
%
%
%
%
%
\end{lemma}

\begin{proof}
We first consider ASEP-H.
Since only the jumps can affect the value of $h_t(x)$ for $x>0$,
and only the creation / annihilation at $x=1$ can affect the value of $h_t(0)$,
we have, by definition \eqref{e:Z-defn} of $Z$ and definition of ASEP-H,
\[
dZ_t(x)= \Omega(x) Z_t (x)\,dt+ dM_t(x)
\]
where the drift term (we suppress the dependence in $t$ to lighten our notation)
\begin{equ}
\Omega(x) =
\begin{cases}
  \nu + (e^{-2\lambda}-1 ) c^L (\eta,x)+ (e^{2\lambda}-1 ) c^R(\eta,x)
	&\qquad (\mbox{if } x>0) \;,  \\
\nu + (e^{-2\lambda}-1 ) r_A^-(\eta(1)) + (e^{2\lambda}-1 ) r_A^+(\eta(1))
	&\qquad (\mbox{if } x=0) \;,
\end{cases}
\end{equ}
and $ M_\Cdot(x) $, $ x\in\Z_{\ge 0} $ are martingales with the desired bracket process (this can be checked easily).

The relation $\Omega (x)Z(x) =\frac{D}{2}\Delta Z(x)$ for $x>0$ can be achieved by setting
\[
\lambda = \frac12 \log\frac{q}{p} \;,
\qquad
\nu = p +q -2\sqrt{p q} \;,
\qquad
D=2\sqrt{pq}
\]
as in the case of the standard ASEP without boundary, see for instance \cite{CorwinReview}. Note that with the choice \eqref{e:p-q} we have $D=1$ and $\lambda<0$; however in the following calculations we will represent all the quantities in terms of $p,q$ without specifying $p,q$ as in \eqref{e:p-q}.

With the constants $\lambda,\nu,D$ determined as above, for $x=0$ we have
\begin{equ} [e:Omega0]
\Omega(0) =
\begin{cases}
 p +q-2\sqrt{pq} + \Big(\frac{q}{p}-1\Big)\alpha &\qquad (\mbox{if } \eta(1)=-1) \;, \\
  p +q -2\sqrt{pq} + \Big(\frac{p}{q}-1\Big)\gamma &\qquad (\mbox{if } \eta(1)=1)\;.
 \end{cases}
\end{equ}
We want to match $\Omega (0)Z (0)$ with $\frac{D}{2}\Delta Z (0)$, with certain ``outer" boundary condition imposed on $Z(-1)$.
Recall that by definition $Z(0)=e^{-\lambda h(0)+\nu t}$, and
 $Z(1)=e^{-\lambda h(0)-\lambda \eta(1)+\nu t}$.

For Robin boundary condition
\[
\mu_A\big(Z(-1)-Z(0)\big) +(1-\mu_A)Z(-1)=0
\qquad \mbox{i.e. } Z(-1)-\mu_A Z(0) =0
\]
(where the real parameter $\mu_A$ interpolates the Dirichlet condition $Z(-1)=0$
and the Neumann condition $Z(-1)-Z(0)=0$),
we have
\[
\Delta Z(0) =Z(1)-(2-\mu_A)Z(0) = (e^{-\lambda \eta(1)}-2+\mu_A)\,Z(0) \;.
\]
The condition $\Omega(0) Z (0)=\frac{D}{2}\Delta Z (0)$ is satisfied for $\eta(1)=-1$ if
\begin{equs}
p +q-2\sqrt{pq} + \Big(\frac{q}{p}-1\Big)\alpha
	=\sqrt{pq}\Big(\sqrt{\frac{q}{p}}-2+\mu_A \Big),
\end{equs}
and for $\eta(1)=1$ if
\begin{equs}
p +q-2\sqrt{pq} + \Big(\frac{p}{q}-1\Big)\gamma
	=\sqrt{pq}\Big(\sqrt{\frac{p}{q}}-2+\mu_A \Big).
\end{equs}
Solving these equations lead precisely the choices of $\alpha,\gamma$ in Definition \ref{def:parameterization}.

For ASEP-B, the proof is analogous. The martingale bracket processes are easily computed. The derivation of our $\alpha,\gamma$ parameterization at the left boundary is the same, and one likewise readily checks that the same derivation with respect to the right boundary condition leads to the forms of $\beta,\delta$ in Definition \ref{def:parameterization}.
%
\end{proof}

\begin{remark}
The range of $\mu_A,\mu_B\in [\sqrt{q/p},\sqrt{p/q}]$ assumed in Definition \ref{def:parameterization} was necessary to ensure the non-negativity of the boundary rates. This limits the types of boundary conditions that can arise. For instance, Neumann boundary condition for the SHE is accessible when $\mu_A=\mu_B=1$ (i.e. $A=B=0$). On the other hand, Dirichlet boundary conditions would require $\mu_A=\mu_B=0$ which is certainly out of the range.

In Definition \ref{def:weak-asym}, we assumed further that $\mu_A,\mu_B \le 1$. 
 This is only a technical restriction as it simplifies heat kernel bounds by disallowing for exponential modes. In work in preparation, \cite{Par17} provides the necessary heat kernel bounds to control possible exponential modes and extends our proof to $A,B <0$ (i.e. $\mu_A,\mu_B>1$) as well as to delta function initial data.

%
\end{remark}


Utilizing the weakly asymmetric scalings from Definition \ref{def:weak-asym} -- in particular their asymptotic expansions in Remark \ref{rem:asyexp} leads to the following.

\begin{lemma} \label{lem:Mbracket}
For ASEP-H one has
\begin{equ} [e:Mest]
\frac{d}{dt}\langle M(x),M(x) \rangle_t =
\begin{cases}
	 \eps Z_t(x)^2
		 -\nabla^+ Z_t(x) \nabla^- Z_t(x) +o(\eps)Z_t(x)^2  &  (\mbox{\textrm{if }} x>0) \\
	 \eps  Z_t(x)^2 + o(\eps)Z_t(x)^2 & (\mbox{\textrm{if }} x=0)
\end{cases}
\end{equ}
where $o(\eps)$ is a term uniformly bounded by constant $C_\eps$ and
	$C_\eps / \eps \to 0$.

For ASEP-B the first estimate in \eqref{e:Mest} holds for every $x\in\{1,\cdots,N-1\}$
and the second estimate in \eqref{e:Mest} holds for $x\in\{0,N\}$.
\end{lemma}

\begin{proof}
The proof for the bulk (the first equation in \eqref{e:Mest}) is standard,
see 
\cite[Proposition~2.1(b)]{CST2016asep} with the parameter $j$ therein set as $j=1/2$.
We only remark that the term $\nabla^+ Z_t(x) \nabla^- Z_t(x)$ arises
from the quadratic term $\eta(x)\eta(x+1)$ in the definition of the rates $c^L,c^R$.
At $x=0$, we use the expansions in Remark \ref{rem:asyexp} to immediately conclude the desired behavior in $\e$.
The proof for the case of ASEP-B is analogous.
\end{proof}

\section{Tightness}\label{Sec4}

\subsection{Estimates of Robin heat kernels}
\label{sec:techRobin}

Before proceeding with the proof of tightness of the sequence $\CZ^\eps$, we need some estimates on discrete  heat kernels with Robin boundary conditions.

\subsubsection{Elastic Brownian motion and random walk}

A (discrete time) random walk $\tilde R_H^e$ on $\Z_{\ge 0}$
with elastic boundary at $-1$
 jumps in the same way as the standard random walk
when it is at $x>0$; and if it is at $0$, then with probability $1/2$ it jumps to the site $1$,
 with probability $\mu_A /2$ it stays at the site $0$ (i.e. ``reflected back by the wall $-1$"), and finally  with probability $(1-\mu_A)/2$ it is killed (i.e. ``absorbed by the wall $-1$").
The {\it continuous time  walk $R_H^e$ on $\Z_{\ge 0}$
with elastic boundary}
 is then defined via $\tilde R_H^e$ by imposing exponential holding time before each jump  in the usual way.
The heat kernel $\bfp^R$
satisfying the Robin condition \eqref{e:DBleft}  is the
 transition probability  for  $R_H^e$
 which is represented as
\begin{equ} [e:discHKrep]
\bfp_t^R(x,y) = e^{-t} \sum_{n=0}^\infty \frac{t^n}{n!} p_n^e(x,y)
\end{equ}
where $p_n^e(x,y)$ is the transition probability of the discrete time random walk $\tilde R_H^e$.
A (discrete time) random walk $\tilde R_B^e$ on $\Lambda_N=\{0,1,\cdots,N\}$
with elastic boundary at both ends
is defined in the same way,
except that if it is at $0$ (resp. at $N$),
 then
 with probability $\mu_A /2$ (resp. at $\mu_B /2$) it stays there  
 and  with probability $(1-\mu_A)/2$ (resp. $(1-\mu_B)/2$) it is killed. 
The {\it continuous time  random walk $R_B^e$ on $\Lambda_N$ with elastic boundary} is then defined via $\tilde R_B^e$ analogously as the half line case.
The heat kernel $\bfp^R$
satisfying the Robin condition \eqref{e:DBright}  is the
 transition probability  for  $R_B^e$.
See for instance \cite{naqvi1982symmetric}.

\begin{remark}
Here and in what follows we will use $\bfp^R$ to represent the discrete heat kernel in both the $\Z_{\geq 0}$ and $\Lambda_N$ cases. Likewise we will use $\mathscr P_T^R$ (introduced below) for the continuous heat kernel in both the $\R_{+}$ and $[0,1]$ cases.
\end{remark}

A {\it Brownian motion $B_H^e$ on $\R_+$ with elastic boundary at $0$}  is defined as
$B^e_H(t) = |B(t)|$ if $t<\mathfrak m$ and killed at time $\mathfrak m$
where $\mathfrak m \eqdef \mathfrak t^{-1} (\mathfrak e / A)$, $\mathfrak e$ is an independent exponential random variable of rate 1,
and $\mathfrak t^{-1}$ is the inverse function of the reflecting Brownian local time
\[
\mathfrak t (t) = \lim_{\eps\to 0} (2\eps)^{-1} \mbox{Leb}(s:|B(s)|<\eps,0\le s\le t)
\]
where $\mbox{Leb}$ is the Lebesgue measure. Obviously if $A=0$ then $B_H^e(t)$ is simply the reflected Brownian motion $|B(t)|$, and if $A\to \infty$ one recovers the Brownian motion
killed at the origin.
The kernel $\mathscr P^R$ which satisfies the Robin boundary condition with parameter $A$ 
 is the transition probability of $B^e_H$.
A {\it Brownian motion $B_B^e$ on $[0,1]$ with elastic boundary at both ends} is defined as
$B^e_B(t) = B^e_H(t)$ if $t<\mathfrak n$ and killed at time $\mathfrak n$
where $\mathfrak n \eqdef \mathfrak t^{-1} (\mathfrak e / B)$, $\mathfrak e$ is as above,
and $\mathfrak t^{-1}$ is the inverse function of the local time
\[
\mathfrak t (t) = \lim_{\eps\to 0} (2\eps)^{-1} \mbox{Leb} (s:B^e_H(s) >1-\eps,0\le s\le t)\;.
\]
The kernel $\mathscr P_t^R$ on $[0,1]$ which satisfies the Robin boundary condition with parameter $A,B$ is the transition probability of $B_B^e$.
See  for instance \cite{ito1963brownian}. 

From these probabilistic interpretations one immediately has the following simple properties.

\begin{lemma} \label{lem:semigp}
The above kernels $\mathscr P_T^R$ 
 and $\bfp_t^R$ 
 satisfy the semi-group properties,
and are both non-negative at all space-time points.
Also the kernels $\mathscr P_T^R$ on $I\in \{\R_+,[0,1]\}$
 and $\bfp_t^R$ on $\Lambda \in \{\Z_{\ge 0},\{0,\cdots,N\}\}$  satisfy
\[
\int_{I} \mathscr P_t^R(X,Y)\,dY \le 1
 \qquad \sum_{y\in \Lambda}\bfp_t^R(x,y) \le 1 \;.
 \]
The inequalities become equalities if and only if
$A=0$ (or $\mu=1$).
\end{lemma}


\begin{lemma} \label{lem:contP-bound}
The kernels $\mathscr P_T^R$ on $\R_+$ or $[0,1]$
satisfy the following upper bound.
For each $\bar T>0$ there exists a constant $C(\bar T)$ such that for all $T<\bar T$,
\[
\mathscr P^R_T(X,Y) \le \frac{C(\bar T)}{\sqrt T} e^{-\frac{|Y-X|^2}{2T}} \;.
\]
\end{lemma}

\begin{proof}
For the reflected Brownian motion on $[0,1]$, for which the Robin heat kernel $\mathscr P^R$
reduces to the Neumann heat kernel $\mathscr P^N$, this is proved in \cite[Eq.~(3.7)]{MR876085} (except that
the equation considered therein has an extra linear damping term and thus the resulting bound has an extra factor $e^{-T}$).
The proof of \cite{MR876085} relies on the explicit expression of $\mathscr P^N$ via the image method.
Since the transition probability of the elastic Brownian motion must be
smaller than the reflected Brownian motion, one immediately obtains the claimed upper bound
for $\mathscr P^R$.
For the case of $\R_+$, one can again prove the claimed upper bound for $\mathscr P^N$ by the method of images and then the bound for $\mathscr P^R$ follows from it. Note that this bound is also easily proved for $\mathscr P^R$ directly by applying the formula in Lemma~\ref{lem:scrP-H} below, which is derived by the image method.
\end{proof}

\subsubsection{A generalized image method}

We start by providing a formula for $\mathscr P_T^R$ using a generalization of the image method (to the Robin boundary condition case). We will not make much use of this continuous space formula, but include it since it motivates the more complicated discrete image method formulas provided in Lemmas \ref{lem:bfp-H} and \ref{lem:Robin-Dint-abs}. In scanning the literature, we came upon some continuous space generalizations of the image method on $\R$ with some Robin boundary condition at the origin (see \cite{12103662, 13014935,  150606560}. Discrete space, and the boundary interval geometry leads to much more involved calculations.

\begin{lemma} \label{lem:scrP-H}
The half line heat kernel $\mathscr P^R_T(X,Y)$ for $X,Y\ge 0$ that satisfies the Robin boundary condition $\partial_X \mathscr P^R_T(X,Y)\Big\vert_{X=0}=A \mathscr P^R_T (0,Y)$
has the following representation
\begin{equ} [e:Robin-cont-HK]
\mathscr P_T^R(X,Y)=
P_T(X-Y)+P_T(X+Y)
	-2A \int_{-\infty}^0  P_T(X+Y-Z)e^{AZ}\, dZ
\end{equ}
where $P_T(X) =\frac{1}{\sqrt{2\pi T}} e^{-X^2 / (2T)}$ is the standard continuous heat kernel for the heat operator $\partial_T - \frac12 \Delta$.
\end{lemma}

\begin{proof}
We will prove \eqref{e:Robin-cont-HK} by showing that given any function $\phi(X)$ on $X\in \R_+$,
\[
u(T,X)=\int_0^\infty \mathscr P^R_T(X,Y) \phi(Y) dY
\]
 solves the equation $\partial_T u = \frac12 \Delta u$ on $\R_+ \times \R_+$ with initial condition $u(0,X)=\phi(X)$ ($X\in\R_+$) and Robin boundary condition
 \[
 \partial_X u(X,T)\Big\vert_{X=0} =A u(0,T) \;.
\]
The above equation can be solved by extending the initial values $\phi$
 so that $\phi'-A\phi$
 is odd, namely, for $X<0$
 one should solve the ODE
 \begin{equ} [e:ext-phi-HC]
 \phi'(X)-A\phi(X) = -\phi'(-X) + A\phi(-X), \quad X<0 \;,
 \end{equ}
and then solve the heat equation on the entire line $\R_+$ with the extended $\phi$ as initial data. Indeed $u(T,X)=\int_{\R} P_T(X-Y)\phi(Y)dY$
satisfies
\begin{equs}
 \partial_X u(X,T)\Big\vert_{X=0} -A u(0,T)
 &=\int_{\R} P_T'(-Y)\phi(Y)dY-A\int_{\R} P_T(-Y)\phi(Y)dY \\
&=\int_{\R} P_T (-Y)(\phi' (Y)  -A \phi(Y) )dY
\end{equs}
which vanishes because $P_T$ is even and  $\phi' -A \phi$ is odd.

To solve \eqref{e:ext-phi-HC}, we use an integrating factor $e^{-AX}$ and get
\[
(e^{-AX}\phi(X))'  = -e^{-AX}(\phi'(-X)-A\phi(-X)),\quad X<0
\]
so
\[
\phi(X) = Ce^{AX} +e^{AX}\int_X^0 e^{-As}(\phi'(-s)-A\phi(-s))\,ds,\quad X<0 \;.
\]
Choose $C$
 so that the extended function $\phi$
 is continuous at $0$, that is $C=\phi(0)$.
 This can be simplified as
\[
\phi(X) = \phi(-X) - 2 A e^{AX} \int_0^{-X} e^{As} \phi(s) \,ds\;,
\qquad X<0\;.
\]
so that
\begin{equs}
u(T,X) &=\int_0^\infty P_T(X-Y) \phi(Y) \,dY
+ \int_{-\infty}^0 P_T(X-Y) \phi(-Y) \,dY \\
& \qquad -2A \int_{-\infty}^0 P_T(X-Y) e^{AY}  \int_0^{-Y} e^{As} \phi(s) \,ds\,dY \\
&=\int_0^\infty \bigg(P_T(X-Y)+P_T(X+Y)
	-2A \int_{-\infty}^0  P_T(X+Y-Z)e^{AZ}\, dZ \bigg) \phi(Y) \,dY
\end{equs}
This shows that the Robin heat kernel is given by \eqref{e:Robin-cont-HK}.
\end{proof}

%

We turn now to the analogous result for the discrete space half line heat kernel.

\begin{lemma} \label{lem:bfp-H}
The half line discrete heat kernel $\bfp^R_t(x,y)$ for $x,y \in \Z_{\ge 0}$ that satisfies the Robin boundary condition $\bfp^R_t(-1,y)=\mu_A \bfp^R_t (0,y)$
has the following representation
\begin{equs} [e:Robin-disc-HK]
\bfp_t^R(x,y)
&=
p_t(x-y)+\mu_A p_t(x+y+1)
	+(\mu_A^2-1) \sum_{z=-\infty}^{-2}
		 p_t(x+y-z) \mu_A^{-z-2}
\end{equs}
where $p$ is the transition probability of standard continuous time random walk on $\Z$ (i.e. jumps left and right by one at rate $1/2$).
\end{lemma}

\begin{proof}
The proof is analogous to that of Lemma~\ref{lem:scrP-H}, but adapted to the discrete setting.
Given $\phi$ on $\Z_{\ge 0}$, one wants to extend it to $\Z$ such that
for $x<0$,
\begin{equ} [e:extendphi]
\phi(x-1)-\mu_A \phi(x) = -\big(\phi(-x-1)-\mu_A \phi(-x)\big) \;.
\end{equ}
The above relation is solved by the function
\begin{equ} [e:extendedphi]
\phi(x) = \mu_A \phi(-x-1) + (\mu_A^2-1) \sum_{k=0}^{-x-2} \mu_A^{-x-2-k}\phi(k)
\qquad x<0
\end{equ}
(if $-x-2<0$ the sum over $k$ is understood as zero).
Therefore the
solution  to the heat equation is given by
\begin{equs}
u(t,x)&=
\sum_{y=-\infty}^\infty p_t(x-y) \phi(y) \\
&=\sum_{y=0}^\infty p_t(x-y) \phi(y)
+ \sum_{y=-\infty}^{-1} p_t(x-y) \mu_A\phi(-y-1)  \\
& \qquad +(\mu_A^2-1)  \sum_{y=-\infty}^{-1}
	 p_t(x-y)   \sum_{k=0}^{-y-2} \mu_A^{-y-2-k} \phi(k) \\
&=\sum_{y=0}^\infty \bigg(p_t(x-y)+\mu_A p_t(x+y+1)
	+(\mu_A^2-1) \sum_{z=-\infty}^{-2-y}
		 p_t(x-z) \mu_A^{-z-y-2} \bigg) \phi(y).
\end{equs}
To check that this $u$ satisfies the Robin boundary condition $u_t(-1)=\mu_A u_t(0)$ for any $t$, observe that
\begin{equs} [e:check-RobinA]
u_t(-1)-\mu_A u_t(0) &=\sum_{y=-\infty}^\infty p_t(-1-y) \phi(y)- \mu_A \sum_{y=-\infty}^\infty p_t(-y) \phi(y)   \\
& =\sum_{y=-\infty}^\infty p_t(-y) ( \phi(y-1)-\mu_A \phi(y)) \;.
\end{equs}
By \eqref{e:extendphi} and $ p_t(-y)= p_t(y)$ the above expression is zero.
So the Robin heat kernel is given by \eqref{e:Robin-disc-HK}.
\end{proof}

From the explicit formulas in Lemma~\ref{lem:scrP-H}
and Lemma~\ref{lem:bfp-H}
one can see that
\[
\mathscr P_t^R(x,y)=\mathscr P_t^R(y,x)\;,
 \qquad
\bfp_t^R(x,y) = \bfp_t^R(y,x) \;,
\]
i.e. they are symmetric in $x,y$. In proving these lemmas we could have just checked that the stated formulas solve the relevant heat equations. We opted for the more detailed derivations since it is informative in attacking the below finite interval case, were we are unable to provide a concise closed form solution to check. Instead, in that case we provide bounds on the solution which suffice for our applications.

%

To find the Robin heat kernel on the discrete finite interval $\Lambda_N$,
we need to apply the above image method in a recursive way.
It will be convenient to introduce the notation
\[
\bar N = N+1
\]
so that $\Z = \cup_{k\in \Z} \{k\bar N, k\bar N+1, \cdots, (k+1) \bar N-1\}$, a union of non-overlapping sets each consisting of $\bar N $ points.
We start with the first several steps of the extension in order to motivate the general case.
Given $\phi$ on $\{0,\cdots,\bar N-1\}$, one first extends it to $\{-\bar N,\cdots,-1\}$ such that
for $x\in\{-\bar N,\cdots,-1\}$,
\begin{equ} [e:disc-int-extA]
\phi(x-1)-\mu_A \phi(x) = -\big(\phi(-x-1)-\mu_A \phi(-x)\big) \;.
\end{equ}
For this we have already obtained above in \eqref{e:extendedphi} that
\begin{equ}  [e:disc-int-ext-N]
\phi(x) = \mu_A \phi(-x-1) + (\mu_A^2-1) \sum_{y=0}^{-x-2} \mu_A^{-x-2-y}\phi(y)
\end{equ}
for $x\in\{-\bar N,\cdots,-1\}$.
We then extend $\phi$ to $\{\bar N,\cdots,2\bar N-1\}$, such that
\begin{equ}  [e:disc-int-extB]
\phi(x)-\mu_B \phi(x-1) = -\big(\phi(2\bar N-x)-\mu_B \phi(2\bar N-x-1)\big)
\end{equ}
for every $x\in \{\bar N,\cdots,2\bar N-1\}$.
Solving the above equation \eqref{e:disc-int-extB} (or directly checking the following result), one has
\begin{equ}  [e:disc-int-extN]
\phi(x) = \mu_B \phi(2\bar N-x-1) + (\mu_B^2-1) \sum_{y=2\bar N-x}^{\bar N-1} \mu_B^{y-2\bar N+x}\phi(y)
\end{equ}
for $x\in\{\bar N,\cdots,2\bar N-1\}$.
We should then iterate this procedure to extend $\phi$ to increasingly larger domains;
but at this point if we were to stop the iteration
 we would get the ``leading order terms" of our heat kernel
\begin{equs}
\bfp_t^R(x,y) & \approx
p_t(x-y)+\mu_A   p_t(x+y+1)+\mu_B   p_t(x+y+1-2\bar N)
	 \\
&+(\mu_A^2-1) \sum_{z=-\bar N}^{-2-y}
		 p_t(x-z) \mu_A^{-z-y-2}
	+(\mu_B^2-1) \sum_{z=2\bar N-y}^{2\bar N-1}
		 p_t(x-z) \mu_B^{z+y-2\bar N} \;,
\end{equs}
for $x,y\in\{0,\cdots,N\}$ (recall that $\bfp_t^R$ is of course always only defined on $x,y\in\{0,\cdots,N\}$). Here ``$\approx$" means that we should actually take into account 
more terms on the right hand side in order to obtain an equality (see the final result \eqref{e:Robin-Dint-abs} below for comparison).

We now further extend $\phi$ to $\{-2\bar N,\cdots,-\bar N-1\}$ according to \eqref{e:disc-int-extA}.
This is possible because the values on $\{0,\cdots,2\bar N\}$
have  already been defined,
and we get \eqref{e:disc-int-ext-N}
for all $x\in \{-2\bar N,\cdots,-\bar N-1\}$.
 We then plug \eqref{e:disc-int-extN} in to get an expression
 only depending on the values of $\phi$ on $\{0,\cdots,\bar N-1\}$.
This eventually yields:
\begin{equs}[e:disc-int-ext-2N]
\phi(x) & =\mu_A \mu_B \phi(2\bar N+x) + (\mu_A^2 -1) \sum_{y=0}^{\bar N-1} \mu_A^{-x-2-y} \phi(y) \\
& + \sum_{y=2\bar N+x+1}^{\bar N-1} \Big( \mu_A (\mu_B^2-1) \mu_B^{y-(2\bar N+x+1)}
	+\mu_B (\mu_A^2-1) \mu_A^{y-(2\bar N+x+1)} \Big) \phi(y)\\
&+ (\mu_A^2-1) (\mu_B^2-1) \sum_{y=2\bar N+x+2}^{\bar N-1}
	\Big(\!\!\!\!\!\! \sum_{\substack{i,j\ge 0\\i+j=y-(2\bar N+2+x)}}\!\!\!\!\!\! \mu_A^i \mu_B^j\Big) \phi(y)
\end{equs}
for every $x\in \{-2\bar N,\cdots,-\bar N-1\}$.

We can iterate this procedure and use the condition \eqref{e:disc-int-extA}
and \eqref{e:disc-int-extB}
{\it in turn} to extend $\phi$ to the entire $\Z$.
Let $u$ be the solution to the heat equation on $\Z$ starting from this extended $\phi$.
Since \eqref{e:disc-int-extA} holds for all $x<0$ by our construction,
as in \eqref{e:check-RobinA} we have $u_t(-1)-\mu_A u_t(0) =0$.
Also, since \eqref{e:disc-int-extB} holds for all $x\ge \bar N$ by construction, we have
\begin{equs}
u_t(\bar N)-\mu_B u_t(\bar N-1)
&=\sum_{y=-\infty}^\infty p_t(\bar N-y) \phi(y)- \mu_B \sum_{y=-\infty}^\infty p_t(\bar N-1-y) \phi(y)   \\
& =\sum_{y=-\infty}^\infty p_t(-y) \big( \phi(\bar N+y)-\mu_B \phi(\bar N-1+y)\big) \;.
\end{equs}
This vanishes because $p_t$ is even, and $\phi(\bar N+y)-\mu_B \phi(\bar N-1+y)$ is odd in $y$
by  \eqref{e:disc-int-extB} (with $x$ in  \eqref{e:disc-int-extB} chosen as $x=\bar N+y$). Therefore $u$ satisfies the desired boundary condition.

We have the following Lemma~\ref{lem:image-phi-Dint}
 to represent the extended $\phi$ in terms of the original given $\phi$ defined on $\Lambda_N$ in a suitable form which is convenient for the following analysis.

For $x\in\{k\bar N,k\bar N+1,\cdots,k\bar N+\bar N-1\}$, we
define 
\begin{equ} [e:def-x-star]
x^\star=
\begin{cases}
 (k+1)\bar N-x-1 &\mbox{if $k$ is an odd integer}\\
x-k\bar N  & \mbox{if $k$ is an even integer. }
 \end{cases}
\end{equ}
It is easy to check that one always has $0\le x^\star \le \bar N-1$. This is the preimage of $x$ under reflection through the two sets of boundaries of $\Lambda_N$.

\begin{lemma} \label{lem:image-phi-Dint}
Let $\phi$ be the function on $\Z$ obtained
by the above recursive extension procedure.
There exists a constant  $C_0$  which only depends on $A,B$, such
 that for each $k\in\Z$,
\begin{equ} [e:inductive-phi]
\phi(x) = I_k \phi(x^\star) +  \eps \sum_{y=0}^{\bar N-1}  E_k (x,y) \phi(y)
\end{equ}
for all $x\in\{k\bar N,k\bar N+1,\cdots,k\bar N+\bar N-1\}$,
where $0<I_k \le 1$ and
\[
 \max_{\substack{k\bar N\le x <(k+1)\bar N \\ 0\le y<\bar N}}|E_k(x,y)| \le C_0^{|k|} \;.
\]
\end{lemma}

One can check that for instance \eqref{e:disc-int-ext-2N} is
indeed of the form \eqref{e:inductive-phi} since $\mu_A^2-1 \sim
\mu_B^2-1 \sim  \eps $ and the sum $\sum \mu_A^i \mu_B^j \sim N\sim \frac{1}{\eps}$.

\begin{proof}
The proof goes by induction.
To begin with, if $k=0$ then $x^\star=x$ so \eqref{e:inductive-phi}
holds with $I_0=1$ and $E_0=0$.
Suppose that the statement  of the lemma is true for $|k|\le m$; we show it for
$k=-m-1$ and $k=m+1$.

Consider the case  $k=-m-1$.
Since $\phi$ is required to satisfy \eqref{e:disc-int-extA},
as in \eqref{e:extendedphi} we have
\begin{equ} [e:kis-m-1]
\phi(x) = \mu_A \phi(-x-1)
	+ (\mu_A^2-1) \sum_{y=0}^{-x-2} \mu_A^{-x-2-y}\phi(y)
\end{equ}
for $x\in [k\bar N,k\bar N+\bar N-1]\cap \Z$,
where $\phi(y)$ on the RHS has been defined
as our induction assumption since $-x-2 \le (m+1)\bar N-2$.
Noting that $m\bar N\le -x-1 \le (m+1)\bar N-1$,
the first term on the RHS of \eqref{e:kis-m-1} is equal to, by inductive assumption,
\begin{equ} [e:-m-1-1st]
\mu_A  I_m \phi\big((-x-1)^\star\big) +\mu_A    \eps \sum_{y=0}^{\bar N-1} E_m(x, y) \phi(y)\;.
\end{equ}
By definition one can check that
\begin{equs}
(-x-1)^\star &= (m+1)\bar N+x = x^\star \qquad \mbox{if $m$ is odd} \\
(-x-1)^\star &= -x-1-m\bar N= x^\star \qquad \mbox{if $m$ is even\;.}
\end{equs}
The second term on the RHS of \eqref{e:kis-m-1} is equal to (noting that $-x-2\le (m+1)\bar N-2$)
\[
 (\mu_A^2-1) \sum_{\ell=0}^m
 	\sum_{y=\ell \bar N}^{(\ell \bar N+\bar N-1)\wedge (-x-2)}
	\mu_A^{-x-2-y} \phi(y)\;.
  \]
Using the inductive assumption, this quantity is equal to
\begin{equ} [e:-m-1-2nd]
 (\mu_A^2-1) \sum_{\ell=0}^m
 	\sum_{y=\ell \bar N}^{(\ell \bar N+\bar N-1)\wedge (-x-2)}
	\mu_A^{-x-2-y}
	 	\Big( I_\ell \phi(y^\star ) + \eps\sum_{z=0}^{\bar N-1}E_\ell (y,z)\phi(z) \Big)\;.
\end{equ}
We now deal with the terms with $I_\ell$
and the terms with $E_\ell$ separately.
Regarding the terms with $I_\ell$ in \eqref{e:-m-1-2nd},
for each fixed $\ell$, as $y$ ranges from $\ell \bar N$ to $(\ell+1)\bar N-1$,
$y^\star$ ranges over $\{0,1,\cdots,\bar N-1\}$,
and this correspondence is one-to-one. So
\begin{equ} [e:-m-1-2nd-1]
\sum_{\ell=0}^m
\sum_{y=\ell \bar N}^{(\ell \bar N+\bar N-1)\wedge (-x-2)}
\mu_A^{-x-2-y}  I_\ell \phi(y^\star )
=
\sum_{y^\star=0}^{\bar N-1} \Big(\sum_{\ell=0}^m  \mu_A^{-x-2-y} I_\ell
	\mathbf 1_{y\le -x-2} \Big)\phi(y^\star )
\end{equ}
where the $y$ in the parentheses on the RHS is determined by $y^\star$  via reversing \eqref{e:def-x-star}, namely,
\begin{equ} [e:def-iota]
y= \iota(y^\star;\ell) \eqdef
\begin{cases}
 (\ell+1)\bar N-y^\star-1 &\mbox{if $\ell$ is an odd integer}\\
y^\star+\ell \bar N  & \mbox{if $\ell$ is an even integer. }
 \end{cases}
\end{equ}
Regarding the terms with $E_\ell$ in \eqref{e:-m-1-2nd},
we have
\begin{equs} [e:-m-1-2nd-2]
\sum_{\ell=0}^m  &
 	\sum_{y=\ell \bar N}^{(\ell \bar N+\bar N-1)\wedge (-x-2)}
	\mu_A^{-x-2-y}
	 	\Big( \eps\sum_{z=0}^{\bar N-1}E_\ell (y,z)\phi(z) \Big)\\
&=
 \eps\sum_{z=0}^{\bar N-1}
 \Big(   \sum_{\ell=0}^m   \sum_{y=\ell \bar N}^{(\ell \bar N+\bar N-1)\wedge (-x-2)}
 	 \mu_A^{-x-2-y}   E_\ell (y,z) \Big)
 \phi(z) \;.
 \end{equs}

Summarizing all the above formulas
\eqref{e:-m-1-1st}, \eqref{e:-m-1-2nd},
\eqref{e:-m-1-2nd-1} and \eqref{e:-m-1-2nd-2},  we have
\begin{equ}
\phi(x) = I_{-m-1} \phi(x^\star) +  \eps \sum_{y=0}^{\bar N-1}  E_{-m-1} (x,y) \phi(y)
 \end{equ}
where
\begin{equ}
I_{-m-1} \eqdef \mu_A I_m  \;,
 \end{equ}
\begin{equs}
E_{-m-1} (x,y)& \eqdef \mu_A E_m(x,y)
+ \eps^{-1} (\mu_A^2 -1)
	\sum_{\ell=0}^m  \mu_A^{-x-2-\iota(y;\ell)} I_\ell
	\mathbf 1_{\iota(y;\ell)\le -x-2} \\
& + (\mu_A^2 -1)  \sum_{\ell=0}^m   \sum_{\bar y=\ell \bar N}^{\ell \bar N+\bar N-1}
 	 \mu_A^{-x-2-\bar y }   \mathbf 1_{\bar y\le -x-2}   E_\ell (\bar y,y)  \;.
 \end{equs}
The coefficient $I_{-m-1}$ clearly satisfies the desired bound since $0<\mu_A\le 1$
and $0<I_m \le 1$ by induction. Since $|\mu_A^2-1|  \le 3A\eps$ for $\eps$ sufficiently small, and $N=1/\eps$, we  have
\begin{equs}
|E_{-m-1} (x,y)| & \le  |E_m(x,y) |
+ 3Am
 + 3A \sum_{\ell=0}^m
 	   \max_{\ell \bar N \le \bar y <(\ell+1) \bar N} |E_\ell (\bar y,y) | \\
& \le
C_0^m + 3Am + 3A \sum_{\ell =0}^m C_0^\ell
 \end{equs}
where we have applied  the inductively assumed bounds on $E_m,E_\ell$.
It is easy to see that there exists $C_0$ which is independent of
$m$
 such that
\[
C_0^m + 3Am + 3A \cdot \frac{C_0^{m+1}-1}{C_0 -1}
\le C_0^{m+1} \;.
\]
for all $m>0$.
Indeed one can divide both sides by $C_0^{m+1}$,
and see that it suffices to show that
for sufficiently large $C_0$
one has
$
\frac{1}{C_0} + \frac{3A}{C_0-1} \le 0.9
$
or $\frac{1+3A}{C_0-1}\le 0.9$.
With this choice of $C_0$
the exponential bound  $|E_{-m-1}| \le C_0^{m+1}$ follows.

The case $k=m+1$ can be shown in the same way. Therefore the inductive proof
is complete.
\end{proof}

The following lemma will be useful for the proofs of the Robin heat kernel estimates on the bounded intervals.
One can immediately see that the statement of the following
lemma is true for the Neumann heat kernel
on the bounded interval with $I_k=1$ and $E_k=0$.

\begin{lemma} \label{lem:Robin-Dint-abs}
The discrete heat kernel $\bfp^R$
on $\{0,\cdots,N\}$ which satisfies
the Robin boundary condition
$\bfp^R_t(-1,y)=\mu_A \bfp^R_t (0,y)$
and $\bfp^R_t(N+1,y)=\mu_B \bfp^R_t (N,y)$ has the following representation:
\begin{equ} [e:Robin-Dint-abs]
\bfp^R_t (x,y) = p_t(x-y)+ \sum_{k\in\Z,k\neq 0} I_k p_t \big(x-\iota(y;k)\big)
+ \eps
 \sum_{k\in\Z,k\neq 0} \sum_{\bar y=k\bar N}^{(k+1)\bar N-1} p_t(x-\bar y)  E_k(\bar y,y)
\end{equ}
where $\iota$ is defined in \eqref{e:def-iota} and
 $I_k$ and $E_k$ satisfy the estimates in  Lemma~\ref{lem:image-phi-Dint}.
\end{lemma}
\begin{proof}
The heat kernel $\bfp^R$ is related with the standard continuous time discrete space random walk  heat kernel $p$ by
\[
\sum_{y=0}^{\bar N-1} \bfp^R_t (x,y) \phi(y)=
\sum_{y\in \Z} p_t (x,y) \phi(y)=
\sum_{k\in\Z} \sum_{y=k\bar N}^{(k+1)\bar N-1} p_t(x-y) \phi(y)\;.
\]
Note that $\phi$ on the middle and right hand sides of the above equation are the extension to all of $\Z$ of $\phi$ on $\Lambda$ as described earlier in this subsection. Applying Lemma~\ref{lem:image-phi-Dint}, the above quantity is equal to
\begin{equs}
\sum_{k\neq 0} & \!\! \sum_{y=k\bar N}^{(k+1)\bar N-1}
p_t(x-y)
	\Big( I_k \phi(y^\star) +  \eps \sum_{z=0}^{\bar N-1}  E_k (y,z) \phi(z) \Big)
+  \sum_{y=0}^{\bar N-1} p_t(x-y)  \phi(y) \\
& \!\!\!\!\!\! =\sum_{y^\star=0}^{\bar N-1}
	\Big( \sum_{k\neq 0} I_k p_t (x-\iota(y^\star;k)) \Big)
	\phi(y^\star)
+  \sum_{z=0}^{\bar N-1} \Big( \eps
 \sum_{k\neq 0} \sum_{y=k\bar N}^{(k+1)\bar N-1}
 	p_t(x-y)  E_k(y,z)
\Big) \phi(z) \\
&\qquad +  \sum_{y=0}^{\bar N-1} p_t(x-y)  \phi(y)\;.
\end{equs}
Therefore $\bfp^R$ is given by \eqref{e:Robin-Dint-abs}.
\end{proof}

\subsubsection{Sturm-Liouville theory}

We need some results on the spectrum  of the discrete operator
 $-\frac12 \Delta$ on the discrete interval $\{0,\cdots,N\} \subset \Z$ with Robin boundary condition
\[
u(-1)- \mu_A u(0)=0 \;,
\qquad
u(N+1)  -\mu_B u(N)=0 \;.
\]

We start by recalling some summation by parts formulas. We only need the formulas
on the bounded intervals here in this subsection, but the formulas for the case of the half line will also be used later. 
For the case of the half line, we define 
\[
\mathbf h_+ \eqdef \{u:\{-1,0\}\cup \Z_+ \to \R \,|\, \lim_{x\to \infty } u(x)=0, 
	\langle u,u\rangle_{\mathbf h_+}<\infty\}
\]
where $\langle u,v\rangle_{\mathbf h_+} \eqdef \sum_{x=-1}^\infty (\nabla^+\! u(x))(\nabla^+\! v(x))$.
(Recall that $\nabla^\pm$ are defined in \eqref{e:FB-der}.)
The following summation by parts formulas are facts of finite difference Laplacian $\Delta$
in general and has nothing to do with the particular type of boundary conditions;
the proof is omitted since it is straightforward.

\begin{lemma}
For functions $u,v$ on $\{-1,0,\cdots, N+1\}$, one has
\begin{equs}
\sum_{x=0}^N  u(x)\Delta v(x)
&= u(N+1) \nabla^+\! v(N)
	+ u(-1) \nabla^- \!v(-1) \\
	&\qquad\qquad\qquad\qquad -\!\!\! \sum_{x=-1}^N (\nabla^+\! u(x))(\nabla^+\! v(x)) , \qquad \label{e:sum-by-parts0}
\\
\sum_{x=0}^N  u(x)\Delta v(x)  &=
\sum_{x=0}^N  v(x)\Delta u(x)
+ u(N+1) \nabla^+ v(N)
+ u(-1) \nabla^- v(0)\\
&\qquad \qquad \qquad \quad -v(N+1) \nabla^+u(N) - v(-1) \nabla^- u(0) \;. \label{e:sum-by-parts1} 
 \end{equs}
For functions $u,v \in \mathbf h_+ $, one has
\begin{equ} 
\sum_{x=0}^\infty  u(x)\Delta v(x) =
\sum_{x=0}^\infty  v(x)\Delta u(x)
+ u(-1) \nabla^- v(0)
- v(-1) \nabla^- u(0) \;. \label{e:sum-by-parts2}
 \end{equ}
\end{lemma}

First of all, we observe  that
if $\mu_A,\mu_B \le 1$, the eigenvalues of the matrix $-\frac12 \Delta$ with Robin boundary condition are all non-negative.
This is because $-\frac12 \Delta$  is a self-adjoint Markov generator of the elastic random walk,
but we can also see this more explicitly as follows. Note that if $u=(u(0),\cdots,u(N))$ is an eigenvector with eigenvalue $\lambda$ for the matrix $-\frac12 \Delta$ with Robin condition (that is a $(N+1)\times (N+1)$ matrix), then, using \eqref{e:sum-by-parts0},
together with the Robin condition, we have
\begin{equs}
2\lambda & = \frac{ \langle u, (-\Delta) u \rangle}{ |u|^2} 
= \frac{\sum_{x=0}^N  u(x) \big(-\Delta u(x) \big) }{\sum_{x=0}^N  u^2(x)} \\
& = \frac{- u(N+1)^2 (1-\mu_B^{-1}) + u(-1)^2 (\mu_A^{-1}-1) +\sum_{x=-1}^N (\nabla^+ u(x))^2}{\sum_{x=0}^N  u^2(x) }
 \ge 0 \;.
\end{equs}
Here, $\langle u, (-\Delta) u \rangle$ is a bilinear form on $\R^{N+1}$,
while in $\sum_{x=0}^N  u(x) \big(-\Delta u(x) \big)$ the $\Delta$ is the usual finite difference Laplacian as in \eqref{e:sum-by-parts0} where we need to
specify two extra values
$u(-1) \eqdef \mu_A u(0)$ and $u(N+1)  =\mu_B u(N)$ using  the Robin condition.

%
%
%
Let $0\le \lambda_0 \le \lambda_1 \le \cdots \le \lambda_N$
be the eigenvalues of $-\frac12 \Delta$ with the above boundary condition.
Let $\psi_k$ be
the k-th normalized (i.e. $\sum_{x=0}^N \psi_k(x)^2 =1$) eigenfunction associated with $\lambda_k$:
\begin{equ} [e:eVec-k]
\psi_k(x)=C_{1,k}  \cos(\omega_k x) + C_{2,k} \sin(\omega_k x)
\end{equ}
Since they satisfy the Robin boundary conditions we have
\begin{equs} [e:pre-equa-kn]
C_{1,k} \cos (-\om_k) + C_{2,k} \sin (-\om_k)  &= \mu_A C_{1,k} \;, \\
C_{1,k} \cos (\om_k (N+1)) + C_{2,k} \sin (\om_k (N+1))
& = \mu_B ( C_{1,k} \cos (\om_k N) + C_{2,k} \sin (\om_k N)) \;.
\end{equs}
From these we cancel out $C_{1,k},C_{2,k}$ and we get, after simplification,
 the equation for $\om_k$:
 \begin{equ} [e:equa-kn]
\sin(\om_k(N+2)) - (\mu_A + \mu_B) \sin(\om_k(N+1)) + \mu_A \mu_B \sin(\om_k N)
=0 \;.
\end{equ}
%
%
Note that $\om=0$ and $\om=\pi$ are always solutions to \eqref{e:equa-kn}.
If $\mu_A=\mu_B =1$, they correspond to constant eigenvector.
But if $(\mu_A,\mu_B) \neq (1,1)$, $\om=0$ and $\om=\pi$ do not correspond to nontrivial eigenvectors since $C_{1,k}$ in \eqref{e:pre-equa-kn} must then be zero
and therefore $\psi=0$.

Let us record here the spectral decomposition of the heat kernel
\begin{equation}\label{eq:specdec}
\bfp^R_t(x,y)
=
\sum_{k=0}^N \psi_{\lambda_k} (x)\psi_{\lambda_k} (y)
e^{-t\lambda_k }\;.
\end{equation}

\begin{lemma} \label{lem:SL-Dint}
Under the above setting, if $\mu_A=\mu_B =1$, we have $\om_k = \frac{k\pi}{N+1} $, and
$\lambda_k = 1-\cos( \frac{k\pi}{N+1} )$.
If $\mu_A+\mu_B < 2$
we have
\[
\lambda_k = 1-\cos(\om_k) \qquad (k=0,\cdots,N)
\]
where $\om_k$ for $k=0,\cdots,N$ are the $N+1$ solutions of \eqref{e:equa-kn} in $(0,\pi)$
and are ordered as $\om_0 \le \om_1 \le \cdots \le \om_N$.
Furthermore, we have
\begin{equ} [e:est-kn]
	\frac{k\pi}{N+1}  \le \om_k  \le \frac{(k+1)\pi}{N+1} \;,
\end{equ}
for every $k=0,\cdots,N$.
\end{lemma}

\begin{proof}
For the case $\mu_A=\mu_B =1$, one can directly check that $\om_k = \frac{k\pi}{N+1} $ solve \eqref{e:equa-kn}. So we focus on the case $\mu_A+\mu_B < 2$.

We compute the derivative of the LHS of \eqref{e:equa-kn}
w.r.t. $\om_k$ at $\om_k=0$, which is equal to
\begin{equs}
( & N+2) - (\mu_A + \mu_B) (N+1) + \mu_A \mu_B N  \\
&\geq (N+2) - (2-\eps A-\eps B) (N+1) +  (1-\eps A-\eps B) N \\
&=\eps A+ \eps B \ge 0
\end{equs}
so the LHS of \eqref{e:equa-kn} is increasing at $0$.
Then it is easy to prove that
the LHS of \eqref{e:equa-kn}
at $\frac{k\pi}{N+1}$ is negative if $k$ is odd and positive if $k$ is even, since at these values
 the LHS of \eqref{e:equa-kn} becomes
 \[
 \sin\big(k\pi + \frac{k\pi}{N+1}\big)  + \mu_A \mu_B \sin\big(k\pi - \frac{k\pi}{N+1}\big)
 \]
and the first term always dominates and determines the sign if $\mu_A \mu_B<1$.
The claim \eqref{e:est-kn} follows immediately by continuity of the LHS of \eqref{e:equa-kn} in $\om_k$.
\end{proof}

\begin{lemma} \label{lem:e-fun-bnd}
Under the above setting, there exists a constant $C>0$
such that for all sufficiently large $N$, $0\le k \le N$, $0\le x \le N$, one has $|\psi_k(x)| \le \frac{C}{\sqrt N}$.
\end{lemma}

\begin{proof}
By elementary trigonometric identity we rewrite  $\psi_k$ as $\psi_k(x)=C_k \sin(\om_k x+\theta_k)$.
When $k=0$, $\om_0 <\frac{\pi}{N+1}$ so it is easy to see that over the interval $\{0,\cdots,N\}$ (which is less than half of the period of $\psi_0$),
the $L^2$-normalized function $\psi_0$ is  bounded by $ \frac{C}{\sqrt N}$.
Since $\psi_k$ is normalized for every $k$,
\begin{equs} [e:1-is-equ]
1=C_{k}^2  \sum_{x=0}^N
\sin^2(\omega_k x+\theta_k)
=C_{k}^2  \sum_{x=0}^N
\frac12 \big(1-\cos(2\omega_k x+2\theta_k)\big)  \\
=C_{k}^2 \Big( \frac{N+1}{2} -
\frac{\sin((N+1)\om_k)\cos(N\om_k+2\theta_k)}{\sin\om_k} \Big)
\end{equs}
where we used a Lagrange trigonometric summation formula.
For $0<k\le N/2$, by Lemma~\ref{lem:SL-Dint} we know that
$\om_k\in [ \frac{k\pi}{N+1}, \frac{(k+1)\pi}{N+1}]$ and therefore for $N$ sufficiently large
$\sin\om_k \ge \mathbf 1_{k=1} \frac{8\om_k}{9} +\mathbf 1_{k>1} \frac{\om_k}{2}   \ge  \frac{8\pi}{9(N+1)}$, so the above quantity \eqref{e:1-is-equ} is larger than
$C_{k}^2 \Big( \frac{N+1}{2} -
\frac{9(N+1)}{8\pi}\Big)$ and therefore $|C_k|$ and thus $|\psi_k|$ is bounded by $ \frac{C}{\sqrt N}$. If $N/2<k <N$, we can write $\sin\om_k = \sin(\pi-\om_k)$ and then follow the same argument.
Finally, when $k=N$, if $\mu_A=\mu_B =1$ then $\om_N = \frac{N\pi}{N+1} $, and otherwise one can check that the LHS of \eqref{e:equa-kn}
at $w=\frac{N\pi}{N+1}$ and $w=\frac{(3N+1)\pi}{3N+3}$ have opposite signs. This implies that $\pi-\om_N> \frac{2\pi}{3N+3}$. For $N$ sufficiently large one has
$ \sin(\pi-\om_N) \ge \frac89 (\pi-\om_N) \ge \frac{16\pi}{9(3N+3)}$
so \eqref{e:1-is-equ} is larger than $C_{k}^2 \Big( \frac{N+1}{2} -
\frac{27(N+1)}{16\pi}\Big)$ and therefore $|C_k|$ and thus $|\psi_k|$ is bounded by $ \frac{C}{\sqrt N}$.
\end{proof}

%
%

\subsubsection{Heat kernel estimates}

With the techniques developed above, we now prove the various heat kernel estimates
which will be useful for the rest of the paper. Proposition~\ref{prop:heat:ker}
and Corollary~\ref{cor:sum-weight} below are the estimates for the $\mu_A$-Robin heat kernel on $\Z_{\ge 0}$, while Proposition~\ref{prop:kerB}
and Corollary~\ref{cor:sum-int} below are the estimates for the $(\mu_A,\mu_B)$-Robin heat kernel on $\{0,\cdots,N\}$. The analogous estimates for the standard heat kernels on the entire $\Z$ are known, for instance \cite{BG,DemboTsai}.

\begin{proposition}\label{prop:heat:ker}
Assume that $\bfp^R$ is the heat kernel on $\Z_{\ge 0}$
with Robin boundary condition \eqref{e:DBleft}
with $\mu=1-\eps A$ for a constant $A>0$.
Given any $b\ge 0$,
for any $|n|\le \lceil t^{1/2} \rceil$, $v\in[0,1]$, $0\le t<t'< \eps^{-2}\bar T$, and $x,y\in\Z_{\ge 0}$,
we have
\begin{align}
	\label{eq:p:esti:sup}
	&
	\bfp^R_{t}(x,y) \leq e^{C(t'-t)} \bfp^R_{t'}(x,y),
\\
	\label{eq:p:holder:time:esti}
	&
	|\bfp^R_{t'}(x,y) - \bfp^R_t(x,y)|
	\leq  C(A)\,
	 (1\wedge t^{-\frac12-v})\,(t'-t)^v ,
\\
	\label{eq:p:esti}
	&
	\bfp^R_t(x,y)
	\leq
	 C(A,b,\bar T) \,(1\wedge t^{-\frac12}) \,e^{-b |x-y|(1\wedge t^{-1/2})},
\\
	\label{eq:del:p:esti}
	&
	|\nabla_n\bfp^R_t(x,y)|
	\leq
	C(A,b,\bar T) \,(1\wedge t^{-\frac{1+v}{2}})\, |n|^v \,e^{-b|x-y|(1\wedge t^{-1/2})},
\end{align}
Here $\nabla_n f(x)\eqdef f(x+n) -f(x)$ acts on the first variable of the functions.
\end{proposition}

\begin{proof}
To prove \eqref{eq:p:esti:sup}, we use
\eqref{e:discHKrep} to get
\[
\bfp_t^R(x,y) \le  e^{-t'} e^{t'-t} \sum_{n=0}^\infty \frac{t'^{\,n}}{n!} p_n^e(x,y)
=e^{t'-t} \bfp_{t'}^R(x,y) \;,
\]
where $p_n^e(x,y)$ is the transition probability of the discrete time elastic random walk.

For the other bounds, we use the explicit formula obtained in Lemma~\ref{lem:bfp-H}:
\begin{equ} [e:explicit-to-use]
\bfp^R_t(x,y)
=
p_t(x-y)+\mu_A p_t(x+y+1)
	+(\mu_A^2-1) \sum_{z=-\infty}^{-2}
		 p_t(x+y-z) \mu_A^{-z-2}
\end{equ}
and the existing estimates for the entire line kernel $p_t(x)$.

To prove \eqref{eq:p:holder:time:esti},
%
%
we make use of (\cite[(A.10)]{DemboTsai}):
\begin{equ} [e:standard-tdiff]
|p_{t'}(x) - p_t(x)|
	\leq  C
	 (1\wedge t^{-\frac12-v})\,(t'-t)^v ,
\end{equ}
and this together with \eqref{e:explicit-to-use} yields
\begin{equ}
|\bfp^R_{t'}(x,y) - \bfp^R_t(x,y)|
\le C
	 (1\wedge t^{-\frac12-v})\,(t'-t)^v \Big(1+\mu_A+(\mu_A^2-1) \sum_{z=-\infty}^{-2}
		 \mu_A^{-z-2} \Big) \;.
\end{equ}
The last factor is bounded by a constant independent of $\eps$. Indeed this is obvious if $A=0$ (i.e. $\mu=1$); and if $A>0$, the sum over $z$ yields a factor
$\frac{1}{1-\mu}$ 
which multiplied by $\mu_A^2-1$ is bounded by a constant independent of $\eps$.

To prove \eqref{eq:p:esti}, we use (\cite[(A.12)]{DemboTsai}):
\begin{equ} [e:stand-p-bnd]
p_t(x)
	\leq
	 C(b) (1\wedge t^{-\frac12})\, e^{-b |x|(1\wedge t^{-1/2})} \;.
\end{equ}
Note that by $e^{-b (x+y)(1\wedge t^{-1/2})} \le e^{-b |x-y|(1\wedge t^{-1/2})} $
the first two terms in \eqref{e:explicit-to-use} satisfy the desired bound.
Since $\mu_A\le 1$
and $x,y\in\Z_{\ge 0}$,
\begin{equs}
 \sum_{z=-\infty}^{-2}  &
		 p_t(x+y-z) \mu_A^{-z-2}
\le C(b)  \sum_{z=-\infty}^{-2}
(1\wedge t^{-\frac12}) \,e^{-b (x+y-z)(1\wedge t^{-1/2})} \\
&\le C(b)  (1\wedge t^{-\frac12})\,
 e^{-b (x+y)(1\wedge t^{-1/2})}
 \frac{e^{-2b(1\wedge t^{-\frac12})}}{1-e^{-b(1\wedge t^{-\frac12})}}
\le C(b)\,
 e^{-b (x+y)(1\wedge t^{-1/2})}
\end{equs}
where we summed over $z$ and used
\begin{equ} [e:elem-eq1q]
e^{-q}/(1-e^{-q}) \le 1/q
\end{equ}
 for any $q\ge 0$.
By the assumption on $\mu_A$ and using $t\le \eps^{-2}\bar T$, we have
\[
|\mu_A^2-1| \le C(A)\, \eps \le C(A,\bar T)\, (1\wedge t^{-\frac12}) .
\]
So the last term of \eqref{e:explicit-to-use} also satisfies the desired bound and
thus
we obtain  \eqref{eq:p:esti}.

To prove \eqref{eq:del:p:esti}, we can use (\cite[(A.13)]{DemboTsai}):
\begin{equ} [e:standard-xdiff]
|\nabla_n p_t(x)|
	\leq
	C(b)\, (1\wedge t^{-\frac{1+v}{2}}) \,|n|^v \,e^{-b|x|(1\wedge t^{-1/2})},
\end{equ}
and proceed in the same way as the proof for \eqref{eq:p:esti} to obtain \eqref{eq:del:p:esti}.
\end{proof}

\begin{corollary} \label{cor:sum-weight}
Let $\bfp^R$ be the heat kernel on $\Z_{\ge 0}$ as above.
Given any $a\ge 0$,
for any  
$t\in[0,\bar T\eps^{-2}]$, $x\in\Z_{\ge 0}$,
we have
\begin{equs}
	\sum_{y \ge 0}     \bfp^R_t(x,y)
		\,e^{a\eps y} \, e^{a|x-y| (1\wedge t^{-\frac12})}
	&\leq C(a,A,\bar T) \,e^{a\eps x} \;, \label{eq:p:esti:sum} \\
	\sum_{y \ge 0}     \nabla_x \bfp^R_t(x,y)
		\,e^{a\eps y} \, e^{a|x-y| (1\wedge t^{-\frac12})}
	&\leq C(a,A,\bar T)\, e^{a\eps x} \, t^{-\frac12} \;. \label{eq:dp:esti:sum}
\end{equs}
\end{corollary}

\begin{proof}
To prove \eqref{eq:p:esti:sum}, note that its LHS is bounded above by
\begin{equ} [e:eq:p:esti:sumtp]
e^{a\eps x} \sum_{y\ge 0} \bfp^R_t(x,y) \, e^{a |x-y| (\eps +  (1\wedge t^{-\frac12}))} \;,
\end{equ}
and applying   \eqref{eq:p:esti} this is further bounded above by
\[
C(A,b,\bar T)\, e^{a\eps x} \sum_{z\in \Z}
 (1\wedge t^{-\frac12}) \,e^{-b |z|(1\wedge t^{-1/2})}
\, e^{a |z| (\eps +  (1\wedge t^{-\frac12}))}
\]
for any $b>0$.
Without loss of generality we assume $\eps^2 < \bar T$,
and since $t \le \eps^{-2} \bar T$, one has $\bar T^{\frac12} (1\wedge t^{-\frac12}) > \eps$. Choosing $b=2 a (\bar T^{\frac12}+1)$,
the above quantity is bounded  above by
\[
C(A,a,\bar T) \,e^{a\eps x} \sum_{z\in \Z}
 (1\wedge t^{-\frac12}) \, e^{-\frac{b}{2} |z|(1\wedge t^{-1/2})}
 \;.
\]
 \eqref{eq:p:esti:sum} then follows by performing the sum over $z$ and applying \eqref{e:elem-eq1q}.

The estimate \eqref{eq:dp:esti:sum} follows in the same way using \eqref{eq:del:p:esti}
in place of  \eqref{eq:p:esti}.
\end{proof}

\begin{proposition}\label{prop:kerB}
Assume that $\bfp^R$ is the heat kernel on $\{0,1,\cdots,N\}$
with Robin boundary condition
\[
\bfp^R_t(-1,y)=\mu_A \bfp^R_t(0,y)\;, \quad
\bfp^R_t(N+1,y)=\mu_B \bfp^R_t(N,y)\quad \forall 0\le t\le \eps^{-2}\bar T, 0\le y\le N
\]
with $\mu_A=1-\eps A$ and $\mu_B=1-\eps B$  for  constants $A>0,B>0$.
Given any $b\ge 0$,
for any $|n|\le \lceil t^{1/2} \rceil$, $v\in[0,1]$, $0\le t<t'< \eps^{-2}\bar T$,
and $0\le x,y \le N$,
we have all the bounds stated  in Proposition~\ref{prop:heat:ker} where the constants depend on $A$ and $B$ now.
\end{proposition}

\begin{proof}
The proof for \eqref{eq:p:esti:sup} follows in the same way as the half line case.

Turning to \eqref{eq:p:holder:time:esti}, unlike in the half line case, it does not seem to be easy to apply the standard heat kernel estimate \eqref{e:standard-tdiff}
combined with \eqref{e:Robin-Dint-abs} to reach the conclusion. This is because \eqref{e:Robin-Dint-abs} involves summation over all periods while \eqref{e:standard-tdiff} does not capture any spatial decay. Instead of trying to improve upon estimates to follow this route, we observe below that \eqref{eq:p:holder:time:esti} can be seen as a consequence of spectral properties of $-\frac12 \Delta$ as contained in Lemma~\ref{lem:SL-Dint}.

First note that by \eqref{eq:specdec}
\[
\bfp^R_t(x,y)-
\bfp^R_{t'}(x,y)
=
\sum_{k=1}^N \psi_{\lambda_k} (x)\psi_{\lambda_k} (y)
(e^{-t\lambda_k } - e^{-t'\lambda_k } ) \;.
\]
By Lemma~\ref{lem:SL-Dint} there exist constants $C_1,C_1',C_2,C_2'$ such that
\begin{equ} [e:upperlower-lambda]
C_1' \frac{k^2}{N^2}
\le  C_1 w_k^2 \le  \lambda_k = 1-\cos(\om_k) \le C_2 w_k^2
 \le C_2' \frac{k^2}{N^2} \;.
\end{equ}
Using the upper bound in \eqref{e:upperlower-lambda}, one has
\[
|1- e^{(t-t')\lambda_k } |  \le C |(t-t')\lambda_{k})|^v \le  C |t-t'|^v  k^{2v}/N^{2v}
\]
for $v\in[0,1]$ and $\om_k \le \pi$.
By Lemma~\ref{lem:e-fun-bnd} the
eigenfunctions $\psi_{\lambda_k}$ can be
bounded by constant times $\sqrt{N}$, thus by the lower bound in \eqref{e:upperlower-lambda}
\begin{equs}
|\bfp^R_t(x,y)-
\bfp^R_{t'}(x,y)|
& \le  \frac{C}{N} \sum_{k=1}^N e^{-t\lambda_k }  |t-t'|^v (1\wedge \frac{k^{2v}}{N^{2v}}) \\
& \le   \frac{C}{N} \sum_{k=1}^N e^{-C_1' t k^2 /N^2}  |t-t'|^v (1\wedge \frac{k^{2v}}{N^{2v}}) \;.
\end{equs}
Summing over $k$ we obtain the desired bound \eqref{eq:p:holder:time:esti}.

To prove \eqref{eq:p:esti},
we invoke Lemma~\ref{lem:Robin-Dint-abs},
and use the standard heat kernel bound \eqref{e:stand-p-bnd}.
Since the terms with $|k|\le 2$ can be all bounded using \eqref{e:stand-p-bnd},
one only needs to deal with the terms with $|k| > 2$.
The second term on the RHS of \eqref{e:Robin-Dint-abs} can be bounded by
\begin{equs}
C(\tilde b) & (1\wedge t^{-\frac12}) \sum_{k > 2}
	I_k  \; e^{-\tilde b |x-\iota(y;k)|(1\wedge t^{-1/2})} \\
&\le C(\tilde b)(1\wedge t^{-\frac12}) \, e^{-\tilde b N(1\wedge t^{-1/2})}
	  \cdot \frac{e^{-2 \tilde b N(1\wedge t^{-1/2})}}{1-e^{- \tilde b N(1\wedge t^{-1/2})}}
\end{equs}
since $|I_k|\le 1$. Using $e^{-2q}/(1-e^{-q}) \le 1/q$ for any $q\ge 0$,
the above expression is bounded by
\[
 C(\tilde b) \, (1\wedge t^{-\frac12}) \, \frac{e^{-\tilde b N(1\wedge t^{-1/2})}}{\tilde b N (1\wedge t^{-1/2})}
	  \le C(\bar T,\tilde b) \,(1\wedge t^{-\frac12}) \,e^{-\tilde b |x-y|(1\wedge t^{-1/2})}
\]
where we used
$t\le \eps^{-2} \bar T$ and $\eps =1/N$.

The last term on the RHS of \eqref{e:Robin-Dint-abs} 
can be bounded by, using  \eqref{e:stand-p-bnd} and the the bound for $E_k$ in Lemma~\ref{lem:image-phi-Dint},
\begin{equs}
\eps C(\tilde b)&
 \sum_{k >2} \sum_{\bar y=k \bar N}^{(k+1) \bar N-1} (1\wedge t^{-\frac12}) \,
 e^{-\tilde b |x-\bar y|(1\wedge t^{-1/2})}  \, C_0^{|k|} \\
& \le
 C(\tilde b,C_0) \,(1\wedge t^{-\frac12}) \, e^{-(\tilde b N(1\wedge t^{-1/2}) }
 \sum_{k >0}  e^{(-\tilde b N(1\wedge t^{-1/2}) +\log C_0)|k| }
\end{equs}
where we used $|x-\bar y| \ge (|k|-1)N$ and replaced the sum over $\bar y$
by a factor $N=1/\eps$.
As in the proof of Corollary~\ref{cor:sum-weight},
we have $N\bar T^{\frac12} (1\wedge t^{-\frac12}) > 1$.
Choosing $\tilde b = 2\log C_0 T^{\frac12}$,
and summing over $k$,
the above quantity is bounded by (noting that $|x-y|\le N$)
\begin{equ}
 C(A,B,b,\bar T) \,(1\wedge t^{-\frac12}) \,e^{-b |x-y|(1\wedge t^{-1/2})} \;.
\end{equ}
The proof of \eqref{eq:p:esti} is then completed.

The proof of \eqref{eq:del:p:esti} follows analogously,
using the standard kernel estimates \eqref{e:standard-xdiff} together with \eqref{e:Robin-Dint-abs}.
\end{proof}

\begin{corollary} \label{cor:sum-int}
Let $\bfp^R$ be the heat kernel on $\{0,1,\cdots,N\}$
with Robin boundary condition
 as above.
For any  
$t\in[0,\bar T\eps^{-2}]$, $x\in \{0,1,\cdots,N\}$,
we have
\begin{equs}
\sum_{y =0}^N     \bfp^R_t(x,y)
	&\leq C(A,B,\bar T)  \;,\\
	\label{eq:dp-int-sum}
	\sum_{y =0}^N     \nabla_x \bfp^R_t(x,y)
		\,e^{a|x-y| (1\wedge t^{-\frac12})}
	&\leq C(A,B,\bar T) \,t^{-\frac12} \;.
\end{equs}
\end{corollary}

\begin{proof}
These are  consequences of \eqref{eq:p:esti}, \eqref{eq:del:p:esti},
and the proof is analogous with that of Proposition~\ref{cor:sum-weight}.
\end{proof}

\subsection{Proof of tightness}
The rest of this section is devoted to the proofs of Propositions~\ref{prop:Holder-B} and \ref{prop:Holder-H}, which will yield tightness of the rescaled
process. Recall that $\|f_t(x)\|_n \eqdef (\E|f_t(x)|^n)^{\frac1n}$ denotes the $L^n$-norm. Recall also the definitions given in Section \ref{Sec2}. The following two proposition provide vital continuity estimates from which the tightness follows by standard arguments.

\begin{proposition} \label{prop:Holder-H}
Let $Z$ be the G{\"a}rtner transformed process
for ASEP-H.
Fix $ \bar T <\infty $, $ n \in \N $, $ \alpha\in(0,1/2),$
and some near equilibrium initial conditions as in Assumption~\ref{def:nearEq},
with the corresponding finite constant $ a $.
Then, there exists some finite constant $ C $ such that
\begin{equs}
	\label{e:uniform-H}
	\Vert Z_t(x) \Vert_{2n} &\leq C e^{a\eps x}
\\
	\label{e:Holderx-H}
	\Vert Z_t(x)-Z_t(x')\Vert_{2n} &\le
	C (\eps |x-x'|)^\alpha e^{a\eps (x+x')}
\\
	\label{e:Holdert-H}
	\Vert Z_t(x) - Z_{t'}(x) \Vert_{2n} &\le
	 C \eps^\alpha(1\vee|t'-t|^\frac{\alpha}{2})  e^{2a\eps x}
\end{equs}
for all $t,t'\in[0,\eps^{-2} \bar T]$ and $ x,x'\in\R_+$ with $|x-x'|\le \eps^{-1}$.
\end{proposition}

\begin{proposition} \label{prop:Holder-B}
Let $Z$ be the G{\"a}rtner transformed process
for ASEP-B.
Fix $ \bar T <\infty $, $ n \in \N $, $ \alpha\in(0,1/2),$ and some initial conditions as in Assumption~\ref{def:icB}.
Then, there exists  finite constant $ C $ such that
\begin{equs}
	\label{e:uniform-B}
	\Vert Z_t(x) \Vert_{2n} &\leq C
\\
	\label{e:Holderx-B}
	\Vert Z_t(x)-Z_t(x')\Vert_{2n} &\le
	C (\eps |x-x'|)^\alpha
\\
	\label{e:Holdert-B}
	\Vert Z_t(x) - Z_{t'}(x) \Vert_{2n} &\le
	 C \eps^\alpha (1\vee|t'-t|^\frac{\alpha}{2})
\end{equs}
for all $t,t'\in[0,\eps^{-2} \bar T]$ and $ x,x'\in [0,N] $.
\end{proposition}

Recall the scaled process $\mathcal Z^\eps$ from Definition \ref{def:scaled}.
\begin{proposition}\label{prop:tight}
For ASEP-H with near equilibrium initial conditions as in Assumption~\ref{def:nearEq},
the law of $ \{\mathcal Z^\eps\}_\eps $
is tight in $ D([0,\bar T] ; C(\R_+) )$ and
any limit point of $ \{\mathcal Z^\eps\}_\eps $ is in $ C([0,\bar T] ; C(\R_+)) $.
For ASEP-B with boundary conditions as in Assumption~\ref{def:icB}
the law of $ \{\mathcal Z^\eps\}_\eps $
is tight in $ D([0,\bar T] ; C([0,1]))$
any limit point of $ \{\mathcal Z^\eps\}_\eps $ is in $ C([0,\bar T] ; C([0,1])) $.
\end{proposition}

\begin{proof}
Let $I=\R_+$ for ASEP-H and $I=[0,1]$ for ASEP-B.
First of all, the bounds  or \eqref{e:uniform-H} and \eqref{e:Holderx-H}, or \eqref{e:uniform-B} and \eqref{e:Holderx-B},
imply tightness in $C(I)$, see for instance
\cite{MR1700749}.
Then \eqref{e:Holdert-H} or \eqref{e:Holdert-B}
and
\cite[Proposition~4.9]{BG} (and the arguments below that proposition in \cite{BG} with $C(\R)$ replaced by $C(I)$) combine to imply tightness in $D([0,\bar T] ; C(I))$ and that
the limiting points must lie in $ C([0,\bar T]; C(I)) $. The only necessary change in the   arguments  \cite[Proposition~4.9]{BG} is that the metric on $C(\R_+)$
is defined as $\rho(f,g)= \sum_{n=1}^\infty \frac{1}{2^n} \big( 1\wedge \max_{x\in [0,n]} |f(x)-g(x)| \big)$, and for $C([0,1])$ we can simply use the $C^\infty$ norm.
\end{proof}

The following technical result is  useful for $L^n$ estimates.

\begin{lemma}\label{lem:BDG}
Let $\Lambda\in\{\Z_{\ge 0}, \{0,\cdots,N\}\}$.
Given any $ n\in\N $,
there exists a finite constant $ C $ such that,
for any deterministic function $f_s(x,x')$: $[0,\infty)\times\Lambda^2\to\R$
and any $ t\leq t' \in [0,\infty) $ with $ t'-t \geq 1 $,
\begin{equ} [e:BDG]
	\Big\Vert
	\int_t^{t'} \sum_{x'\in \Lambda} f_s (x,x')  dM_s (x') \Big\Vert_{2n}^2
\le
	C \eps \int^{ t' }_{ t } \sum_{x' \in \Lambda}
	 \bar f_{s} (x,x')^{\;2}  \Vert Z_s^2 \Vert_n (x') \,ds
\end{equ}
where
\begin{equ} [e:BDG-ave]
\bar f_{s}(x,x')\eqdef \sup_{|s'-s| \leq 1 } | f_{s'}(x,x')|
\end{equ}
and $M$ is the martingale introduced in Lemma~\ref{e:DSHE}.
\end{lemma}

\begin{proof}
This is essentially  \cite[Lemma~3.1]{DemboTsai}, see also \cite[Lemma~3.1]{CST2016asep}.
Fixing $t$ and calling $R_{t'}(x)$ the integral on the left hand side of \eqref{e:BDG}, by the \ac{BDG} inequality,
one only needs to bound $\Vert [ R_\Cdot(x) ]_{t'} \Vert_n$,
where $ [ R_\Cdot(x) ]_{t'} $ is the optional quadratic variation,
i.e. the sum of the squares of all the jumps of $R_\Cdot(x)$ over time $(t,t']$.
The only inputs from the martingale $M$ to the proof in   \cite[Lemma~3.1]{DemboTsai}
are that a jump of $M(x')$ at time $s$ equals $\big((q/p)^{\pm 1}-1\big) Z_{s^-}(x')^2 \le C\eps Z_{s^-}(x')^2$, and that
\begin{align}\label{eq:Z:ss}
	\sup_{s\in (s_1,s_2]} Z_s(x') \leq e^{2 \sqrt{\eps} N_I(x')} Z_{s_1}(x')
\end{align}
where $s_1,s_2\in I$ and $N_I(x')$ is
the number of jumps occurred at $x'$ during a unit time interval $I$ which is stochastically bounded by a Poisson random variable with constant rate -
these remain true in our case, and therefore the detailed proof is omitted.
\end{proof}

%
The rest of this section is devoted to the proofs of Propositions~\ref{prop:Holder-H} and \ref{prop:Holder-B}.
We rewrite the discrete SHE \eqref{e:DSHE}
in the following integrated form:
\begin{equ} \label{e:intDSHE}
	Z_t (x) = \sum_{x'\in \Lambda}\bfp^R_t (x,x') Z_0(x')
		 + \int_0^t \sum_{x'\in \Lambda} \bfp^R_{t-s} (x,x') dM_s(x') \;,
\end{equ}
with $\Lambda\in\{\Z_{\ge 0}, \{0,\cdots,N\}\}$ depending the half line or bounded interval case.

We start from the proof of Proposition~\ref{prop:Holder-H}.

\begin{proof}[Proof of Proposition~\ref{prop:Holder-H}]
In what follows, the value of constants may change from line to line (and within lines). Let $I_1$ and $I_2$ denote the first and second terms
on the RHS of \eqref{e:intDSHE}, respectively.

We begin by proving the {\it uniform bound} \eqref{e:uniform-H}.
First, by \eqref{eq:p:esti:sum} we have
the following bound on the Robin heat kernel
\begin{equ} \label{e:ev-exp}
	\sum_{y\ge 0} \bfp^R_t (x,y) e^{a\eps y}
	 \le C e^{a\eps x}
	 \qquad \mbox{for } t\le \eps^{-2} \bar T \;.
\end{equ}
For $I_1$, by the triangle inequality we have
\[
	\Vert I_1(t,x)^2\Vert_n = \Vert I_1(t,x)\Vert_{2n}^2
	\le \Big(  \sum_{x'}\bfp^R_t (x,x') \Vert Z_0(x')\Vert_{2n} \Big)^2.
\]
Combining this with \eqref{e:ev-exp} and \eqref{e:init-uniform}, we obtain
\begin{align}\label{eq:I1}
	\Vert I_1(t,x)^2 \Vert_{n} \leq C e^{2a\eps x}.
\end{align}
Turning to bounding $ I_2 $,
we assume $ t \geq 1 $ and apply Lemma~\ref{lem:BDG}
with $ f_s(x,x')=\bfp^R_{t-s}(x,x') $ to obtain
\begin{equ}
	\Vert I_2(t,x)^2 \Vert_n
\le
	C \eps \int_0^t \sum_{y\ge 0}
	 \bar \bfp_{ t-s}^2 (x,y) \Vert Z_s^2 \Vert_n (y) ds
\end{equ}
where $\bar \bfp$ is the local supremum of $\bfp^R$ defined as in \eqref{e:BDG-ave}.
By \eqref{eq:p:esti:sup}
we have $\bfp^R_t \le C \bfp^R_{t'}$ for $0\le t'-t\le 1$ 
and by \eqref{eq:p:esti} we have the  estimate
$\bfp^R_t \le Ct^{-\frac12}$. Therefore $\bar \bfp_{t-s}^2(x,y) \leq C (t-s)^{-1/2} \bfp_{t-s}^R(x,y)$ and hence
\begin{equ}
	\Vert I_2(t,x)^2 \Vert_n
\le
	C \eps \int_0^t (t-s)^{-\frac12}
	\Big( \sum_{y\ge 0} \bfp^R_{ t-s }(x,y)\Vert Z_s^2 \Vert_n (y)\Big) ds\;,
	\qquad \text{ for } t \geq 1\;.
\end{equ}
Combining this with \eqref{eq:I1} yields
\begin{equ}\label{eq:iter}
	\Vert Z_t^2(x) \Vert_{n}
	\leq
	C e^{2a\eps x}
	+ C \eps \int_0^t (t-s)^{-\frac12} \Big(
		\sum_{y\ge 0} \bfp^R_{ t-s }(x,y)\Vert Z_s^2 \Vert_n (y)
	\Big) ds\;.
\end{equ}
The bound \eqref{eq:iter} was derived for $ t \geq 1 $,
but it in fact holds true also for $ t \leq 1 $.
This is so because, by \eqref{e:uniform-H} and \eqref{eq:Z:ss} with $ (s_1,s_2]=(0,t] $,
we already have $ \Vert Z^2_t(x) \Vert_{2n} \leq C e^{2a\eps x} $,
for $ t \leq 1 $.
With this, iterating this inequality,
using the semi-group property of the heat kernel $\bfp^R$  
and \eqref{e:ev-exp},
we then arrive at
\begin{equ}
	\Vert Z_t^2(x) \Vert_{n}
	\leq
	\Big(
		C e^{2a\eps x}
		+ \sum_{j=1}^\infty
		 \frac{C^j}{j!} \Big(
		\eps\int_0^t s^{-1/2} ds \Big)^j
		e^{2a\eps x}
	\Big) \;.
\end{equ}
With $t\le \eps^{-2}\bar T$, the desired result \eqref{e:uniform-H} follows.

Now we turn to proving the {\it spatial H\"older estimate} \eqref{e:Holderx-H}.
For this we extend $Z_0$ to the entire $\Z$ (we still denote it by $Z_0$), such that
$Z_0(x-1)-\mu_A Z_0(x)$ is odd in $x$:
\begin{equ} [e:HxI1-ext]
Z_0(x-1)-\mu_A Z_0(x) = - (Z_0(-x-1)-\mu_A Z_0(-x)) \qquad (x<0).
\end{equ}
As in Section~\ref{sec:techRobin}, this implies that
so that
\begin{equ} [e:p-conv-Z0]
I_1(t,x)= \sum_{y \ge 0}\bfp^R_t (x,y) Z_0(y) = \sum_{y\in\Z}  p_t(x-y)  Z_0(y)
= p_t * Z_0(x) \;.
\end{equ}
The advantage of the extension is that now $p_t(x-y)$ only depends on the difference between $x$ and $y$.
By triangular inequality, we have
\begin{equ} [e:p-conv-Z02]
 \Vert I_1(t,x) -I_1(t,x') \Vert_{2n}^2
	\le \Big( \sum_{\bar x \in \Z} p_t(\bar x)
		\Vert Z_0(x-\bar x)-Z_0(x'-\bar x)\Vert_{2n}  \Big)^2 \;.
\end{equ}
By \eqref{e:init-Holder} we have $\Vert Z_0(x)-Z_0(x')\Vert_n
	\le C (\eps|x-x'|)^\alpha e^{a\eps(x+x')}$ for $x,x'\in \Z_{\geq 0}$, where $\alpha\in(0,\frac12)$. This is also true for the extended $ Z_0$ with any $x,x'$.  To see this rewrite \eqref{e:HxI1-ext} as  (recall that $\mu=1-\eps A$)
\[
Z_0(x-1)- Z_0(x) = - (Z_0(-x-1)- Z_0(-x))-\eps A (Z_0(x)+Z_0(-x)) \qquad (x<0)\;.
\]
Then, for  $x<x'<0$, sum the above identity over the points between $x$ and $x'$ which yields
\begin{equ} [e:Z0-Z0imeps]
Z_0(x)- Z_0(x') = Z_0(-x-1)-Z_0(-x'-1)-\eps A \sum_{z=x'}^{x+1} (Z_0(z)+Z_0(-z)) \;.
\end{equ}
Therefore, by \eqref{e:Z0-Z0imeps}, \eqref{e:init-uniform} and \eqref{e:init-Holder},
\begin{equs}
\Vert Z_0(x)-Z_0(x')\Vert_n
	&  \le C (\eps|x-x'|)^\alpha e^{a\eps(x+x')} + C\eps|x-x'|  e^{a\eps(x+x')}  \\
& \le C (\eps|x-x'|)^\alpha e^{a\eps(x+x')}
\end{equs}
since $|x-x'|\le 1/\eps$ and $\alpha\in (0,\frac12)$.
For $x<0<x'$ we have the same bound by $|x|^\alpha+|x'|^\alpha \le C (x'-x)^\alpha$.
For the {\it standard} heat kernel $p$, $\sum_{\bar x\in \Z} p_t(\bar x) e^{a\eps \bar x}\le C $, thus the
RHS of \eqref{e:p-conv-Z02} is bounded by
\begin{equ} \label{e:I1Holderx}
 \Big( \sum_{\bar x\in \Z} p_t(\bar x) \,
		(\eps|x-x'|)^\alpha e^{a\eps(|x-\bar x|+|x'-\bar x|)}   \Big)^2
\le C
 (\eps|x-x'|)^{2\alpha} e^{2a\eps(x+x')} \;.
\end{equ}
For $\Vert I_2(t,x) -I_2(t,x') \Vert_{2n}^2$,
we apply Lemma~\ref{lem:BDG}
with $f_s(x,\bar x)=\bfp^R_{t-s}(x',\bar x) -\bfp^R_{t-s}(x,\bar x)$,
and use the already-proved result \eqref{e:uniform-H} to bound $\Vert Z_s^2 (x')\Vert_n $
by $e^{2a\eps x'}$.
We then
use the fact that
\begin{equs}
\big( \bfp^R_{t-\lfloor s\rfloor}(x',\bar x)  &  -\bfp^R_{t-\lfloor s\rfloor}(x,\bar x) \big)^2 \\
&   \le
\big| \bfp^R_{t-\lfloor s\rfloor}(x',\bar x) -\bfp^R_{t-\lfloor s\rfloor}(x,\bar x) \big| \,
\big( \bfp^R_{t-\lfloor s\rfloor}(x',\bar x) + \bfp^R_{t-\lfloor s\rfloor}(x,\bar x) \big)
\end{equs}
and for the first factor we apply the $L^\infty$ gradient estimate \eqref{eq:del:p:esti} which implies
\[
\big| \bfp^R_{t-\lfloor s\rfloor}(x',\bar x) -\bfp^R_{t-\lfloor s\rfloor}(x,\bar x) \big|
	\leq
	C (1\wedge (t-s)^{-\frac{1}{2}-\frac{\alpha}{2}}) |x-x'|^\alpha \;,
\]
and for the second factor we use the $L^1$ bound \eqref{e:ev-exp}
to integrate out the weight $e^{2a\eps x'}$.
This yields
\[
\Vert I_2(t,x) -I_2(t,x') \Vert_{2n}^2 \le C\eps |x-x'|^\alpha
	e^{2a\eps (x+x')}
	\int_0^t (t-s)^{-\frac12-\frac{\alpha}{2}} ds \;.
\]
Noting that $\int_0^t (t-s)^{-\frac12-\frac{\alpha}{2}} ds \le Ct^{\frac{1-\alpha}{2}}\le C(\bar T) \eps^{\alpha-1}$, we
arrive at the bound \eqref{e:Holderx-H}.

Next we prove the {\it temporal H\"older estimate} \eqref{e:Holdert-H}.
Without lost of generality, we assume $ t < t'-1 $.
For $I_1$, we first note that $I_1(t,x)\ge 0$ since both $\bfp^R$ and $Z_0$ are positive.
We then use the semi-group property  of $\bfp^R$  
and $\sum_{y} \bfp^R_{t'-t}(x,y)  \le 1$  (Lemma~\ref{lem:semigp})
to get
\begin{equs} [eq:I1:tt]
	I_1(t',x) - I_1(t,x)
	&=\sum_{z\ge 0} \bfp^R_{t'-t} (x,z) I_1(t,z) - I_1(t,x) \\
	&\le  \sum_{z\ge 0} \bfp^R_{t'-t}(x,z) \,\big(I_1(t,z)-I_1(t,x)\big).
\end{equs}
By the spatial H\"older estimate \eqref{e:Holderx-H}, we have
$
	\Vert I_1 (t,z)-I_1 (t,x) \Vert_{2n}
	\le C
	(\eps|x-z|)^\alpha e^{a\eps (x+z)}  
	 .
$
Also, since $|r|^\alpha \le C e^{|r|}$ for any $r$, one has
\begin{equs} [e:power-by-exp]
	\sum_{z\ge 0} & \, |x-z|^\alpha \, \bfp^R_{t'-t}(x,z) \,e^{a\eps (x+z)} \\
	&\le C \sum_{z\ge 0} e^{a |x-z|(1\wedge (t'-t)^{-\frac12})}
		 (1\vee (t'-t)^{\frac{\alpha}{2}}) \, \bfp^R_{t'-t}(x,z)\, e^{a\eps (x+z)} \\
	&\leq
	C e^{2a \eps x} (1 \vee |t'-t|^{\frac{\alpha}{2}})
\end{equs}
where we applied Corollary~\ref{cor:sum-weight} in the last inequality.
Therefore one obtains the desired bound $ \Vert I_1 (t,x)-I_1 (t,x) \Vert_{2n}\leq C\e^{\alpha}|t'-t|^{\alpha/2} e^{a\e |x|}$.

Regarding the term $I_2$, we write $I_2(t',x)-I_2(t,x) = J_1(t,t',x)+J_2(t,t',x)$ where
\begin{equs}
J_1(t,t',x) & =\int_t^{t'} \sum_{y\ge 0}  \bfp^R_{t'-s} (x,y) dM_s(y)\,, \\
\quad J_2(t,t',x) & =\int_0^t \sum_{y\ge 0} (\bfp^R_{t'-s}-\bfp^R_{t-s}) (x,y) dM_s(y)\,.
\end{equs}
For the term $ J_1 $, applying Lemma~\ref{lem:BDG}, then proceeding
as above with the uniform bound   \eqref{e:uniform-H}  on $Z$,
 the bound \eqref{eq:p:esti} and Corollary~\ref{cor:sum-weight},
we obtain
\begin{align*}
	\Vert J_1(t,t',x)^2 \Vert_{n}
	\leq
	C ( \eps^\alpha |t'-t|^\frac{\alpha}{2} e^{a\eps x}  )^2 \;.
\end{align*}
As for $ J_2 $, applying  Lemma \ref{lem:BDG} using
$
	(\bfp^R_{t'-s}-\bfp^R_{t-s})^2 \le |\bfp^R_{t'-s}-\bfp^R_{t-s}|\, (\bfp^R_{t'-s}+\bfp^R_{t-s})
$
followed by the estimate \eqref{eq:p:holder:time:esti}
\begin{align*}
	|\bfp^R_{t'-s} - \bfp^R_{t-s}|
	\leq
	C (1\wedge t^{-\frac12-\alpha}) \,(t'-t)^\alpha
\end{align*}
together with again
 the uniform bound \eqref{e:uniform-H} and Corollary~\ref{cor:sum-weight},
one obtains the desired bound
$ \Vert J_2\Vert_{2n}^2 \leq C \eps^{2\alpha} |t'-t|^\alpha e^{2a\eps|x|} $.
Combining all these bounds concludes the proof of the proposition.
\end{proof}

\begin{proof}[Proof of Proposition~\ref{prop:Holder-B}]
The proof follows in the same way as the proof of Proposition~\ref{prop:Holder-H},
except that we simply take $a=0$ (erasing the weights $e^{a\eps x}$)
and replace the summation domain $\Z_+$ of the spatial variables by $\{0,\cdots,N\}$,
and then apply
Proposition~\ref{prop:kerB} and Corollary~\ref{cor:sum-int}
instead of Proposition~\ref{prop:heat:ker} and Corollary~\ref{cor:sum-weight}
for all the heat kernel estimates.
Also, regarding the bound for $ \Vert I_1(t,x) -I_1(t,x') \Vert_{2n}^2$
in the proof of the spatial H\"older estimate \eqref{e:Holderx-H},
for two points $x<x'<0$ with $|x-x'|\le N$ we need
\eqref{e:Z0-Z0imeps} and for two points $0<x<x'$ with $|x-x'|\le N$
we then replace the constant $A$ in \eqref{e:Z0-Z0imeps} by $B$.
\end{proof}

\section{Proof of the main theorem}\label{Sec5}

\subsection{A crucial cancellation}

Now that we have proved the tightness result, we would like to identify a limit as the mild solution to SHE \eqref{e:SHE}. We want to show that
the martingale $M$ converges to $\mathscr Z \dot W$, so the bracket
of $M$ should behave like $\eps \mathscr Z^2$ where
the factor $\eps$ is the correct scaling  factor we need.
In view of \eqref{e:Mest}, however, there is another term $\nabla^+ Z_t(x) \nabla^- Z_t(x) $ in the bracket. The ``key estimate" in \cite[Section~4.2]{BG}
shows that this term is actually small; in fact if $Z_t$ is replaced by a heat kernel $p_t$,
then one has the following identity (\cite[Lemma~A.1]{BG}, or \cite[Lemma 4.2]{CST2016asep}):
\[
\sum_{x\in \Z}  \int_0^\infty \nabla^+ p_t(x) \nabla^- p_t(x) \,dt = 0 \;,
\]
which is a crucial cancellation around which the key estimate revolves.
Note that the integrand $\nabla^+ p_t(x) \nabla^- p_t(x)$ can be also written
as $-\nabla^+ p_t(x) \nabla^+ p_t(x-1)$.
We show an analogue of this identity in Proposition~\ref{prop:key-identity} below. On $\R$, the identity can be proved via integration by parts since the heat kernel is just a function of the different $x-y$. In our present situation, this approach fails (the heat kernel depends on the actual values of $x$ and $y$, not just their difference) and we develop below a new method of proof using Green's functions. It is worth noting that the half line identity \eqref{e:key-id-H} is exactly the same as that for the full line, while the bounded interval identity \eqref{e:key-id-B} involves non-zero contribution for all choices of $x$ and $\bar{x}$ -- a new feature.

\begin{proposition} \label{prop:key-identity}
For the Robin heat kernel $\bfp^R$ on $\{0,1,\cdots,N\}$, one has
\begin{equ}[e:key-id-B]
  \sum_{y=0}^N  \int_0^\infty
	\nabla^+_x \bfp^R_t(x,y) \nabla^+_x \bfp^R_t(\bar x,y) \,dt
=\mathbf 1_{x=\bar x} (1-c) - \mathbf 1_{x\neq \bar x} c
\end{equ}
for all $x,\bar x\in \{0,\cdots,N-1\}$,
where the constant $c$ is independent of $x,\bar x$ and such that
$0 \le c\le C\eps$
for some constant $C>0$, and $c=0$ if $\mu_A=1$ or $\mu_B=1$.
For the Robin heat kernel $\bfp^R$ on $\Z_{\ge 0}$, one has
\begin{equ}[e:key-id-H]
\sum_{y=0}^\infty  \int_0^\infty
	\nabla^+_x \bfp^R_t(x,y) \nabla^+_x \bfp^R_t(\bar x,y) \,dt = \mathbf 1_{x=\bar x}
\end{equ}
for all $x,\bar x\ge 0$.
\end{proposition}

\begin{proof}
For the case of the finite intervals,
using the spectral decomposition in \eqref{eq:specdec} ($\{\lambda_k\}_{k=0}^N$ are  the eigenvalues of $-\frac12 \Delta$
with $(\mu_A,\mu_B)$-Robin boundary condition and $\{\psi_k\}_{k=0}^N$  are the corresponding eigenfunctions),  one has
\begin{equs}\label{eq:speckey}
\sum_{y=0}^N &  \int_0^\infty \nabla^+_x \bfp^R_t(x,y) \nabla^+_x \bfp^R_t(\bar x,y) \,dt  \\
&=
\sum_{y=0}^N  \int_0^\infty
	\sum_{k=0}^N \sum_{\bar k=0}^N \nabla^+_x \psi_k(x) \, \psi_k(y) \,  e^{-\lambda_k t} \,
	 \nabla^+_x \psi_{\bar k} (\bar x) \, \psi_{\bar k} (y) \,  e^{-\lambda_{\bar k} t}dt \;.
\end{equs}

The eigenfunctions $\psi_k$ are orthonormal, i.e. $\sum_y \psi_k(y) \psi_{\bar k}(y) =\delta_{k\bar k}$, since the Laplacian with Robin boundary conditions is a finite symmetric matrix.
%
Using this and performing the summation over $y$ we obtain that the above expression in \eqref{eq:speckey} equals
\begin{equ} [e:after-sum-y]
 \int_0^\infty
	\sum_{k=0}^N  \nabla^+_x \psi_k(x)
	 \nabla^+_x \psi_{\bar k} (\bar x)  e^{-2\lambda_{k} t}dt
=\sum_{k=0}^N \frac{ \nabla^+_x \psi_k(x)
	 \nabla^+_x \psi_{k} (\bar x)  }{2\lambda_k}\eqdef F(x,\bar{x}) \;.
\end{equ}

By
Lemma~\ref{lem:Fxy}, it turns out that $F(x,x)$ does not depend on $x$.
Consider $F(0,0)$.
Let $G^R$ be the Green's function of $-\frac12 \Delta$ with the same Robin boundary condition.
Note that $G^R$ is symmetric: $G^R(x,y)=G^R(y,x)$.
Since
\begin{equs} [e:G00-01-11]
-1 & =\frac12\Delta_x G^R(x,0)\Big\vert_{x=0} = \frac12G^R(1,0)-\frac12(2-\mu_A) G^R(0,0)
\\
0 & =\frac12\Delta_x G^R(x,1)\Big\vert_{x=0} = \frac12G^R(1,1)-\frac12(2-\mu_A) G^R(0,1)
\end{equs}
we have, using the identity \eqref{e:GGGG} from Lemma~\ref{lem:Fxy} below,
\begin{equ} [e:F00]
2F(0,0) = G^R(0,0)+G^R(1,1)-2 G^R(1,0)
= (\mu_A-1)^2 G^R(0,0)+ 2 \mu_A
\end{equ}
where in the last step we solved for $G^R(0,1)$ and $G^R(1,1)$ in terms of $G^R(0,0)$
from \eqref{e:G00-01-11}.

Now if $\mu_A=1$, then $F(0,0)=1$ and invoking   \eqref{e:Fcc} from Lemma~\ref{lem:Fxy} below, we obtain \eqref{e:key-id-B} with $c=0$.
We therefore assume $\mu_A<1$, namely $A>0$ below.

Invoking the formula \eqref{e:GR-exact} for $G^R(0,0)$ in Lemma~\ref{lem:GR-exact} below,
and recalling that $\mu_A=1-\eps A$ and  $\mu_B=1-\eps B$,
we can straightforwardly check from \eqref{e:F00} that 
\[
F(0,0) 
=  \frac{(\eps A)^2 (1+N\eps B) }{\eps(A+B) +AB\eps^2 N}+ (1-\eps A)
= \frac{A+B+AB\eps (N-1)}{A+B +AB\eps N} \;.
\]
With $N=1/\eps$, we have
\[
F(0,0)=\frac{A+B+AB-AB\eps}{A+B+AB}
\]
from which we immediately see that the constant $c$ in \eqref{e:Fcc} below
satisfies $0\le  c\le C\eps$ for some constants $C>0$
which only depend on $A,B$.
Therefore invoking \eqref{e:Fcc} again one has $F(x,y) =-c$ where $c$ is as above if $x\neq y$,
and thus
\eqref{e:key-id-B} is proved.

To prove \eqref{e:key-id-H} for the case of the half line,
we start with a Robin heat kernel $\bfp^R$ on $\{0,\cdots,N\}$ with $\mu_A$
and, say, $\mu_B =0$. The same arguments above lead us  to \eqref{e:F00}, namely
\[
\sum_{y=0}^N   \int_0^\infty \nabla^+_x \bfp^R_t(x,y) \nabla^+_x \bfp^R_t(\bar x,y) \,dt
= \frac12 (\mu_A-1)^2 G^R(0,0)+  \mu_A - \mathbf 1_{x\neq \bar x}
 \;.
\]
Taking the limit $N\to \infty$ and  applying  \eqref{e:GR-exact1}  of Lemma~\ref{lem:GR-exact} to the above equation
we obtain  \eqref{e:key-id-H}.
\end{proof}



\begin{lemma} \label{lem:Fxy}
Let $F(x,y)$ be defined as in \eqref{e:after-sum-y}.
Then the function $F$ can be represented as
\begin{equ}[e:GGGG]
F(x,y) = \frac12 \Big(G^R(x,y)+G^R(x+1,y+1)-G^R(x+1,y)-G^R(x,y+1) \Big)
\end{equ}
where $G^R$ is the Green's function of $-\frac12 \Delta$ with Robin boundary condition.
Furthermore, there exists a constant $c$ (possibly depending on $\eps$) such that
\begin{equ} [e:Fcc]
F(x,y)=\mathbf 1_{x=y} (1-c) - \mathbf 1_{x\neq y} c \;.
\end{equ}
In other words $F$ is equal to a constant on the diagonal, and is equal to another constant off diagonal, and these two constants differ by $1$.
\end{lemma}

\begin{proof}
The proof of \eqref{e:GGGG} is simple. Indeed,
since $G^R(x,y)=\sum_{k} \lambda_k^{-1} \psi_k(x)\psi_k(y) $, it is clear that
\[
F(x,y)=
\sum_{k} (2\lambda_k)^{-1}(\psi_k(x+1) - \psi_k(x))(\psi_k(y+1) - \psi_k(y))
\]
is equal to RHS of \eqref{e:GGGG}.
It is also clear that $F(x,y)=F(y,x)$.

To prove \eqref{e:Fcc}, note that for all $x\in\{0,\cdots,N-2\}$,
\begin{equs}
2\nabla^+_x F(x,y)
&= \Big( G^R(x+1,y)+G^R(x+2,y+1)-G^R(x+2,y)-G^R(x+1,y+1) \Big) \\
& \qquad -\Big( G^R(x,y)+G^R(x+1,y+1)-G^R(x+1,y)-G^R(x,y+1) \Big) \\
& =-\Delta_x G^R(x+1,y) + \Delta_x G^R(x+1,y+1) \\
& =2 ( \mathbf 1_{x+1=y} -\mathbf 1_{x=y} ) \;.
\end{equs}
So $F$ only changes values when crossing the diagonal:
\[
\nabla^+_x F(x,y)=
\begin{cases}
-1 &\mbox{if } x=y \\
1 & \mbox{if } x+1=y \\
0 & \mbox{otherwise}
 \end{cases}
\]
This immediately yields \eqref{e:Fcc}.
\end{proof}

%
%

\begin{lemma} \label{lem:GR-exact}
Assume that $\mu_A<1$ or $\mu_B<1$. 
Let $G^R$ be the Green's function of $-\frac12 \Delta$ with Robin boundary condition as above. 
On the finite interval, one has
\begin{equ} [e:GR-exact]
G^R(0,0)=2\cdot \frac{N+1-N\mu_B}{N+2-(N+1)(\mu_A+\mu_B)+N\mu_A\mu_B} \;.
\end{equ}
On $\Z_{\ge}$ one has
\begin{equ} [e:GR-exact1]
G^R(0,0)=\frac{2}{1-\mu_A} \;.
\end{equ}
\end{lemma}

\begin{proof}
If $N=1$, then the operator $- \Delta$ on $\{0,1\}$ with Robin boundary condition is the 2-by-2 matrix
\[
\begin{bmatrix}
    2-\mu_A      & -1  \\
    -1      &  2-\mu_B
\end{bmatrix}_{2\times 2}
\]
The first entry (i.e. the upper-left one) of its inverse matrix can be directly computed, which is
exactly \eqref{e:GR-exact} with $N=2$.
In fact the numerator $2-\mu_B$ on the RHS of \eqref{e:GR-exact}
is precisely the minor of the above matrix
deleting the first row and the first column, 
and the denominator $3-2(\mu_A+\mu_B) + \mu_A\mu_B$
of the RHS of \eqref{e:GR-exact}
is the determinant of the above two by two matrix.

We prove \eqref{e:GR-exact}  for general $N$ by induction and Cramer's rule.
Suppose that for all numbers less than $N$,
the determinant of $- \Delta$ is given by the denominator of the RHS of \eqref{e:GR-exact},
and the cofactor of $2-\mu_A$ (i.e. the minor by deleting the first row and the first column)
 is given by the numerator of  the RHS of \eqref{e:GR-exact}
 (so in particular \eqref{e:GR-exact} holds for all numbers less than $N$), we show
that this is also the case for $N$.

If we denote by $M_N$ the cofactor of the first (i.e. upper-left) entry $2-\mu_A$ (which is a determinant of an $N$ by $N$ matrix) of the matrix $-\Delta$, namely
\[
-\Delta=
\begin{bmatrix}
    2-\mu_A      & (-1,0,\cdots,0)  \\
    (-1,0,\cdots,0)^T      &  *
\end{bmatrix}_{(N+1) \times (N+1)}
\qquad
\det (*) = M_N
\]
then it is easy to see that ($M_0$ is understood as $1$)
\[
M_N = 2\times M_{N-1} - M_{N-2} \;.
\]
Using the induction assumption,
\[
M_N =  2 ( N-(N-1)\mu_B ) - ( N-1-(N-2)\mu_B ) =( N+1)-N\mu_B
\]
which is the desired numerator.

Turning to the determinant of $- \Delta$, we have
\[
\det (- \Delta) = (2-\mu_A) \times M_{N} - M_{N-1}
\]
and again using the induction assumption together with the formula for $M_N$ we have just proved, we have
\[
\det (- \Delta) = (2-\mu_A) (N+1-N\mu_B) - (N - (N-1)\mu_B )
\]
which is equal to the denominator of \eqref{e:GR-exact}.
Therefore \eqref{e:GR-exact} holds for all $N$.

The half line case  \eqref{e:GR-exact1} follows immediately by taking $N\to \infty$.
\end{proof}

%

\begin{corollary} \label{cor:intK-cstar}
Let $\bar T>0$, $a\ge 0$.
For the Robin heat kernel $\bfp^R$ on  $\Z_{\ge 0}$, there exist $\eps_0>0$ and $c_\star<1$ such that
\begin{equ}  \label{e:c-star}
 \sum_{y\ge 1}  \int_0^{\eps^{-2} \bar T}
	\big| \nabla^+_x    \bfp^R_t(x,y) \nabla^-_x \bfp^R_t( x,y)\big| \,
	e^{a\eps |x-y|}  \,dt
 \le c_\star  \;,
\end{equ}
for every $x\ge 1$, $\eps <\eps_0$. Moreover,
for any $S\in[0,\bar T]$, there exists $C>0$ such that with $s:=\eps^{-2}S$ one has
\begin{equ}
 \label{e:C-star}
 \sum_{y\ge 1}  \int_0^{s}
	\big| \nabla^+_x    \bfp^R_t(x,y) \nabla^-_x \bfp^R_t( x,y)\big| \,
	e^{a\eps |x-y|}  \,(s- t)^{-\frac12}\,dt
 \le C\eps  \;,
\end{equ}
for every $x\ge 1$, $\eps <\eps_0$.
For the Robin heat kernel $\bfp^R$ on $\{0,1,\cdots,N\}$, the above bounds hold
for every $x\in\{1,\cdots,N-1\}$
 with $a=0$ and $y$ summing from $1$ to $N-1$.
\end{corollary}

Before our proof, we remark that the proof of these estimates for the standard heat kernel on entire line can be found in \cite[Lemma~A.2--A.3]{BG}.

\begin{proof}
We first consider the finite interval case. To prove \eqref{e:c-star},
using Cauchy-Schwartz inequality,
the LHS of \eqref{e:c-star} with $a=0$ 
is strictly smaller than 
\begin{equ} [e:naiveCS]
\Big(\sum_{y=0}^N  \int_0^\infty
	\big( \nabla^+_x    \bfp^R_t(x,y)\big)^2   \,dt\Big)^{\frac12}
\Big(\sum_{y=0}^N  \int_0^\infty
	\big( \nabla^-_x    \bfp^R_t(x,y)\big)^2   \,dt\Big)^{\frac12}
\end{equ}
which is bounded by $1$ by Proposition~\ref{prop:key-identity} with $x=\bar x$;
here the Cauchy-Schwarz  inequality is strict because $ \nabla^+_x    \bfp^R \neq \nabla^-_x    \bfp^R$. 
This is the argument in \cite{BG} and is sufficient to show that the LHS of \eqref{e:c-star} is bounded by a constant strictly smaller than $1$, but to show that this constant $c_\star$ is actually uniform in $x$, we need to be more careful. We apply the Lagrange identity (which captures the sharpness of Cauchy-Schwartz inequality):
\begin{equs} [e:Lagrange]
\Big(\sum_{y=0}^N  &
	\big( \nabla^+_x    \bfp^R_t(x,y)\big)^2   \Big)
\Big(\sum_{y=0}^N 
	\big( \nabla^-_x    \bfp^R_t(x,y)\big)^2   \Big)
- \Big(\sum_{y=0}^N 
	\big| \nabla^+_x    \bfp^R_t(x,y) \nabla^-_x \bfp^R_t( x,y)\big| \Big)^2 \\
&= \sum_{y,\bar y=0}^N \Big(  
	\big| \nabla^+_x \bfp^R_t(x,y) \nabla^-_x \bfp^R_t( x,\bar y)\big|
		-  \big| \nabla^+_x \bfp^R_t(x,\bar y) \nabla^-_x \bfp^R_t( x,y)\big| \Big)^2 \;.
\end{equs}
It is clear that if  $ \nabla^+_x    \bfp^R $ and $ \nabla^-_x    \bfp^R$ were equal the RHS of 
\eqref{e:Lagrange} would be zero. We claim that there exists $t_0>0$ and $\eps_0>0$
such that for all $t\le t_0$, all $x\in\{1,\cdots,N-1\}$ and all $\eps<\eps_0$,
one has 
\begin{equ} [e:csClaim]
\nabla^\pm_x \bfp^R_t(x,x) \le -\frac{9}{10} \;,
 \quad \mbox{and} \quad
 |\nabla^+_x \bfp^R_t (x,x-1)| \vee   |\nabla^-_x \bfp^R_t (x,x+1)|  \le \frac{1}{10} \;.
\end{equ}
Here $\nabla^\pm_x$ only acts on the {\it first} variable of $\bfp^R$ as before.
(In fact in the special case $t=0$ the claim obviously holds since the first quantity is equal to $-1$ and the second equal to $0$.)
Assuming this claim, for $t\le t_0$ the RHS of \eqref{e:Lagrange}
is  bounded from below by the term\footnote{We bound the sum over $y,\bar y$   from below by only one term in the sum because the summand concentrates around the region where $t$ is small and $y,\bar y$ are close to $x$ (but $y\neq \bar y$).}
 with $y=x$ and $\bar y=x+1$,
\[
\Big(  
	\big| \nabla^+_x \bfp^R_t(x,x) \nabla^-_x \bfp^R_t( x,x+1)\big|
		-  \big| \nabla^+_x \bfp^R_t(x,x+1) \nabla^-_x \bfp^R_t( x,x)\big| \Big)^2
\ge \frac12
\]
where the last step used \eqref{e:csClaim}. Therefore  for all $t$, \eqref{e:Lagrange} can be rewritten  as
\begin{equs} [e:csLower]
\Big(\sum_{y=0}^N  &
	\big( \nabla^+_x    \bfp^R_t(x,y)\big)^2   \Big)^{\frac12}
\Big(\sum_{y=0}^N 
	\big( \nabla^-_x    \bfp^R_t(x,y)\big)^2   \Big)^{\frac12}
- \sum_{y=0}^N 
	\big| \nabla^+_x    \bfp^R_t(x,y) \nabla^-_x \bfp^R_t( x,y)\big| \\
&\ge  \frac{\mathbf 1_{t\le t_0}}{2} \Big( \prod_{\sigma\in\{\pm\}}
\Big(\sum_{y=0}^N  
	\big( \nabla^\sigma_x    \bfp^R_t(x,y)\big)^2   \Big)^{\frac12}
+\sum_{y=0}^N 
	\big| \nabla^+_x    \bfp^R_t(x,y) \nabla^-_x \bfp^R_t( x,y)\big|  \Big)^{-1} \\
&\ge \mathbf 1_{t\le t_0} \prod_{\sigma\in\{\pm\}}
\Big(\sum_{y=0}^N  
	\big( \nabla^\sigma_x    \bfp^R_t(x,y)\big)^2   \Big)^{-\frac12} 
\;.
\end{equs}
Now by \eqref{eq:del:p:esti} one has
$|\nabla^\pm_x \bfp^N_t(x,y)|\leq C $ and by \eqref{eq:dp-int-sum} one has
$\sum_{y=0}^N  
	 \nabla^\pm_x    \bfp^R_t(x,y) \le Ct^{-\frac12}  $,
thus for $t_0/2\le t\le t_0$ the last line of \eqref{e:csLower}  is bounded from
below by a constant $\tilde c$ (which only depends on $t_0$),
in other words it is bounded from below by $\mathbf 1_{t_0/2\le t\le t_0} \tilde c$
for all $t$.

Now we integrate over $t$ from $0$ to $\infty$ on both sides of \eqref{e:csLower}; note that
the time integral of the first term in the first line of \eqref{e:csLower}
can be bounded from above by \eqref{e:naiveCS} using Cauchy-Schwartz inequality,
which is then bounded from above by $1$ as mentioned above.
Therefore we conclude that
\[
 1-\int_0^\infty 
\sum_y 
	\big| \nabla^+_x    \bfp^R_t(x,y) \nabla^-_x \bfp^R_t( x,y)\big| dt 
\ge \int_{t_0/2}^{t_0} \tilde c \,dt=\tilde c t_0/2
\]
which is the desired uniform bound (with $c_\star =1-\tilde c t_0/2$).
It thus remains to prove the claim \eqref{e:csClaim}. For this, we can first prove
the similar estimates for the standard heat kernel on the entire line
\begin{equ} 
\nabla^\pm p_t(0) \le -\frac{19}{20} \;,
 \quad \mbox{and} \quad
 |\nabla^+ p_t (1)| \vee   |\nabla^- p_t (-1)|  \le \frac{1}{20} 
\end{equ}
for sufficiently small time $t$, which is easy to show by using the relation between $p$
and the discrete time standard heat kernel (a relation as \eqref{e:discHKrep}) and the explicit formula for the latter kernel.
We then apply Lemma~\ref{lem:Robin-Dint-abs} and exponential spatial decay 
of $\nabla^\pm p_t$ to obtain \eqref{e:csClaim}.

To prove \eqref{e:C-star}, note that 
\begin{equs}
 \sum_{y=1}^N &  \int_0^{\frac{s}{2}}
	\big| \nabla^+_x    \bfp^R_t(x,y) \nabla^-_x \bfp^R_t( x,y)\big| \,
	 ( s-  t)^{-\frac12}\,dt \\
&\le C \sum_{y=1}^N  \int_0^{\frac{s}{2}}
	\big| \nabla^+_x    \bfp^R_t(x,y) \nabla^-_x \bfp^R_t( x,y)\big| \,
	s^{-\frac12}\,dt
 \le \eps C  \;,
\end{equs}
where in the last step we used $s^{-\frac12}=\eps S^{-\frac12}$
and the bound \eqref{e:c-star}. On the other hand, by   \eqref{eq:del:p:esti} with $n=1,v=1$ and \eqref{eq:dp-int-sum} with $a=0$,
\begin{equs}
 \sum_{y=1}^N  \int_{\frac{s}{2}}^s
	\big| \nabla^+_x    \bfp^R_t(x,y) \nabla^-_x \bfp^R_t( x,y)\big| \,
	 ( s-  t)^{-\frac12}\,dt 
\le C \int_{\frac{s}{2}}^s
	t^{-1} t^{-\frac12}( s-  t)^{-\frac12}\,dt\\
 \le C \Big(-2(s-t)^{\frac12} s^{-1} t^{-\frac12} \Big|_{t=\frac{s}{2}}^s\Big) 
 \le C s^{-1} =C\eps^2 S \le C\eps^2 \;,
\end{equs}
so  \eqref{e:C-star} is proved.

Regarding the kernel half line case, 
note that both $\bfp^R$ and $e^{a\eps |x-y|}$ depend on $\eps$. 
Denote by $L(\eps,\bar\eps)$ the LHS of \eqref{e:c-star} with $e^{a\eps |x-y|}$ replaced by $e^{a\bar \eps |x-y|}$. We show that there exists $\eps_0>0$ such that for all $\eps,\bar\eps<\eps_0$, one has $L(\eps,\bar\eps)\le c_\star<1$, thus in particular $L(\eps,\eps)\le c_\star<1$.

By \eqref{eq:del:p:esti} one has
$|\nabla^\pm_x \bfp^R_t(x,y)|\leq C (1\wedge t^{-1})  $,
and as in the proof of Corollary~\ref{cor:sum-weight} (namely following the arguments starting from \eqref{e:eq:p:esti:sumtp} which lead to the bound \eqref{eq:p:esti:sum})
we have
$
	\sum_{y \ge 0}     \bfp^R_t(x,y) \,e^{a\eps |x-y|}
	\leq C  $.
So
\[
\int_0^{\infty}
\!
\sum_{y\ge 0}
	\big| \nabla^+_x    \bfp^R_t(x,y) \nabla^-_x \bfp^R_t( x,y)\big|
	\,e^{a\eps |x-y|}  \,dt
\le  C \int_0^\infty (1\wedge t^{-1})\,t^{-\frac12}\,dt
\le  C \;.  
\]
With this integrability estimate at hand and note that $e^{a\bar \eps|x-y|}$ monotonically decreases to $1$ as $\bar\eps\to 0$, we can apply the dominated convergence theorem, 
\[
\lim_{\bar\eps\to 0} L(\eps,\bar\eps)=
\int_0^\infty
\sum_{y\ge 0}
	\big| \nabla^+_x    \bfp^R_t(x,y) \nabla^-_x \bfp^R_t( x,y)\big|   \,dt \;.
\]
The RHS can be bounded by a constant $c_\star$ strictly smaller than $1$, uniformly in $\eps<\eps_0$ and $x\ge 1$,  following the same arguments as in the finite interval case.
This implies that by slightly increasing $c_\star$ and slightly decreasing $\eps_0$,
one has $L(\eps,\bar\eps) \le c_\star$ for all $\eps,\bar\eps<\eps_0$ which is the desired bound.

\eqref{e:C-star} for the half line case then follows similarly as the finite interval case by cutting the integral into two parts $t<s/2$ and $t\ge s/2$ and applying \eqref{e:c-star}.
\end{proof}

\subsection{Key estimate and identifying the limit}

In order to identify the limit of $\mathcal Z^\eps$,
we will use an equivalent formulation of the mild solution called the martingale solution.

\begin{definition} \label{def:martingale-solution}
Let $I $ be the interval $[0,1]$ or $[0,\infty)$  
and $(\varphi,\psi)\eqdef \int_I \varphi(X)\psi(X)\,dX$. Let
\begin{equs} [e:bc-test]
C^\infty_A &\eqdef \{\phi\in C_c^\infty(\R) \,|\, \varphi'(0)=A\varphi(0)\} \;,\\
C^\infty_{A,B}& \eqdef \{\phi\in C_c^\infty(\R) \,|\, \varphi'(0)=A\varphi(0),\varphi'(1)=-B \varphi(1)\} \;,
\end{equs}
where $C_c^\infty (\R)$ is the space of compacted supported smooth functions on $\R$.
We say that a probability measure $\CQ$ on $C(\R_+,C(I))$
 {\it solves the martingale problem} for SHE \eqref{e:SHE}
with initial condition $ \mathscr Z^{ic}$
if it satisfies the following requirements.
Letting $ \mathscr Z$ be the canonical coordinate in $C(\R_+,C(I))$
we have
$ \mathscr Z(0,\cdot) = \mathscr Z^{ic}$
in distribution, and
for all $\varphi\in C_A^\infty$
if $I=\R_+$ or all $\varphi\in C_{A,B}^\infty $
 if $I=[0,1]$, the processes
\begin{equs}
	\label{eq:Nt}
	N_T(\varphi )
	& \eqdef
	(\mathscr Z_T,\varphi )- (\mathscr  Z_0,\varphi)
	- \frac12 \int_0^T (\mathscr  Z_S,\varphi'')\,dS
\\
	\label{eq:LambdaT}
	Q_T(\varphi)
	& \eqdef
	N_T(\varphi)^2 -\int_0^T (\mathscr  Z_S^2,\varphi^2 )\,dS
\end{equs}
are $\CQ$-local martingales.
If  $I=[0,\infty)$, 
we further require that
for all $\bar T>0$, there exists $a\ge 0$ such that
\begin{equ} \label{e:mart-uni}
\sup_{T\in[0,\bar T]} \sup_{X\in\R_+} e^{-a X}
	\E \big( \mathscr  Z_T(X)^2\big) <\infty \;.
\end{equ}
\end{definition}

{\bf Notation:}
Throughout  the rest of this section, we denote $\Lambda=\{0,\cdots,N\}$ for ASEP-B, and $\Lambda=\Z_{\ge 0}$ for ASEP-H; $\Lambda_0=\{1,\cdots,N-1\}$ for ASEP-B, and $\Lambda_0=\Z_{> 0}$ for ASEP-H.

\begin{proposition} \label{prop:unique}
For ASEP-H, consider any initial conditions $Z^\e_0$ satisfying Assumption~\ref{def:nearEq}
such that $ Z^\e_0 \Rightarrow \mathscr Z^{ic} $ as $ \e \to 0 $, where $ \mathscr Z^{ic} \in C(\R_+)$. Then any limit point of $\mathcal Z^\eps$ solves the martingale problem on $\R_{+}$ in Definition \ref{def:martingale-solution} with initial data $ \mathscr Z^{ic} $ satisfying the Robin boundary condition with parameter $A$.

For ASEP-B, consider any initial conditions $Z_0^\e$ satisfying Assumption~\ref{def:icB}
such that $ Z^\e_0 \Rightarrow \mathscr Z^{ic}$ as $ \e \to 0 $, where $ \mathscr Z^{ic} \in C([0,1])$. Then any limit point of $\mathcal Z^\eps$ solves the martingale problem on $[0,1]$ in Definition \ref{def:martingale-solution} with initial data $ \mathscr Z^{ic} $ satisfying the Robin boundary condition with parameter $(A,B)$.
%
%
\end{proposition}

\begin{proof} 
By the uniform bound \eqref{e:uniform-H} in the case of ASEP-H, 
any limit point of the family $\mathcal Z^\eps$ satisfies \eqref{e:mart-uni}.
Since $\mathcal Z^\e_0 \Rightarrow \mathscr Z^{ic} $,
the initial condition of the martingale problem is also satisfied
for any limit point. So it only remains to show that any limit point satisfies
the conditions  \eqref{eq:Nt} and \eqref{eq:LambdaT}.

Define for all $t\in[0,\eps^{-2}\bar T]$,
\begin{equs}
	(Z_{t},\varphi)_\eps
	&\eqdef \eps \sum_{x=0}^N \varphi(\eps x) Z_{t}(x) \quad \varphi\in C^\infty_{A,B} \;,
	\\
	 \mbox{or} \quad
	(Z_{t},\varphi)_\eps
	&\eqdef \eps \sum_{x=0}^\infty \varphi(\eps x) Z_{t}(x) \quad \varphi\in C^\infty_A \;,
\end{equs}
for the finite interval and the half line case respectively.

Consider the microscopic analogs of \eqref{eq:Nt} and \eqref{eq:LambdaT} as ($\Delta$ below is the discrete Laplacian with Robin boundary conditions)
\begin{equs}[eq:Nte]
	N^\eps_T(\varphi) & \eqdef
	(Z_{\eps^{-2} T},\varphi)_\eps - (Z_0,\varphi)_\eps
	-\frac12 \int_0^{\eps^{-2}T}
	(\Delta Z_{s},\varphi)_\eps \,ds
\\
	Q_T^\e(\varphi) &\eqdef
	N^\e_T(\varphi)^2
	- \langle N^\eps_T(\varphi)\rangle \;.
\end{equs}
Indeed, by Lemma~\ref{lem:HC},
$ N^\eps_T(\varphi) $ and hence $ Q^\eps_T(\varphi) $ are martingales.
Now we would like to show that $N^\eps_T(\varphi)$
can be rewritten as
\begin{equ} [e:aim-to-N]
\eps \!\sum_{X\in \eps \Lambda} \!\mathcal Z^\eps_{T}(X) \varphi(X)
- \eps \!\sum_{X\in \eps \Lambda} \!\mathcal Z^\eps_0 (X)\varphi(X)
-\frac12 \int_0^{T} \eps \!\sum_{X\in \eps \Lambda}
	\mathcal Z^\eps_S(X)\varphi''(X) \,dS + \mbox{Err}^{(1)}
\end{equ}
and that $Q_T^\e(\varphi)$
can be rewritten as
\begin{equ} [e:aim-to-Lambda]
N^\e_T(\varphi)^2 -
\int_0^{T} \eps \!\sum_{X\in \eps \Lambda}
	\mathcal Z^\eps_S(X)^2 \varphi(X)^2 \,dS+ \mbox{Err}^{(2)}
\end{equ}
where the error terms $\mbox{Err}^{(1)},\mbox{Err}^{(2)}$ vanish in probability as $\eps\to 0$.
Here 
$\eps \Lambda =\{X\in \R \,|\, X/\eps \in \Lambda\}$.
Assuming this vanishing error, the proof is completed by passing to the limit along a converging subsequence,
and noting that $\eps\Lambda\to I$ and the Riemann sums converge to the integrals by continuity of $\phi$
and $\mathcal Z^\eps$ and its limit $\mathcal Z$. Thus, we conclude that any limiting point $\mathcal Z$ of the sequence $\mathcal Z^\eps$
satisfies \eqref{eq:Nt} and \eqref{eq:LambdaT}. So show that the error terms in \eqref{e:aim-to-N} and \eqref{e:aim-to-Lambda} goes to zero as desired.

We start by arguing for \eqref{e:aim-to-N}. Consider the last term in the expression of $N^\eps_T(\varphi)$.
In the finite interval case, applying the summation by parts formula \eqref{e:sum-by-parts1} with $u(\Cdot)=\varphi(\eps\Cdot)$ and $v(\Cdot)=Z_s(\Cdot)$
(recall that $N=1/\eps$ so that for instance $u(N+1)=\varphi(1+\eps)$), we find that
\begin{equs}
	( \Delta Z_{s}, & \varphi)_\eps
	=
	\eps
	\sum_{x=0}^N Z_s (x) \Delta\varphi(\eps x)
	+ \eps \nabla^+Z_s (N) \varphi(1+\eps) + \eps \nabla^- Z_s (0) \varphi(-\eps) \\
	& \qquad\qquad - \eps Z_s (N+1) \big( \varphi(1+\eps)-\varphi(1)\big)
		- \eps Z_s (-1) \big(\varphi(-\eps)-\varphi(0)\big) \\
	& =\eps \sum_{x=0}^N Z_s (x) \Delta\varphi(\eps x)
	-\eps^2 B\, Z_s (N) \varphi(1+\eps)
	- \eps^2 A \,Z_s (0) \varphi(-\eps) \\
	&\qquad - \eps (1-\eps B) Z_s (N) \big( \varphi(1+\eps)-\varphi(1)\big)
	- \eps (1-\eps A) Z_s (0) \big(\varphi(-\eps)-\varphi(0)\big)
\end{equs}
where $\Delta\varphi(\eps x) = \varphi(\eps (x+1))+\varphi(\eps (x-1))-2\varphi(\eps x)$.
Note that in the last equality we used the  Robin boundary condition
for $Z_s$, 
for instance $ \nabla^+Z_s (N) = (\mu_B-1) Z_s (N)= -\eps B\, Z_s (N)$.
Therefore
\begin{equs} [e:pre-bc-phi]
 \int_0^{\eps^{-2}T} \!\!\!\!\!\!\!\!\!\!\!\!
	(\Delta Z_{s},\varphi)_\eps \,ds
&= \int_0^{\eps^{-2}T} \!\!\!\! \!\!\Big( \eps \sum_{x=0}^N Z_s (x) \varphi ''(\eps x)
	-B Z_s (N) \varphi(1)
	-A Z_s (0) \varphi(0)  \\
&\qquad\qquad
	- Z_s(N) \phi'(1) + Z_s(0) \phi'(0)
	 + R_0^\eps (\phi) \Big) \eps^2\,ds
\end{equs}
where $\phi''(\eps x)$ stands for the second continuous derivative of $\varphi$ evaluated at $\eps x$
and the error term
\begin{equs}
R_0^\eps & (\phi) \eqdef
\Big(Z_s \,,\, \eps^{-2}\Delta\phi(\eps \Cdot) - \phi''(\eps \Cdot)\Big)_\eps \\
	&  +\Big( B Z_s (N) \varphi(1)- B Z_s (N) \varphi(1+\eps) \Big)
	+\Big( A Z_s (0) \varphi(0)- A Z_s (0) \varphi(-\eps) \Big) \\
& + \Big(Z_s(N) \phi'(1) - (1-\eps B) Z_s (N) \frac{ \varphi(1+\eps)-\varphi(1)}{\eps} \Big)\\
& + \Big(-Z_s(0) \phi'(0) -(1-\eps A) Z_s (0) \frac{ \varphi(-\eps)-\varphi(0)}{\eps} \Big) \;.
\end{equs}
Since $\phi$ is smooth, by the uniform bound \eqref{e:uniform-B} on $Z$,
we have $\E(R_0^\eps (\phi)^2) \to 0$ as $\eps\to 0$.
 Invoking the assumed Robin boundary condition \eqref{e:bc-test} on $\phi$
 the four boundary terms on the RHS of \eqref{e:pre-bc-phi} add up to zero.

Therefore by change of variables $X=\eps x$, $S=\eps^2 s$ and the definition
\eqref{e:calZ} of $\mathcal Z^\eps$,
we have cast $N^\eps_T(\varphi)$
into the desired form \eqref{e:aim-to-N} where the error
$\mbox{Err}^{(1)}=-\frac12 \int_0^T R_0^\eps (\phi)dS$
vanishes in probability.
This completes the proof that  any limiting point $\mathcal Z$ of the sequence $\mathcal Z^\eps$
satisfies \eqref{eq:Nt} for the finite interval case.
 The half line case can be shown in the same way by applying the summation by parts formula \eqref{e:sum-by-parts2} and invoking both the discrete boundary conditions for $Z$ and the continuous boundary condition for the test function $\phi$.


%

Turning to $ Q_T^\e(\varphi) $ and \eqref{e:aim-to-Lambda},
we apply Lemma~\ref{lem:Mbracket} to calculate $ \langle N^\eps_T(\varphi)\rangle $
and obtain the following expression for $ Q_T^\e(\varphi) $:
\begin{equs}\label{eq:LambdaTe}
	Q^\eps_T(\varphi) &\eqdef
	N_T^\eps(\varphi)^2
		- \eps^2 \int_0^{\eps^{-2}T} (Z_{s}^2,\varphi^2)_\eps \,ds
		+R_1^\eps (\varphi) +R_2^\eps (\varphi)  
\end{equs}
where 
\begin{equs}
R_1^\eps(\varphi) &\eqdef
o(\eps^2)
	\int_0^{\eps^{-2}T} (Z_{s}^2,\varphi^2)_\eps \,ds \;, \\
R_2^\eps(\varphi) &\eqdef  \eps^2
\int_0^{\eps^{-2} T} \sum_{x\in \Lambda_0} \nabla^- Z_{s}(x)\nabla^+ Z_{s}(x) \varphi(x)^2  \,ds \;,
\end{equs}
where in the definition of $R^\eps_2$, $x$ is summed in the ``bulk"
$\Lambda_0$.
Since the second term on RHS of \eqref{eq:LambdaTe}
is equal to the second term in \eqref{e:aim-to-Lambda} after passing to the macroscopic variables $T,X$,
it now suffices
to prove that $ \E(R^\e_i(\varphi))^2 \to 0$, for $i=1,2$.
By the uniform bound \eqref{e:uniform-H} or \eqref{e:uniform-B}  on $Z$,
we clearly have $\E (R_{1}^\eps(\varphi)^2 ) \to 0$.
To control $ R_2^\eps(\varphi)^2 $,
we follow \cite{BG} by using the ``key estimate'' as in Lemma~\ref{lem:key-est} in the following.
Letting $ \mathcal F_{t} \eqdef \sigma( Z_s(x) : x\in\Z_{\ge 0}, s\leq t ) $
denote the canonical filtration and for $y\in\Lambda_0$ letting
\begin{align}\label{eq:U}
	U^\e(y,s,s') \eqdef
	\E ( \nabla^- Z_{s}(y)\nabla^+ Z_{s}(y) \,|\,  \mathcal F_{s'} ),
\end{align}
we have that
\begin{equs}
\E ( & R_2^\eps(\varphi)^2 ) \\
& =
\eps^4
 \int_0^{\eps^{-2}T} \!\!\!\! ds \int_0^s ds'
	 \sum_{x,y\in\Lambda_0} \varphi(\eps x)^2 \varphi(\eps y)^2
	\E \Big(
	\nabla^- Z_{s'}(x)\nabla^+ Z_{s'}(x) U^\e(y,s,s') \Big) \;.
\end{equs}
Here $x,y$ are both summed in the bulk $\Lambda_0$ as above.
With $ |\nabla^\pm Z_t(x)| \le C\eps^{\frac12} Z_t(x) $,
we further obtain
\begin{equ} \label{e:double-E}
	\E (R_2^\eps(\varphi)^2 )
  \le C \eps^5
 \int_0^{\eps^{-2}T} \!\!\!\!  ds \int_0^s ds'
 \sum_{x,y\in \Lambda_0} \varphi(\eps x)^2 \varphi(\eps y)^2
	\E \Big(
	Z_{s'}(x)^2 U^\e(y,s,s') \Big) \;. 	
\end{equ}

Note if we were to simply use $ |\nabla^\pm Z_t(x)| \le C\eps^{\frac12} Z_t(x) $
to ``brutally" bound
	$ |U^\e(y,s,s')| $ by $ \e C Z^2_s(y) $,
the resulting bound on $ \E(R^\e_2(\varphi)^2) $ would be of order $ O(1) $,
(since the change of time and space variables to macroscopic variables gives $\eps^{-6}$),
which would be insufficient for our purpose.
To show $ \E(R^\e_2(\varphi)^2) \to 0 $,
we divide the time integrals in \eqref{e:double-E} into two parts.
The first part consists of the region with $s'<\eps^{-3/2}$,
where the above brutal estimate gives a bound of $O(\eps^{1/2})$.
For the second part which is the rest of the integral
we utilize the key Lemma~\ref{lem:key-est} below,
which provides  the important decay factor $(t-s)^{-\frac12}$ that is typically $O(\eps)$, thus together with the factor $\eps^{\frac12-\delta}$ there improving the above brutal bound.
Indeed, as in \cite[Proof of Proposition~4.11]{BG}, in the case of ASEP-H, conditioning on $\{Z_{s'}(x) \le \bar K\}$, \eqref{e:double-E} can be bounded by
\[
C \bar K^2 \eps^3 e^{a\eps (x+y)}
 \int_0^{\eps^{-2}T} \!\!  \int_0^s
	 \eps^{\frac12-\delta} (s-s')^{-\frac12} ds' \,ds
\le  	
C e^{a\eps (x+y)} \eps^{\frac12-\delta} \bar K^2 \;,
\]
 while conditioning on $\{Z_{s'}(x) > \bar K\}$
\eqref{e:double-E} can be bounded using Chebyshev inequality by $C e^{a\eps (x+y)}  \bar K^{-2}$. Letting $\eps\to 0$ and then $\bar K\to \infty$ we obtain $ \E(R^\e_2(\varphi)^2) \to 0 $.
For ASEP-B the argument is the same with $a=0$.

Therefore we have shown that $Q^\eps_T(\varphi)$ as in
\eqref{eq:LambdaTe}
can be indeed cast into the form \eqref{e:aim-to-Lambda} with error term
$\mbox{Err}^{(2)} = R_1^\eps (\varphi) +R_2^\eps (\varphi)$ which vanishes in probability. By the arguments below \eqref{e:aim-to-Lambda},
we conclude that any limit $\mathcal Z_\eps$ satisfies the conditions in Definition~\ref{def:martingale-solution} and thus
solves the martingale
problem.
\end{proof}

\begin{lemma} \label{lem:key-est}
Assume the above setting, and consider both the ASEP-H and the ASEP-B models.
For all $\bar T>0$, $\delta>0$, there are constants $ a,C>0 $ such that
\begin{equ} \label{e:key-est}
	\sup_{x\in \Lambda_0} e^{-a\eps x}
	\E | U^\e(x,t,s) |
	\le C
	 \eps^{\frac12-\delta} (t-s)^{-\frac12}
\end{equ}
for all $\eps^{-3/2} \le s< t \le \e^{-2}\bar T$ and all $\eps>0$.
Here, for ASEP-B, the constant $a$ may be taken to be zero.  
\end{lemma}

\begin{proof}[Proof of Lemma~\ref{lem:key-est}]
The proof follows similar argument as in \cite[Lemma~4.8]{BG}.
Throughout the proof we assume $a=0$ in the case of ASEP-B.
%
Let $Z_t(x)=I_t(x)+N_t^t(x)$ where
\[
I_t (x)=\sum_{y\in \Lambda} \bfp^R_t(x,y)Z_0(y) \;, \qquad
 N_s^t(x) \eqdef \int_0^s \sum_{y\in\Lambda} \bfp^R_{t-\tau} (x,y) dM_\tau (y) \;.
 \]
Note that $ N_s^t(x)$ is a martingale in $s$.
For $s\le r\le t$, one has
\begin{equs} [e:dNdN]
\E \Big(
		\nabla^- N_{r}^t(x)\nabla^+ N^t_{r}(x)
	& \,\big|\, \mathcal F_{s}\Big)
=\nabla^- N_{s}^t(x)\nabla^+ N^t_{s}(x) \\
 & + \E\Big(  \int_s^r \sum_{y\in\Lambda}
 	 K_{t-\tau} (x,y) d\langle M(y),M(y)\rangle_\tau
\,\Big|\, \mathcal F_s\Big)
\end{equs}
where
\begin{equ} \label{e:def-K}
	K_t(x,y) \eqdef \nabla^+_x \bfp^R_t(x,y) \nabla^-_x \bfp^R_t(x,y) \;.
\end{equ}
With $ U^\e(x,t,s) $ defined as in \eqref{eq:U} and with $\E(  N_r^t(x) | \mathcal F_s )=N_s^t(x)$,
one has by \eqref{e:dNdN}
\begin{equs}
	U^\e(&  x,t,s)
	= \nabla^- I_t(x) \nabla^+ I_t(x)
	+ \nabla^- I_t(x) \nabla^+ N_s^t(x)
	+ \nabla^- N_s^t(x) \nabla^+ I_t(x)
\\
	&+
	\nabla^- N_{s}^t(x)\nabla^+ N^t_{s}(x)
	+ \E\Big(  \int_s^t \sum_{y\in\Lambda}
		K_{t-\tau} (x,y) d\langle M(y) \rangle_\tau
	\,\Big|\, \mathcal F_s\Big) \;.   \label{e:five-terms}
\end{equs}

We bound the $L^1$-norms (i.e. $\E|\Cdot|$) of the terms on the RHS.
For the first four terms,
by the Cauchy-Schwartz inequality one needs only to show
\begin{align}\label{eq:IN:bd}
	\E(\nabla^\pm I_t(x))^2 \;, \; \E(\nabla^\pm N^t_s(x))^2
	\leq C \eps^{\frac12}(t-s)^{-\frac12}e^{2a\eps x}.
\end{align}

{\it Estimates for $\E(\nabla^\pm I_t(x))^2 $.} We use \eqref{e:init-uniform} to obtain
\begin{equs}
	\E \Big( \big(\nabla^\pm I_t(x) \big)^2 \Big)
	&=
	\sum_{y,y'} \nabla^\pm \bfp^R_t(x,y)\nabla^\pm \bfp^R_t(x,y')
	\,\E \big( Z_0(y) Z_0(y')\big)
\\
	&\le
	C  \Big( \sum_{y\in\Lambda}  \nabla^\pm \bfp^R_t(x,y)\,e^{a\eps y}  \Big)^2 \;.
	\label{e:EgradI2}
\end{equs}
By Corollary~\ref{cor:sum-weight} or
Corollary~\ref{cor:sum-int}, one can bound the above quantity by
$C  e^{2a\e x } t^{-1}$.
%
Further expressing $ t^{-1} $ as $ t^{-1/2} t^{-1/2} $,
and applying $ t^{-\frac12} < (t-s)^{-\frac12} $
and $ t^{-1/2} \leq \e^{3/4} $ (since we assume $t \geq \e^{-3/2} $),
we obtain desired bound on $\E(\nabla^\pm I_t(x))^2 $ as in \eqref{eq:IN:bd}.

{\it Estimates for $ \E(\nabla^\pm N_s^t(x))^2 $.} One has
\begin{align*}
	\E\Big( \big(\nabla^\pm &  N_s^t(x) \big)^2\Big)
	 =
	\E \int_0^s \sum_{y\in\Lambda} \big(\nabla^\pm \bfp^R_{t-\tau} \big)^2 (x,y)
	 d\langle M \rangle_\tau (y)\;
\\
	&\leq	
	C
	\int_0^s
	\Big( \sup_{y\in\Lambda} |\nabla^+ \bfp^R_{t-\tau}(x,y)| \Big)
	\Big( \sum_{y\in\Lambda} |\nabla^- \bfp^R_{t-\tau}(x,y)|\,
	\E|\tfrac{d~}{d\tau}\langle M \rangle_\tau(y)| \Big) d\tau .
\end{align*}
By combining Lemma~\ref{lem:Mbracket},
the bound $\nabla^{\pm} Z_t(y)\leq C \eps^{\frac12} Z_t(y)$,
 and the uniform bounds \eqref{e:uniform-H} or \eqref{e:uniform-B},
one has $ \E|\frac{d}{d\tau}\langle M(y) \rangle_\tau| \le C\eps e^{a\eps y}$;
We then apply the same reasoning as used above to bound \eqref{e:EgradI2},
together with 
\eqref{eq:del:p:esti} for $v=1$,
to obtain
\begin{align*}
	\E \Big( \big(\nabla^\pm N_s^t(x) \big)^2\Big)
	\leq
	C \e e^{2a x} \int_0^s (t-\tau)^{-3/2} ds.
\end{align*}
Upon integrating over $ \tau $,
we obtain the desired bound on $ \E(\nabla^\pm N_s^t(x))^2 $ as in \eqref{eq:IN:bd}.

{\it Estimates for the last term on the RHS of \eqref{e:five-terms}.}
We use the explicit expression of the predictable quadratic variation given in Lemma~\ref{lem:Mbracket}
to re-write the last term on the RHS of \eqref{e:five-terms} as
$ \tilde I_1 + \tilde I_2 $ where
\begin{equs}
	\tilde I_1 (s,t,x) &\eqdef ( \eps +o(\eps)) \sum_{y\in\Lambda} \int_s^t K_{t-\tau}(x,y)\,
		\E (Z_\tau(y)^2 \,|\,\mathcal F_s) \,d\tau \;,
\\
	\tilde I_2 (s,t,x) &\eqdef - \sum_{y\in\Lambda_0} \int_s^t K_{t-\tau}(x,y)\,
		\E (\nabla^-Z_\tau(y)\nabla^+Z_\tau(y) \,|\,\mathcal F_s) \,d\tau  \;.
\end{equs}

By Lemma~\ref{lem:key-of-key} below, we can bound  $\E | \tilde I_1 (s,t,x)|$
by $C \eps^{\frac12-\delta} e^{a\eps x} (t-s)^{-1/2} $.
So we have obtained the desired bound on all the terms coming from the RHS of \eqref{e:five-terms}
except for the term $\tilde  I_2 $;
but $\tilde  I_2$ contains the same conditional expectation
on the LHS of \eqref{e:five-terms},
which means that one can bound this conditional expectation
in terms of an expression involving this conditional expectation itself.
%
%
Indeed, collecting the bounds for  the terms in \eqref{e:five-terms},
one has
\begin{equs} [e:iterate-EE]
e^{-a\eps x}& \E | U^\e(x,t,s) |
  \le C\eps^{\frac12-\delta}(t-s)^{-\frac12}
+ e^{-a\eps x}
\sum_{y\in\Lambda_0} \int_s^t  |K_{t-\tau} (x,y)|\cdot \E | U^\e(y,\tau,s) | \,d\tau \\
 & \le C\eps^{\frac12-\delta}(t-s)^{-\frac12}
+
\sum_{y\in\Lambda_0} \int_s^t  |K_{t-\tau} (x,y)|  e^{a\eps |x-y|} \cdot e^{-a\eps y} \E | U^\e(y,\tau,s) | \,d\tau \;.
\end{equs}
Now as in \cite[Lemma~4.8]{BG} we iterate the bound
 \eqref{e:iterate-EE} to find that $e^{-a\eps x} \E | U^\e(x,t,s) |$ is bounded by
\begin{equs}
 &C\eps^{\frac12-\delta}(t-s)^{-\frac12} 
+ \sum_{y_1\in\Lambda_0}
	 \int_s^t  |K_{t-\tau_1} (x,y_1)|
	e^{a\eps |x-y_1|}   \Big(
		C\eps^{\frac12-\delta}(\tau_1-s)^{-\frac12} 
\\
	& \qquad  \qquad 
		 + \sum_{y_2\in\Lambda_0} \! \int_s^{\tau_1}  
		|K_{\tau_1-\tau_2} (y_1,y_2)|
		e^{a\eps |y_1-y_2|}  
		  \cdot e^{-a\eps y_2} \E | U^\e(y_2,\tau_2,s) |  \,d\tau_2\Big) \,d\tau_1.
\end{equs}
Expand the parenthesis we then obtain three terms, with the second term
(a single summation-integration over space-time) given by $A_1$ defined below,
and the third term (a double summation-integration over space-time) again involving the expectation of $U^\eps$.
Repeat this iteration we eventually get
\begin{equ} 
e^{-a\eps x} \E | U^\e(x,t,s) |
  \le C\eps^{\frac12-\delta}(t-s)^{-\frac12}
+\sum_{n\ge 1} A_n(x,t,s)
\end{equ}
where, with $y_0 := x$, $\tau_0:=t$
\begin{equ} 
A_n(x,t,s)
= \int
\sum_{y_1,\cdots,y_n \in \Lambda_0}
 C\eps^{\frac12-\delta}(\tau_n-s)^{-\frac12} \prod_{i=1}^n |K_{\tau_{i-1}-\tau_i} (y_{i-1},y_i)|
 e^{a\eps |y_{i-1}-y_i |}
  \,\prod_{i=1}^n d\tau_i \;,
\end{equ}
where the  time integrals are over $s\le \tau_n\le \cdots \le \tau_1 \le t$.
To estimate $A_n$ for each $n$,
we first integrate over $\tau_n\in[s,\tau_{n-1}]$ and sum over $y_n\in\Lambda_0$
in the following way: by a change of variable $\tau_n=\tau_{n-1}-r$
we instead integrate over $r\in [0, \tau_{n-1}-s]$, and we are in the scope of applying \eqref{e:C-star} and obtain a factor $C\eps$.
We then integrate and sum over the other space-time variables using \eqref{e:c-star},
which yields a bound by $c_\star^{n-1}$ where $c_\star$ was introduced in \eqref{e:c-star}.
Together with the factor $ C\eps^{\frac12-\delta}$ on the right hand side of 
the above definition of $A_n$, we obtain that $|A_n(x,t,s)| \le  C\eps^{\frac32-\delta} c_\star^{n-1}$.
Since $c_\star<1$
one obtains a convergent series 
and the sum over $n$ yields a bound by $C\eps^{\frac32-\delta}$
which is smaller than $ C\eps^{\frac12-\delta}(t-s)^{-\frac12}$
where the last constant $C$ depends on $\bar T$.
Therefore the bound \eqref{e:key-est} follows.
\end{proof}

\begin{lemma} \label{lem:key-of-key}
Let $\mathcal F_s$ be the filtration defined above, and $K_t(x,y)$ be the deterministic function defined in \eqref{e:def-K}. There exist $a,C>0$ such that for any $\delta>0$, $0\le s\le t \le \eps^{-2}\bar T$, $x\in \Lambda_0$, and $\eps>0$ sufficiently small, we have
\begin{equ} [e:key-of-key]
\E \Big[ \Big|
\eps \sum_{y\in\Lambda} \int_s^t K_{t-\tau}(x,y)\,
		\E (Z_\tau(y)^2 \,|\,\mathcal F_s) \,d\tau \Big|\Big]
\le
C \eps^{\frac12-\delta} e^{a\eps x} (t-s)^{-1/2} \;.
\end{equ}
For ASEP-B, we can take $a=0$.
\end{lemma}

\begin{proof}
To start with we note that by definition $K_t(x,y)= -\nabla^+_x \bfp^R_t(x,y) \nabla^+_x \bfp^R_t(x-1,y)$.
Using
 Proposition~\ref{prop:key-identity} (with the variable $\bar x$ therein taken as $\bar x=x-1$)
we have $\sum_y \int_0^\infty K_\tau(x,y) d\tau=O(\eps)$.
The quantity in the $L^1$ norm $\E|\Cdot|$ on LHS of \eqref{e:key-of-key} can be then rewritten  as
\begin{equs}  [e:keyI1]
	&\eps \sum_{y\in \Lambda} \int_s^t K_{t-\tau}(x,y)\,
	\E (Z_\tau(y)^2-Z_t(x)^2 \,|\,\mathcal F_s) \,d\tau
\\
	& \qquad
	+\eps \,  \E (Z_t(x)^2 \,|\,\mathcal F_s)
	\sum_{y\in\Lambda} \int_{t-s}^\infty K_\tau(x,y) \,d\tau
+ O(\eps^2)\,  \E (Z_t(x)^2 \,|\,\mathcal F_s)
\\
	 &=: J_1+J_2+J_3
\end{equs}
where we have called the three terms as $J_1,J_2,J_3$.
Here 
we have simply subtracted and added a term 
$\sum_{y} \int_s^t K_{t-\tau}(x,y)\,
	\E (Z_t(x)^2 \,|\,\mathcal F_s) \,d\tau$
and made a change of variable $t-\tau \to \tau$ and applied  Proposition~\ref{prop:key-identity};
the expression $O(\eps^2)$ stands for $\eps$ times the constant $c$ in \eqref{e:key-id-B} in the case of ASEP-B, and is actually $0$ in the case of ASEP-H.
We estimate the $L^1$ norm of these three terms separately as follows.
Throughout this proof we set $a=0$ for ASEP-B.

{\it Estimates of $J_3$.}
Note that  $(t-s)^{-1/2} \ge \eps \bar T^{-\frac12} > \eps^{\frac32}$ for $\eps>0$ sufficiently small. Using this and the uniform bound \eqref{e:uniform-H} or \eqref{e:uniform-B} on $Z$ which states that $\Vert Z_t(x) \Vert_{2} \leq C e^{\frac12 a\eps x}$ for a constant $a$, we concluded that that $\E|J_3|$
is bounded by the RHS of \eqref{e:key-of-key}.

{\it Estimates of $J_2$.}
Recall that $K$ is a product of two factors $\nabla^{+}\bfp^R$ and $\nabla^{-}\bfp^R$.
 Applying   
\eqref{eq:del:p:esti}  with $ n=v=1 $ to one of the factors $\nabla^{\pm}\bfp^R$ in $K$, and \eqref{eq:dp:esti:sum} or \eqref{eq:dp-int-sum} to the other factor $\nabla^{\mp}\bfp^R$,
we can bound $ \sum_{y} |K_\tau(x,y)| $ by  $C(1\wedge \tau^{-3/2}) e^{a\eps x}$.
Using this and the uniform bound \eqref{e:uniform-H} or \eqref{e:uniform-B},
we obtain the desired bound on $ J_3 $ as
\begin{equ} \label{e:I1-2nd}
	\E |J_3(s,t,x)|
	\leq
		C \e e^{a\eps x} \int_{t-s}^\infty \tau^{-\frac32} \,d\tau
	\leq
		C \e e^{a\eps x} (t-s)^{-\frac12} \;.
\end{equ}

{\it Estimates of $J_1$.} The idea of controlling $ J_1$
is to use the fact that $ K_{t-\tau}(x,y) $
concentrates on values of $ (\tau,y) $ which are close to $ (t,x) $,
and that, thanks to the H\"older estimates \eqref{e:Holderx-H}--\eqref{e:Holdert-H}, or \eqref{e:Holderx-B}--\eqref{e:Holdert-B},
$ |Z_\tau(y)^2 -Z_t(x)^2| $ is small when $ (\tau,y) \approx (t,x) $.
More precisely,
with
\begin{align*}
	|Z_\tau(y)^2-Z_t(x)^2|
	\leq
	(Z_\tau(y)+Z_t(x)) \big( |Z_\tau(y)-Z_t(y)| + |Z_t(y)-Z_t(x)| \big)
\end{align*}
we use the Cauchy--Schwarz inequality and the H\"older estimates
\eqref{e:Holderx-H}--\eqref{e:Holdert-H} or \eqref{e:Holderx-B}--\eqref{e:Holdert-B}, for $ \alpha = \frac12 -\delta $
to obtain
\begin{align*}
	\E |Z_\tau(y)^2-Z_t(x)^2|
	\le
	C \e^{\frac12 -\delta} e^{a\eps (x+y)} \Big( |y-x|^{\frac12-\delta}+
	(|t-\tau|\vee 1)^{\frac14-\delta/2} \Big) \;.
\end{align*}
%
Therefore, invoking \eqref{eq:del:p:esti} with $n=v=1$
\begin{equs}
	\E |J_1(s,t,x)|
	\leq
	C &\eps^{\frac32-\delta} e^{a\eps x}
	\int_0^{\eps^{-2} \bar T}
	((t-\tau)^{-1}\wedge 1)  
\\
	&
	\times
	\Big(
		\sum_y  | \nabla^- \bfp_{t-\tau} (x,y) |
		\,e^{a\eps y}
		\Big( |y-x|^{\frac12-\delta}+(|t-\tau|\vee 1)^{\frac14-\delta} \Big)
	\Big)
	\,d\tau \;.
\end{equs}
Arguing as in \eqref{e:power-by-exp},
together with \eqref{eq:dp:esti:sum} or \eqref{eq:dp-int-sum},	
we obtain
\begin{equs}
	\E |J_1(s,t,x)|
	\leq
	C \eps^{\frac32-\delta} e^{a\eps x}
	\int_0^{\eps^{-2} \bar T}
	\Big(1\wedge (t-\tau)^{-1} \Big)
	\Big((t-\tau)^{\frac14-\delta} \vee 1 \Big) \Big(t-\tau \Big)^{-\frac12} d\tau
	 \;.
\end{equs}
Here the last factor $(t-\tau )^{-\frac12}$ arises from the
RHS of \eqref{eq:dp:esti:sum} or \eqref{eq:dp-int-sum}.
The integral is bounded by a constant
because as $\tau\to t$ the integrand behaves as $(t-\tau )^{-\frac12}$,
and as $\tau\to \infty$ the integrand behaves as $\tau^{-\frac54}$.
So $\E |J_1(s,t,x)|$ is bounded by $C \eps^{\frac32-\delta} e^{a\eps x} $.
With $ (t-s)^{-1/2} \ge t^{-1/2} \ge {\bar T}^{-1/2} \e $,
the desired bound
\[
	\E |J_1(s,t,x)|
	\leq
	C \eps^{\frac12-\delta} e^{a\eps x} (t-s)^{-1/2}
\]
follows. This concludes the proof of the lemma.
\end{proof}

We show the uniqueness for martingale problem, and its equivalence with the mild solution to the SHE.

\begin{proposition} \label{prop:mart-mild}
Let $I $ be the interval $[0,1]$ or $[0,\infty)$,
and  $\mathscr Z^{ic} \in C(I)$ be a  random function.
If $I=[0,\infty)$ we assume that for each $p\ge 1$ there exists $a>0$ such that
\begin{equ}
\sup_{X\in\R_+} e^{-a X}
	\E \big( \mathscr Z^{ic}(X)^p\big) <\infty \;.
\end{equ}
Then the martingale problem (Definition~\ref{def:martingale-solution}) has a unique solution whose law coincides with that of the mild solution to SHE with initial condition $\mathscr Z^{ic} $.
\end{proposition}

\begin{proof}
We only need to show the uniqueness since the existence follows immediately from
the convergence along subsequences of $\mathcal Z^\eps$ provided in Proposition \ref{prop:unique}.
To prove uniqueness of solution to the martingale problem,
we need to show a martingale representation theorem which,
possibly by extending the probability space and the filtration,
 represents
the local martingale $N$ in \eqref{eq:Nt} as a stochastic integral of $\mathscr Z$ against a Wiener process, essentially following \cite{BG} or \cite{MR958288}.
By Definition~\ref{def:martingale-solution} there exists
a sequence of stopping times $\{\tau_N\}$ such that $\lim_{N\to \infty} \tau_N = +\infty$
$\CQ$-a.s., and $N^{\tau_N}(\phi)$ is a square integrable martingale
and $\Lambda^{\tau_N} (\phi)$ is a martingale. The associated orthogonal martingale measure (see \cite{MR876085} for definition) $N(dS\,dX)$ has quadratic variation measure $\mathscr Z_S(X)^2 dS\,dX$.

We introduce a cylindrical Wiener process $\overline W$ on $L^2(I)$ (possibly by extending the probability space and the filtration) which is independent of $\mathscr Z$. Let $\CQ'$
be the probability measure on the extended space. Define a process
\[
\tilde W_T^N (\phi) \eqdef
\int_0^{T\wedge \tau_N} \!\!\!\! \int_I \frac{1}{\mathscr Z_S(X)}\mathbf 1_{ \{\mathscr Z_S(X)\neq 0 \}} \phi(X) N(dS\,dX)
+ \int_0^{T\wedge \tau_N}\!\!\! \Big(1_{\{\mathscr Z_S(X)= 0 \} }\phi,d\overline W_S \Big).
\]
It is easy to see by L\'evy characterization that $\tilde W_T^N$ is simply 
a  cylindrical Wiener process  on $L^2(I)$ (which we denote by $W_T$) stopped at $\tau_N$,
 and that one has the martingale representation
$
N_T^{\tau_N}(\varphi )=\int_0^{T\wedge \tau_N} (\mathscr Z_S \phi,dW_S)
$.
Using this representation together with \eqref{eq:Nt}, one has
\[
	(\mathscr Z_T,\varphi )- (\mathscr  Z_0,\varphi)
	=
	\int_0^{T\wedge \tau_N} (\mathscr Z_S \phi,dW_S)
	+ \frac12 \int_0^{T\wedge \tau_N} (\mathscr  Z_S,\varphi'')\,dS
\]
$\CQ'$-a.s.. Sending $N\to \infty$ we obtain that $\mathscr Z$ is  the  ``weak  solution" (in PDE sense) defined in \cite[Eq.~3.9]{MR876085}. In order to identify
$\mathscr Z$ as the mild solution, as in \cite[Eq.~3.10]{MR876085}
we can actually show that for all smooth functions $\psi(T,X)$ (which also depends on $T$) such that $\psi'(T,X) |_{X=0}=A\psi(T,0)$,
as well as  $\psi'(T,X) |_{X=1}=-B \psi(T,1)$ if $I=[0,1]$, one has
\[
	(\mathscr Z_T,\psi(T) )- (\mathscr  Z_0,\psi(0))
	=
	\int_0^{T} (\mathscr Z_S \psi(S),dW_S)
	+ \frac12 \int_0^{T} \Big( \mathscr  Z_S,\psi''(S)
		+\frac{\partial\psi}{\partial S}(S) \Big)\,dS \;.
\]
Choose $\psi(S,Y)= \int_I \mathscr P^R_{T-S}(Y,U)\phi(U)dU$
which clearly satisfies the above Robin boundary condition since $\mathscr P^R$ does,
and is also such that $\psi''+\frac{\partial\psi}{\partial S}=0$.
Sending $\phi(\Cdot)\to \delta_X(\Cdot)$ and using the symmetric property $\mathscr P^R_S(X,Y)=\mathscr P^R_S(Y,X)$, we see that
$\psi(S,\Cdot) \to \mathscr P^R_{T-S}(X,\Cdot)$,
$\psi(T,\Cdot) \to \delta_X(\Cdot)$, and $\psi(0,\Cdot) \to \mathscr P^R_{T}(X,\Cdot)$.
Therefore
$\mathscr Z$ is the mild solution defined in Definition~\ref{def:mild} and in particular the martingale solution is unique.
\end{proof}

\begin{proof}[Proof of Theorems~\ref{thm:main-H} and \ref{thm:main-B}]
The rescaled processes are tight according to Proposition~\ref{prop:tight},
and by
Proposition~\ref{prop:unique}
any limiting point solves the martingale problem, with the desired boundary conditions and initial condition.
According to Proposition~\ref{prop:mart-mild}
the law of the martingale solution coincides with that of the mild solution.
Therefore Theorems~\ref{thm:main-H} and \ref{thm:main-B} follow immediately.
\end{proof}

 \ack 


I.C. would like to acknowledge Kirone Mallick for discussions regarding the matrix product ansatz.
We also thank Marielle Simon, Patricia Gon\c{c}alves and Nicolas Perkowski for letting us know about their work in preparation \cite{Goncavlesetal} when we sent them a draft of our paper, and in turn for sending / discussing with us their draft.
We would also like to thank the referee for the helpful comments.
I.C. was partially supported by the NSF through DMS-1208998, DMS-1664650, the Clay Mathematics Institute through a Clay Research Fellowship and the Packard Foundation through a Packard Fellowship for Science and Engineering. H.S. was partially supported by the NSF through DMS-1712684.


\frenchspacing
\bibliographystyle{cpam}
%
\bibliography{./refs}

\end{document}